\documentclass{amsart}
\usepackage{amsthm}
\usepackage{amssymb, latexsym}
\usepackage{enumerate}
\usepackage{mathtools,color}
\usepackage{dsfont}
\usepackage{bbm}

\usepackage{scalerel,stackengine}
\stackMath

\newcommand\card{\operatorname{card}}

\newcommand \p{\partial}

\newcommand\reallywidecheck[1]{%
\savestack{\tmpbox}{\stretchto{%
  \scaleto{%
    \scalerel*[\widthof{\ensuremath{#1}}]{\kern-.6pt\bigwedge\kern-.6pt}%
    {\rule[-\textheight/2]{1ex}{\textheight}}%WIDTH-LIMITED BIG WEDGE
  }{\textheight}%
}{0.5ex}}%
\stackon[1pt]{#1}{\scalebox{-1}{\tmpbox}}%
}

\newtheorem{prop}{Proposition}
\newtheorem{thm}[prop]{Theorem}
\newtheorem{rem}{Remark}
\newtheorem{lem}[prop]{Lemma}
\newtheorem{res}{Result}
\begin{document}

\title[Barely Supercritical Klein-Gordon Equations] {Global existence of solutions of a loglog energy-supercritical Klein-Gordon equation}
\author{Tristan Roy}
\address{American University of Beirut, Department of Mathematics}
\email{tr14@aub.edu.lb}

\begin{abstract}
We prove global existence of the solutions of the loglog energy-supercritical Klein-Gordon equation $ \partial_{tt} u - \triangle u + u = -|u|^{\frac{4}{n-2}} u
\log^{\gamma} \left( \log ( 10 + |u|^{2} ) \right) $, with $n \in \{ 3,4,5 \}$,  $ 0 <  \gamma < \gamma_{n} $, and data $(u_{0},u_{1}) \in H^{k} \times H^{k-1}$ for $k > 1$ (resp. $ \frac{7}{3} > k > 1$) if $n \in \{ 3,4 \}$ (resp. $n=5$). The proof is by contradiction. Assuming that blow-up occurs at a maximal time of existence, we perform an analysis close to this time in order to find a finite bound of a Strichartz-type norm, which eventually leads to a contradiction with the blow-up assumption. %by using in particular arguments in \cite{nakimrn,triroyjensen,triroysmooth}, which
\end{abstract}

\maketitle

\section{Introduction}

We shall study the solutions of the following defocusing nonlinear  Klein-Gordon equation in dimension $n$, $n \in \{ 3,4,5 \}$: % \blc CH \bc

\begin{equation}
\begin{array}{ll}
\partial_{tt} u - \triangle u + u & =  -|u|^{\frac{4}{n-2}} u g(|u|)
\end{array}
\label{Eqn:KgBar}
\end{equation}
Here $ g(|u|):= \log^{\gamma} \left( \log ( 10 + |u|^{2} ) \right)$ and $ \gamma >  0$. Here $\log$ denotes the natural logarithm. The solutions of (\ref{Eqn:KgBar}) satisfy three properties that we use throughout this paper:

\begin{itemize}

\item the \textit{time translation invariance}: if $u$ is a solution of (\ref{Eqn:KgBar}) and $t_{0}$ is a fixed number then
$\tilde{u}$ defined by $\tilde{u}(t,x) := u(t-t_0,x) $ is also a solution of (\ref{Eqn:KgBar}).

\item the \textit{space translation invariance}: if $u$ is a solutions of (\ref{Eqn:KgBar}) and $x_{0} \in \mathbb{R}^{n}$ then
$\tilde{u}$ defined by $\tilde{u}(t,x) := u(t,x-x_{0})$ is also a solution of (\ref{Eqn:KgBar}).

\item the \textit{time reversal invariance}: if $u$ is a solution of (\ref{Eqn:KgBar}) then $\tilde{u}$ defined by
$\tilde{u}(t,x) : = u(-t,x)$ is also a solution of (\ref{Eqn:KgBar}).

\end{itemize}
This equation has many connections with the following semilinear Klein-Gordon equation, $p > 1$

\begin{equation}
\begin{array}{ll}
\partial_{tt} v - \triangle v + v & = -|v|^{p-1} v,
\end{array}
\label{Eqn:KgPowerp}
\end{equation}
which in turn is related to the following semilinear wave equation

\begin{equation}
\begin{array}{ll}
\partial_{tt} v - \triangle v  & = -|v|^{p-1} v \cdot
\end{array}
\label{Eqn:WavePowerp}
\end{equation}
(\ref{Eqn:WavePowerp}) has a natural scaling: if $v$ is a solution of (\ref{Eqn:WavePowerp}) with data $\left( v(0),\partial_{t} v(0) \right) := (v_{0},v_{1})$ and if $\lambda \in \mathbb{R}$ is a parameter then $v_{\lambda}(t,x) := \frac{1}{\lambda^{\frac{2}{p-1}}} v \left( \frac{t}{\lambda}, \frac{x}{\lambda}
\right)$ is also a solution of (\ref{Eqn:WavePowerp}) but with data $ \left( v_{\lambda}(0,\cdot) ,\partial_{t} v_{\lambda}(0,\cdot) \right) :=
\left( \frac{1}{\lambda^{\frac{2}{p-1}}} v_{0} \left( \frac{\cdot}{\lambda} \right), \frac{1}{\lambda^{\frac{2}{p-1}+1}}  v_{1} \left( \frac{\cdot}{\lambda} \right) \right)$. If $s_{p}:= \frac{n}{2}- \frac{2}{p-1}$ then the $\dot{H}^{s_{p}}$ norm of the initial data is invariant under the
scaling: this is why (\ref{Eqn:WavePowerp}) is said to be $\dot{H}^{s_{p}}$- critical. If $p=1 + \frac{4}{n-2}$ then (\ref{Eqn:WavePowerp})
is said to be $\dot{H}^{1}-$ critical (or energy-critical) and (\ref{Eqn:KgPowerp}) is said to be energy-critical. The short-time behavior and the long-time behavior
of solutions of energy-critical Klein-Gordon equations have been extensively studied in the literature: in particular the linear asymptotic behavior
(i.e scattering) was proved in \cite{nakimrn}. If $ p > 1 + \frac{4}{n-2}$ then $s_{p} > 1$ and we are in the energy-supercritical regime. Since for all $\epsilon > 0$ there exist $c_{\epsilon} > 0$ such that $ \left| |u|^{\frac{4}{n-2}} u \right| \lesssim \left| |u|^{\frac{4}{n-2}} u g(|u|) \right| \leq c_{\epsilon} \max{(1, | |u|^{\frac{4}{n-2}+ \epsilon} u | ) }$ then the nonlinearity of (\ref{Eqn:KgBar}) is said to be barely energy-supercritical. Barely energy-supercritical equations have been studied extensively in the literature: see e.g \cite{bulutdods,shihhsiwei,triroyrad, triroyradfoc,triroyjensen,triroysmooth,tao}. We write below a local-wellposedness result:

\begin{prop}
Let $n \in \{ 3,4,5 \}$. If $n \in \{ 3,4 \}$ then let $ 1 < k $ and let $F([0,T_{l}]):=  L_{t}^{\frac{2(n+1)}{n-1}} H^{k-\frac{1}{2}, \frac{2(n+1)}{n-1}} ([0,T_{l}])  $.
If $n =5$ then let $ 1 < k < \frac{7}{3}$ and let
$F([0,T_{l}]):=  L_{t}^{\frac{2(n+1)}{n-1}} H^{k-\frac{1}{2}, \frac{2(n+1)}{n-1}} ([0,T_{l}]) \cap L_{t}^{2} H^{k-1, \frac{2n}{n-3}} ([0,T_{l}]) $. Let $(u_{0},u_{1}) \in H^{k} \times H^{k-1}$ and $ M \in \mathbb{R} $ be such that $\| (u_0,u_{1}) \|_{H^{k} \times H^{k-1} } \leq M$. Then there exists $ \delta := \delta(M) > 0 $ that has the following property: if $ T_{l} > 0 $ is a number such that if

\begin{equation}
\begin{array}{ll}
\left\| \cos{( t \langle D \rangle )} u_{0} +  \frac{\sin \left( t \langle D \rangle \right) }{ \langle D \rangle} u_{1}  \right\|_{L_{t}^{\frac{2(n+1)}{n-2}} L_{x}^{\frac{2(n+1)}{n-2}} ([0,T_{l}])}  & \leq \delta,
\end{array}
\label{Eqn:SmallLinProp}
\end{equation}
then there exists a unique

\begin{equation}
\begin{array}{ll}
u \in \mathcal{C}([0,T_{l}], H^{k} ) \cap \mathcal{C}^{1} ([0,T_{l}],H^{k-1}) \cap F ([0,T_{l}]) \cap
\mathcal{B} \left( L_{t}^{\frac{2(n+1)}{n-2}} L_{x}^{\frac{2(n+1)}{n-2}}([0,T_{l}]) ; 2 \delta  \right)
\end{array}
\label{Eqn:Spaceu}
\end{equation}
such that

\begin{equation}
\begin{array}{l}
u(t) = \cos{(t \langle D \rangle) } u_{0} + \frac{\sin {(t \langle D \rangle) }}{\langle D \rangle} u_{1} -
\int_{0}^{t} \frac{\sin{ \left( (t-t') \langle D \rangle \right) }}{\langle D \rangle} \left( |u(t^{'})|^{\frac{4}{n-2}} u(t^{'}) g(|u(t^{'})|) \right) \, dt^{'}
\end{array}
\label{Eqn:DistribNlkg}
\end{equation}
is satisfied in the sense of distributions. Here $ \mathcal{B} \left( L_{t}^{\frac{2(n+1)}{n-2}} L_{x}^{\frac{2(n+1)}{n-2}}([0,T_{l}]) ; \bar{r}  \right) $
denotes the closed ball centered at the origin with radius $\bar{r}$ in $ L_{t}^{\frac{2(n+1)}{n-2}} L_{x}^{\frac{2(n+1)}{n-2}}([0,T_{l}]) $.
\label{Prop:LocalWell}
\end{prop}
The proof of Proposition \ref{Prop:LocalWell} is given in Appendix $B$.

\begin{rem}
A number $T_{l}$ that satisfies the smallness condition above is called a time of local existence.
\end{rem}

\begin{rem}
The proof of Proposition \ref{Prop:LocalWell} shows that one can choose $\delta$ as a function that decreases as $M$ increases and that
goes to zero as $M \rightarrow \infty$.
\label{Rem:Delta}
\end{rem}

This allows by a standard procedure to define the notion of maximal time interval of existence $I_{max}:= (T_{-},T_{+})$, that is the union of all the open intervals $I$ containing $0$ such that there exists a (unique) solution $ u \in \mathcal{C} ( I, H^{k} ) \cap \mathcal{C}^{1} (I,H^{k-1}) \cap
F (I) $ that satisfies (\ref{Eqn:DistribNlkg}) for all $ t \in I $.

\begin{rem}
In the sequel we denote by $H^{k}-$ solution of (\ref{Eqn:KgBar}) a distribution constructed by this standard procedure that

\begin{itemize}
\item satisfies (\ref{Eqn:DistribNlkg}) for some $(u_{0},u_{1}) \in H^{k} \times H^{k-1}$ and for all $t \in I_{max}$
\item lies in $ \mathcal{C} ( I, H^{k} ) \cap \mathcal{C}^{1} (I,H^{k-1}) \cap F (I) $ for all interval $I \subset I_{max} $
\end{itemize}

\end{rem}

\begin{rem}
Note that if $u$ is an $H^{k}-$ solution of (\ref{Eqn:KgBar})  then $\| u \|_{L_{t}^{\frac{2(n+1)}{n-2}} L_{x}^{\frac{2(n+1)}{n-2}} (I)} < \infty $ for
all $I \subsetneq I_{max}$.
\footnote{Indeed we may assume WLOG that $I:= [a,b]$. Let $r$ be such that $\frac{n-2}{2(n+1)} + \frac{n}{r} = \frac{n}{2}- \frac{1}{2}$. Then
$\| u \|_{L_{t}^{\frac{2(n+1)}{n-2}} L_{x}^{\frac{2(n+1)}{n-2}} (I)}  \lesssim \left\| \langle D \rangle^{k- \frac{1}{2}} u \right\|_{L_{t}^{\frac{2(n+1)}{n-2}} L_{x}^{r} (I)} $.
Interpolation shows that there exists $\theta \in [0,1]$ such that $ \left\| \langle D \rangle^{k- \frac{1}{2}} u \right\|_{L_{t}^{\frac{2(n+1)}{n-2}} L_{x}^{r} (I)} \lesssim
\left\| \langle D \rangle^{k- \frac{1}{2}} u \right\|^{\theta}_{L_{t}^{\frac{2(n+1)}{n-1}} L_{x}^{\frac{2(n+1)}{n-1}} (I)}
\left\| \langle D \rangle^{k- \frac{1}{2}} u \right\|^{1 - \theta}_{L_{t}^{\infty} L_{x}^{\frac{2n}{n-1}} (I)} $. Since
$ \left\| \langle D \rangle^{k- \frac{1}{2}} u \right\|_{L_{t}^{\infty} L_{x}^{\frac{2n}{n-1}} (I)} \lesssim
\left\| \langle D \rangle^{k} u \right\|_{L_{t}^{\infty} L_{x}^{2} (I)}$, we get
$\| u \|_{L_{t}^{\frac{2(n+1)}{n-2}} L_{x}^{\frac{2(n+1)}{n-2}} (I)} < \infty $.
}
\\
\end{rem}

\begin{rem}
In the sequel we say that $u$ is an $H^{k}-$ solution of (\ref{Eqn:KgBar}) on an interval $I$ if $u$ is an $H^{k}-$ solution of (\ref{Eqn:KgBar}) and
$I \subset I_{max}$.
\end{rem}
Next we investigate the asymptotic behavior of $H^{k}$-solutions of (\ref{Eqn:KgBar}) for $n \in \{ 3,4,5 \}$. We first prove in Section \ref{Sec:CritBlowUp} the following proposition:

\begin{prop}
Let $u$ be an $H^{k}-$ solution of (\ref{Eqn:KgBar}). If $|I_{max}|< \infty $ then

\begin{equation}
\begin{array}{ll}
\| u \|_{L_{t}^{\frac{2(n+1)}{n-2}} L_{x}^{\frac{2(n+1)}{n-2}} (I_{max})} & = \infty
\end{array}
\end{equation}
\label{Prop:GlobWellPosedCrit}
\end{prop}
We then provide the reader with a criterion for proving (by contradiction) global existence of $H^{k}-$ solutions of (\ref{Eqn:KgBar})
(i.e all the $H^{k}-$ solutions of (\ref{Eqn:KgBar}) exist for all time: in other words, $|I_{max}| = \infty$ for all data $(u_{0},u_{1}) \in H^{k} \times H^{k-1}$): see remark below.

\begin{rem} (Criterion for global existence)

Let $u$ be an $H^{k}-$ solution of (\ref{Eqn:KgBar}). Let $\epsilon_{0}$ be a constant such that $ 0 <  \epsilon_{0} \lesssim 1 $. Assume that we can prove that there
there exists a function $f$ that has finite values such that for all $(u_{0},u_{1}) \in H^{k} \times H^{k-1}$ the $H^{k}-$ solution $u$ with data $(u_{0},u_{1})$ satisfies the estimate below:

\begin{equation}
\begin{array}{ll}
\| u \|_{L_{t}^{\frac{2(n+1)}{n-2}} L_{x}^{\frac{2(n+1)}{n-2}} \left( [0, \epsilon_{0}) \right)} & \leq  f  \left(  \| (u_{0},u_{1}) \|_{H^{k} \times H^{k-1}} \right)
\end{array}
\label{Eqn:StrichZeroEpsZero}
\end{equation}
Then global existence of $H^{k}-$ solutions of (\ref{Eqn:KgBar}) holds. Indeed, if not we see from Proposition \ref{Prop:GlobWellPosedCrit} and time reversal invariance that there exist data $(u_{0},u_{1}) \in H^{k} \times H^{k-1}$ and a constant $ 0 < \bar{\epsilon} < \epsilon_{0}$ such that

\begin{equation}
\begin{array}{ll}
\| u \|_{L_{t}^{\frac{2(n+1)}{n-2}} L_{x}^{\frac{2(n+1)}{n-2}} \left( [T_{+} - \bar{\epsilon}, T_{+} ) \right)} & =   \infty
\end{array}
\nonumber
\end{equation}
for $u$ an $H^{k}-$ solution of (\ref{Eqn:KgBar}) with data $(u_{0},u_{1})$. Moreover  $\| u \|_{L_{t}^{\frac{2(n+2)}{n-2}} L_{x}^{\frac{2(n+2)}{n-2}} (K) } < \infty $ for
$K$ interval such that $K \subsetneqq  [T_{+} - \bar{\epsilon}, T_{+} )$. By time translation ( with $\bar{t} := T_{+} - \bar{\epsilon} $ ) we see that there exists an
$H^{k}-$ solution (that we still denote by $u$) such that $ \| u \|_{L_{t}^{\frac{2(n+1)}{n-2}} L_{x}^{\frac{2(n+1)}{n-2} } \left( [0,\bar{\epsilon}) \right)} = \infty $ and
$ \| u \|_{L_{t}^{\frac{2(n+1)}{n-2}} L_{x}^{\frac{2(n+1)}{n-2} } (K)} < \infty $ for $K$ interval such that $ K \subsetneqq  [0,\bar{\epsilon})$. This contradicts
(\ref{Eqn:StrichZeroEpsZero}).

\label{Rem:GlobalCrit}
\end{rem}

The main result of this paper is a global existence result for (\ref{Eqn:KgBar}), namely

\begin{thm}
Let $n \in \{ 3,4,5 \}$. \\
Let $I_n$ defined as follows: if $n \in \{3,4\}$ then $I_n := (1,\infty)$ and if
$n=5$ then  $I_n := \left( 1, \frac{7}{3} \right)$. Let $\gamma_{n}$ be defined as follows:

\begin{equation}
\gamma_{n} := \left\{
\begin{array}{l}
\frac{1}{6}, \; n =3 \\
\frac{4}{49}, \; n = 4 \\
\frac{1}{22}, \; n = 5
\end{array}
\right.
\nonumber
\end{equation}
Let $u$ be an $H^{k}-$ solution of (\ref{Eqn:KgBar}) with $ 0 < \gamma < \gamma_{n}$ and with data $(u_{0},u_{1}) \in H^{k} \times H^{k-1}$, $k \in I_{n}$.
Then $u$ exists for all time.

\label{thm:main}
\end{thm}

We now explain the main interest of this paper. Our goal is to prove a global existence result for solutions of loglog energy-supercritical Klein-Gordon equations of the form
(\ref{Eqn:KgBar}). In our previous work (see \cite{triroysmooth}), we have managed to prove global existence and scattering of solutions of $3d-$ loglog energy-supercritical wave equations of the form $\partial_{tt} u - \triangle u = -|u|^{4} u \log^{\gamma} \left( \log ( 10 + |u|^{2}) \right)$ for a range of positive $\gamma$ s and for data
$(u_{0},u_{1}) \in \tilde{H}^{2} \cap \tilde H^{1}$ \footnote{Recall that $\tilde{H}^{m} := \dot{H}^{m} \cap \dot{H}^{m-1}$}. The scattering follows from the finiteness of a Strichartz-type norm of the solution on $\mathbb{R}$ (namely $\| u \|_{L_{t}^{4} L_{x}^{12}(\mathbb{R})}$) and that of the norm
$ \left\| (u,\partial_{t} u) \right\|_{L_{t}^{\infty}\tilde{H}^{2}(\mathbb{R}) \times L_{t}^{\infty} H^{1} (\mathbb{R})}$: see Appendix $D$ \footnote{see also introduction in \cite{duyroy}}. The finiteness of these norms is proved by using strong Morawetz-type estimates inside cones. Unfortunately these estimates are not available for $H^{k}-$ solutions of (\ref{Eqn:KgBar}), because (\ref{Eqn:KgBar}) contains the mass term $u$. In \cite{nakimrn}, a finite bound of a
Strichartz-type norm of solutions of energy-critical Klein-Gordon equations on a time interval of size roughly equal to one was found. The proof of this bound relies upon
 methods of concentration (in the spirit of \cite{bourg}), weighted Morawetz-type estimates, and decay estimates inside cones. It should be possible to prove a similar estimate for $H^{k}-$solutions of (\ref{Eqn:KgBar}). More precisely we prove in Section \ref{Sec:ProofProp} the following proposition \footnote{The definition of $X_{k}(K,u)$ is given in Section \ref{Sec:Prelim}.}

\begin{prop}
Let $u$ be an $H^{k}-$ solution of (\ref{Eqn:KgBar}) on an interval $K := [0,a] \subset [0, \epsilon_{0})$. There exists a constant
$C_{1} \gg 1 $ such that if $ X_{k}(K,u) \leq M $ for some $M \gg 1$ then

\begin{equation}
\begin{array}{ll}
\| u \|_{L_{t}^{\frac{2(n+1)}{n-2}} L_{x}^{\frac{2(n+1)}{n-2}}(K)} & \leq  C_{1}^{C_{1} g^{b_{n}+}(M)}
\end{array}
\label{Eqn:BoundStrichLong}
\end{equation}
with $b_{n}$ such that

\begin{equation}
b_{n} = \left\{
\begin{array}{l}
6 , n = 3 \\
\frac{49}{4}, n = 4 \\
22, n =5
\end{array}
\right.
\end{equation}
\label{Prop:BoundLong}
\end{prop}
The proposition above shows that we have a finite bound of a Strichartz-type norm (namely $  \| u \|_{L_{t}^{\frac{2(n+1)}{n-2}} L_{x}^{\frac{2(n+1)}{n-2}}(K)} $ ) of an $H^{k}-$ solution $u$ on the interval $K$ assuming that an \textit{a priori} bound of some norms at $H^{k}-$ regularity % \footnote{By norm at $H^{k}-$ regularity we mean a norm that requires
%at least the data to be in $H^{k} \times H^{k-1}$ for it to be finite when it is applied to a solution of the linear Klein-Gordon equation by using (\ref{Eqn:StrichEst})}
holds on this interval. Observe that this estimate depends slowly on the \textit{a priori} bound: this observation is crucial to control \textit{a posteriori} these norms on intervals of size roughly equal to one (see Section \ref{Sec:Thmmain}). The proof of Proposition \ref{Prop:BoundLong} relies also upon some (local-in-time) nonlinear estimates: in order for these estimates to depend slowly on the \textit{a priori} bound, we prove in Section \ref{Sec:JensenIneq} some Jensen-type inequalities (in the spirit of \cite{triroyjensen}) and then fractional Leibnitz-type estimates that have this slow dependence property.  These fractional Leibnitz-type rules are also used  in Section \ref{Sec:CritBlowUp} to prove Proposition \ref{Prop:GlobWellPosedCrit} and consequently the criterion of global existence of $H^{k}-$ solutions of (\ref{Eqn:KgBar}) (see Remark \ref{Rem:GlobalCrit}). In Section \ref{Sec:Thmmain} we prove the main result of this paper, i.e Theorem \ref{thm:main}. The proof combines the estimate (\ref{Eqn:BoundStrichLong}) on an interval of size roughly equal to one with an iteration argument on small subintervals to find an \textit{a posteriori} bound of the Strichartz-type norm and the norms at $H^{k}-$ regularity on this interval. This proves global existence by Remark \ref{Rem:GlobalCrit}.

\section{Preliminaries}
\label{Sec:Prelim}

\subsection{General notation}

We recall some general notation.\\
\\
If $a \in \mathbb{R}$ then $\langle a \rangle := \left( 1+ a^{2} \right)^{\frac{1}{2}}$.  We write
$a \lesssim b$ (resp. $a \ll b$ ) if there exists a positive constant (resp. positive and small constant compare with $1$) $C$ (resp. $c$) \footnote{In particular $C$ and $c$
do not depend on $a$ and $b$} such that $ a \leq C b$ (resp. $ a \leq c b $). We write $a \gtrsim b$ (resp. $a \gg b$) if $b \lesssim a$  (resp. $b \ll a$). We write
$a \approx b $ if $a \lesssim b$ and $b \lesssim a$. It may be that the constants $C$ or $c$ depend on some parameters
$\alpha_{1}$, ..., $\alpha_{m}$: unless otherwise specified, we do not mention them, for sake of simplicity. We define $b+$ to be a number $ b + \epsilon $ for
some $ 0 < \epsilon \ll 1 $ \footnote{In view of what is written above, if $(a,b,d) \in \mathbb{R}^{3}$, then $a \lesssim b^{d+}$ means that there exists a constant
$C > 0$ that may depend on $\epsilon$ and such that $ a \leq C b^{d + \epsilon} $.}. \\
Unless otherwise specified, we let in the sequel $f$ (resp. $u$) be a function depending on space (resp. space
and time). Unless otherwise specified, for sake of simplicity, we do not mention the spaces to which $f$ and $u$
belong in the estimates: this exercise is left to the reader.\\

\subsection{Other notation}

Let $r > 1$ and let $ 0 < m < \frac{n}{r}$. We denote by $m_{r}^{*}$ the number that satisfies

\begin{equation}
\begin{array}{l}
\frac{1}{m_{r}^{*}} = \frac{1}{r} - \frac{m}{n}
\end{array}
\nonumber
\end{equation}

Let $j \in \mathbb{R}$. We define

\begin{equation}
X_{j}(J,u) := \| u \|_{ L_{t}^{\frac{2(n+1)}{n-1}} H^{j - \frac{1}{2},\frac{2(n+1)}{n-1}} (J)} + \| u \|_{L_{t}^{\infty} H^{j} (J)}
+ Y_{j}(J,u), \; \text{with}
\nonumber
\end{equation}
$Y_{j}(J,u) := 0 $ if $n \in \{ 3,4 \}$ and $Y(J,u) := \| u \|_{L_{t}^{2} H^{j- 1, \frac{2n}{n-3}}(J) }$ if $n=5$. \\
\\

Let $u$ be an $H^{k}-$ solution of (\ref{Eqn:KgBar}), with $k$ defined in Proposition \ref{Prop:LocalWell}. We define for $t \in I_{max}$

\begin{equation}
\begin{array}{ll}
E(u(t))  & :=  \frac{1}{2} \int_{\mathbb{R}^{n}} |\partial_{t} u(t,x)|^{2} \; dx  +  \frac{1}{2} \int_{\mathbb{R}^{n}} |\nabla
u(t,x)|^{2} \; dx
+ \frac{1}{2} \int_{\mathbb{R}^{n}} |u(t,x)|^{2} \; dx
+  \int_{\mathbb{R}^{n}} F \left( u(t,x),\overline{u(t,x)} \right) \; dx,
\end{array}
\label{Eqn:EnergyBarely}
\end{equation}
with

\begin{equation}
\begin{array}{ll}
F(z,\bar{z}) & := \int_{0}^{|z|} s^{1_2^{*} - 1} g(s) \, ds \cdot
\end{array}
\label{Eqn:ExpF}
\end{equation}
Observe that $E(u(t))$ is finite. Indeed, integrating by parts once $F(z,\bar{z})$ we get

\begin{equation}
F(z,\bar{z}) = \frac{|z|^{1_{2}^{*}} g(|z|) }{1_{2}^{*}} - \frac{1}{1_{2}^{*}} \int_{0}^{|z|} s^{1_{2}^{*}} g^{'} (s) \; ds \cdot
\label{Eqn:IntPartsFz}
\end{equation}
Hence elementary estimates show that if $\gamma$ is small enough \footnote{in particular if $ 0 < \gamma < \gamma_{n}$, with
$\gamma_{n}$ defined in the statement of Theorem \ref{thm:main}}

\begin{equation}
\begin{array}{ll}
F(z,\bar{z}) & \approx  |z|^{1_{2}^{*}} g(|z|),
\end{array}
\label{Eqn:EquivF}
\end{equation}
which implies that

\begin{equation}
\begin{array}{ll}
\left| \int_{\mathbb{R}^{n}} F(f,\bar{f})(x) \, dx \right| & \lesssim  \| f \|^{1_{2}^{*}}_{L^{1_{2}^{*}}} +
\| f \|^{k_{2}^{*}}_{L^{k_{2}^{*}}}  \\
& \lesssim \langle \| f \|_{H^{k}} \rangle^{k_{2}^{*}},
\end{array}
\end{equation}
the last estimate resulting from the Sobolev embeddings  $L^{1_{2}^{*}} \hookrightarrow H^{k}$ and
$L^{k_{2}^{*}} \hookrightarrow H^{k}$, combined with the estimate $g(|f|) \lesssim  1 + |f|^{k_{2}^{*}-1_{2}^{*}}$. \\
\\
 A simple computation shows that $E(u(t))$ is conserved: in other words, $E(u(t))=E(u(0))$ \footnote{More precisely, the computation holds for smooth solutions (i.e solutions in $H^{p}$ with exponents $p$ large enough). Then $E(u(t))=E(u(0))$  holds for an $H^{k}-$ solution with $k \in I_{n}$ by a standard approximation with smooth solutions}. Therefore, in the sequel, we write $E$ instead of $E(u(t))$ and $E$ denotes the energy of $u$.\\
 \\
 Let $J$ be an interval. Let $(t_0,t) \in J^{2}$. If $u$ is a solution of $ \partial_{tt} u - \triangle u + u = G$ on $J$ then we have

\begin{equation}
\begin{array}{l}
u(t) = u_{l,t_0}(t)+ u_{nl,t_0}(t)
\end{array}
\nonumber
\end{equation}
with $u_{l,t_0}$ denoting the linear part starting from $t_0$, i.e

\begin{equation}
\begin{array}{l}
u_{l,t_0}(t) := \cos{ \left( (t-t_0) \langle D \rangle \right)} u(t_0) + \frac{ \sin{ \left( (t-t_0) \langle D \rangle \right) } }  { \langle D  \rangle} \p_{t} u(t_{0}),
\end{array}
\nonumber
\end{equation}
and $u_{nl,t_0}$ denoting the nonlinear part starting from $t_0$, i.e

\begin{equation}
\begin{array}{l}
u_{nl,t_0}(t) :=  - \int_{t_0}^{t} \frac{ \sin{ \left( (t-s) \langle D \rangle \right) }} {\langle D \rangle} G(s) \; ds \cdot
\end{array}
\nonumber
\end{equation}

\subsection{Jensen inequality and Strichartz-type estimates}

We recall some standard inequalities.\\
\\
Throughout this paper we will constantly use (a generalized form of) the Jensen inequality (see e.g \cite{yeh}). The statement of this inequality is made in
\cite{yeh} for convex functions. The statement of the inequality below follows immediately from that in \cite{yeh} , taking into account that if $f$ is concave then
$-f$ is convex.

\begin{prop}(Jensen inequality, see e.g \cite{yeh})
Let $(X, \mathcal{B}, \mu)$ be a measure space such that $0 < \mu(X) < \infty$. Let $I$ be an open interval and let
$g$ be a $\mu-$ integrable function on a set $D \in \mathcal{B}$ such that $g(D) \subset I$. If $f$ is a concave function on
$I$ then the following holds

\begin{equation}
\begin{array}{ll}
\frac{1}{\mu(D)} \int_{D} f \circ  g \; d \mu &  \leq f \left( \frac{ \int_{D} g \; d \mu }{\mu(D)} \right)
\end{array}
\nonumber
\end{equation}

\end{prop}

%If $ b+ $ appears in a mathematical expression such as
%$a \leq C b+$, then we ignore the dependance of $C$ on $\epsilon$ in order to make our presentation simple. \\ \\

% TO CHECK
%Let $r > 1$ and let $m$ be a positive number such that $m < \frac{n}{r}$. We denote by $m_{r}^{*}$ the number that satisfies
%$\frac{1}{m_{r}^{*}} = \frac{1}{r} - \frac{m}{n}$. Let $\bar{k}$ be a constant such that
%$1 < \bar{k} < \min \left( \frac{n+2}{4}, k \right)$ and $\bar{k} -1 \approx  k -1$ if $k-1 \ll 1$. Then the following Sobolev
%embeddings hold:

%\begin{equation}
%\begin{array}{l}
%\| f \|_{L^{m_{r}^{*}}} \lesssim  \| D^{m} f \|_{L^{r}}, \; \| f \|_{L^{\bar{k}_{2}^{*}}} \lesssim \| f \|_{\tilde{H}^{k}} \\
%k > \frac{n}{2}: \; \| f \|_{L^{\infty}} \lesssim \| f \|_{\tilde{H}^{k}}, \; \text{and} \\
%\| f \|_{L^{\frac{2(n+2)}{n-2}}} \lesssim  \| D f \|_{L^{\frac{2n(n+2)}{n^{2}+4}}} \cdot
%\end{array}
%\label{Eqn:SobolevIneq2}
%\end{equation}
% TO CHECK

We will combine the Jensen inequality with well-known Strichartz-type estimates. We recall now these estimates. Let $J$ be an interval. Let $t_0 \in J$. If $u$ is a solution of $ \partial_{tt} u - \triangle u + u = G$ on $J$ then the following estimates hold

\begin{equation}
\begin{array}{l}
\| u \|_{L_{t}^{\infty} H^{m} (J) } + \| u \|_{L_{t}^{q} L_{x}^{r} (J)}
\lesssim \left\| ( u(t_{0}), \p_{t} u(t_{0}) ) \right\|_{H^{m} \times H^{m-1}} +
\| G \|_{L_{t}^{\tilde{q}} L_{x}^{\tilde{r}} (J)} \cdot
\end{array}
\label{Eqn:StrichEst}
\end{equation}
Here $ m \in [0,1]$ and $(q,r, \tilde{q},\tilde{r})$ satisfying the following admissibility properties:
$(q,r) \in \mathcal{W} := \left\{ (x,y): x \geq 2, \; \frac{1}{x} + \frac{n-1}{2 y} \leq \frac{n-1}{4}, \; \left( x,y, \frac{n-1}{2} \right) \neq (2,\infty,1) \right\}$, \\
$(\tilde{q},\tilde{r}) \in \mathcal{W}^{'} := \left\{ (x',y'): \; \exists (x,y) \in \mathcal{W}: \;  \text{s.t}
\left( \frac{1}{x} + \frac{1}{x'}, \frac{1}{y} + \frac{1}{y'} \right) = (1,1)  \right\} $,  and
$\frac{1}{q} + \frac{n}{r} = \frac{n}{2} - m = \frac{1}{\tilde{q}} + \frac{n}{\tilde{r}} - 2 $. \\
We also have

\begin{equation}
\begin{array}{ll}
\| u \|_{L_{t}^{\infty} H^{m}(J)}  +  \| u \|_{L_{t}^{q} L_{x}^{r} (J)} & \lesssim
\left\|  \left( u(t_{0}), \partial_{t} u(t_{0}) \right) \right\|_{H^{m} \times H^{m-1}} +
\left\| \langle D  \rangle^{m-1} G  \right\|_{L_{t}^{1} L_{x}^{2}(J)}
\end{array}
\label{Eqn:StrichEst2}
\end{equation}

\subsection{Paley-Littlewood projectors}

Throughout this paper we use the Paley-Littlewood technology. Let $\phi$ be a bump function, i.e a function $\phi$ that satisfies the following
properties: it is smooth, $\phi(\xi) =  1$ if $|\xi| \leq \frac{1}{2}$ and $\phi(\xi) = 0 $ if $|\xi| \geq 1$. Let $\psi(\xi) := \phi(\xi)
- \phi \left( \frac{\xi}{2} \right)$. If $N \in 2^{\mathbb{N}}$ then the Paley-Littlewood projectors $P_{N}$, $P_{<N}$, and $P_{\geq N}$ are defined in the Fourier
domain by

\begin{equation}
\begin{array}{l}
\widehat{P_{<N} f}(\xi) := \sum \limits_{M \in 2^{\mathbb{Z}}: M < N} \widehat{P_{M} f}(\xi), \\
\widehat{P_{N} f} (\xi) := \psi \left( \frac{\xi}{N} \right) \hat{f}(\xi), \; \text{and} \\
\widehat{P_{\geq N} f}(\xi) := \hat{f}(\xi) - \widehat{P_{<N} f}(\xi) \cdot
\end{array}
\nonumber
\end{equation}
The Paley-Littlewood projector $P_{0}$ is defined in the Fourier domain by

\begin{equation}
\begin{array}{l}
\widehat{P_{0} f}(\xi) := \phi(\xi) \hat{f}(\xi) \cdot
\end{array}
\nonumber
\end{equation}

\section{Jensen-type inequalities and Leibnitz rules}
\label{Sec:JensenIneq}

In this section we prove Jensen-type inequalities. We then derive Leibnitz-type rules from the Jensen-type inequalities. If
$ f: \mathbb{R}^{+} \rightarrow \mathbb{R}$ is a function then we denote by $\check{f}$ the function such that
$\check{f} (x) = f(x^{2})$.

\subsection{Jensen-type inequalities}

In this subsection we prove the following Jensen-type inequalities:

\begin{prop}

Let $I$ be an interval. Let $\beta > 0$. Let $\check{k} > 1$. Let $F : \mathbb{R}^{+} \rightarrow \mathbb{R}^{+}$ be a function that has the following properties:

\begin{equation}
\begin{array}{l}
(a): \; \forall \mu > 0, \; \exists B > 0 \; s.t \; F^{\mu} \; \text{is concave on} \; (B, \infty) \;  \\
(b): \; \forall \mu > 0, \; \forall \epsilon > 0, \; \exists B > 0 \; \text{s.t} \left[ x > B \Longrightarrow F^{\mu}( x^{\epsilon} ) \geq \frac{1}{10} F ^{\mu}(x)  \right], \; \text{and} \\
(c): \; \forall \mu > 0, \; \forall \nu > 0 \; F^{\mu}(x^{\nu}) \lesssim F^{\mu}(x) \cdot
\end{array}
\nonumber
\end{equation}
Assume that there exist $(P,Q) \in \mathbb{R}^{+} \times \mathbb{R}^{+}$ such that $\| u \|_{L_{t}^{\frac{2(n+1)}{n-2}} L_{x}^{\frac{2(n+1)}{n-2}}(I)} \leq P$ and $ X_{\check{k}}(I,u) \leq Q $. Then

\begin{equation}
\begin{array}{ll}
\| F^{\beta}(|u|^{2}) u \|_{L_{t}^{\frac{2(n+1)}{n-2}} L_{x}^{\frac{2(n+1)}{n-2}}(I)} \lesssim P \check{F}^{\beta}(Q)
\end{array}
\label{Eqn:ConcavEst2}
\end{equation}

\label{Prop:JensenTypeIneq}
\end{prop}

\begin{proof}

We use an argument in \cite{triroyjensen}. Let $ \check{k}-1 \gg \epsilon > 0 $ be a fixed and small enough constant for all the estimates below to be true. There exists
$ A \approx 1 $  such that $ \left[ F^{ \frac{2 (n+1) \beta}{n-2}} \, \text{is concave on} \, (A, \infty) \right]$ and
$ \left[ x \in \mathbb{C}, |x| > A \Longrightarrow   F^{\frac{2 (n+1) \beta}{n-2}}(|x|^{2 \epsilon}) \geq  \frac{1}{10} F^{ \frac{2 (n+1) \beta}{n-2}} (|x|^{2}) \right] $. \\
\\
We see from the triangle inequality that it suffices to estimate

\begin{equation}
\begin{array}{ll}
W_{1} & := \int_{I} \int_{|u(t,x)| \leq A} F^{\frac{2(n+1) \beta}{n-2}} (|u(t,x)|^{2}) |u(t,x)|^{\frac{2(n+1)}{n-2}} \; dx \; dt, \; \text{and} \\
W_{2} & := \int_{I} \int_{|u(t,x)| > A} F^{\frac{2(n+1) \beta}{n-1}} (|u(t,x)|^{2}) |u(t,x)|^{\frac{2(n+1)}{n-2}} \; dx \; dt \cdot
\end{array}
\nonumber
\end{equation}
Elementary considerations show that
$ W_{1} \lesssim \| u \|^{\frac{2(n+1)}{n-2}}_{L_{t}^{\frac{2(n+1)}{n-2}} L_{x}^{\frac{2(n+1)}{n-2}}(I)} \lesssim P^{\frac{2(n+1)}{n-2}} $. In order to estimate $W_{2}$ we use the Jensen inequality twice. More precisely,

\begin{equation}
\begin{array}{ll}
W_{2} & \lesssim \int_{I} \int_{\mathbb{R}^{n}}  F^{\frac{2(n+1) \beta}{n-1}} (|u(t,x)|^{2 \epsilon} ) |u(t,x)|^{\frac{2(n+1)}{n-2}} \; dx \; dt  \\
& \lesssim X_{1} + X_{2},
\end{array}
\nonumber
\end{equation}
with

\begin{equation}
\begin{array}{ll}
X_{1} & := \int_{I} \int_{|u(t,x)|^{2 \epsilon} \leq A}  F^{\frac{2(n+1) \beta}{n-1}} (|u(t,x)|^{2 \epsilon} ) |u(t,x)|^{\frac{2(n+1)}{n-2}} \; dx \; dt, \; \text{and} \\
X_{2} & := \int_{I} \int_{|u(t,x)|^{2 \epsilon} > A}  F^{\frac{2(n+1) \beta}{n-1}} (|u(t,x)|^{2 \epsilon} ) |u(t,x)|^{\frac{2(n+1)}{n-2}} \; dx \; dt \cdot
\end{array}
\nonumber
\end{equation}
Clearly $ X_{1} \lesssim  \| u \|^{\frac{2(n+1)}{n-2}}_{L_{t}^{\frac{2(n+1)}{n-2}} L_{x}^{\frac{2(n+1)}{n-2}}(I)} \lesssim P^{\frac{2(n+1)}{n-2}} $. We have
$X_{2} \lesssim X_{2}^{'}$ with

\begin{equation}
\begin{array}{ll}
X_{2}^{'} & : = \int_{I} \int_{\mathbb{R}^{n}}  F^{\frac{2(n+1) \beta}{n-2}} \left( |u(t,x)|^{2 \epsilon} \mathbbm{1}_{|u(t,x)|^{ 2 \epsilon} > A } \right) |u(t,x)|^{\frac{2(n+1)}{n-2}} \; dx \; dt \cdot
\end{array}
\label{Eqn:X2Bd}
\end{equation}
Then write $I = I_{1} \cup I_{2}$ with $ I_{1} := \left\{ t \in I: \; \| u(t,x) \|_{L_{x}^{\frac{2(n+1)}{n+2}}} = 0 \right\} $ and
$ I_{2} := \left\{ t \in I: \; \| u(t,x) \|_{L_{x}^{\frac{2(n+1)}{n+2}}} \neq 0 \right\} $. If $ t \in I_{1}$ then $u(t) = 0$ and the portion of $X^{'}_{2}$
restricted to $I_{1}$ is equal to zero. In order to estimate the portion of $X^{'}_{2}$ restricted to $I_{2}$, we apply the Jensen inequality with respect to the measure
$ d \mu := |u(t,x)|^{\frac{2(n+1)}{n-2}} dx $. We get $ X_{2} \lesssim Y $
with $ Y := \int_{I_{2}}  F^{\frac{2(n+1) \beta}{n-2}}
\left(
\frac{ \| u(t,.) \|^{\frac{2(n+1)}{n-2} + 2 \epsilon}_{L_{x}^{\frac{2(n+1)}{n-2} + 2 \epsilon}} }
{ \| u(t,.) \|^{\frac{2(n+1)}{n-2}}_{L_{x}^{\frac{2(n+1)}{n-2}}}}
\right) \| u (t,.) \|^{\frac{2(n+1)}{n-2}}_{L_{x}^{\frac{2(n+1)}{n-2}}} \; dt $.

Let $ w(t) := \frac{ \| u(t,.) \|^{\frac{2(n+1)}{n-2} + 2 \epsilon}_{L_{x}^{\frac{2(n+1)}{n-2} + 2 \epsilon}} }
{ \| u(t,.) \|^{\frac{2(n+1)}{n-2}}_{L_{x}^{\frac{2(n+1)}{n-2}}}} $.
We write $ Y = Y_{1} + Y_{2} $ with

\begin{equation}
\begin{array}{ll}
Y_{1} & := \int_{I_{2}} \mathbbm{1}_{|w(t)| \leq A} F^{\frac{2(n+1) \beta}{n-2}} \left( w(t) \right) \| u(t,.) \|^{\frac{2(n+1)}{n-2}}_{L_{x}^{\frac{2(n+1)}{n-2}}} \; dt \\
Y_{2} & := \int_{I_{2}} F^{\frac{2(n+1) \beta}{n-2}} \left( \mathbbm{1}_{ |w(t)| > A} w(t) \right)  \| u(t,.) \|^{\frac{2(n+1)}{n-2}}_{L_{x}^{\frac{2(n+1)}{n-2}}} \; dt
\end{array}
\nonumber
\end{equation}
Clearly $ Y_{1} \lesssim  \| u \|^{\frac{2(n+1)}{n-2}}_{L_{t}^{\frac{2(n+1)}{n-2}} L_{x}^{\frac{2(n+1)}{n-2}}(I)} \lesssim P^{\frac{2(n+1)}{n-2}}$.
We then estimate $Y_{2}$. We may assume WLOG that $\| u \|_{L_{t}^{\frac{2(n+1)}{n-2}} L_{x}^{\frac{2(n+1)}{n-2}} (I_{2})} > 0 $. We apply the Jensen inequality with respect to the measure $ d \mu := \mathbbm{1}_{I_{2}}(t) \| u(t,.) \|^{\frac{2(n+1)}{n-2}}_{L_{x}^{\frac{2(n+1)}{n-2}}} dt $ to get

\begin{equation}
\begin{array}{ll}
Y_{2} & \lesssim \| u \|^{\frac{2(n+1)}{n-2}}_{L_{t}^{\frac{2(n+1)}{n-2}} L_{x}^{\frac{2(n+1)}{n-2}}(I)}
F^{\frac{2(n+1) \beta}{n-2}}  \left( \| u \|^{\frac{2(n+1)}{n-2} + 2 \epsilon}_{ L_{t}^{\frac{2(n+1)}{n-2} + 2 \epsilon} L_{x}^{\frac{2(n+1)}{n-2} + 2 \epsilon}(I)} \right)
\end{array}
\nonumber
\end{equation}
Hence we see that

\begin{equation}
\begin{array}{ll}
\| F^{\beta}(|u|^{2}) u \|_{ L_{t}^{\frac{2(n+1)}{n-2}} L_{x}^{\frac{2(n+1)}{n-2}} (I) } & \lesssim \| u \|_{L_{t}^{\frac{2(n+1)}{n-2}} L_{x}^{\frac{2(n+1)}{n-2}}(I)}
F^{\beta}  \left( \| u \|_{L_{t}^{\frac{2(n+1)}{n-2} + 2 \epsilon} L_{x}^{\frac{2(n+1)}{n-2} + 2 \epsilon}(I) } \right)
\end{array}
\nonumber
\end{equation}
Let $r$ be such that  $\frac{1}{\frac{2(n+1)}{n-2} + 2 \epsilon} + \frac{n}{r} = \frac{n}{2} - \frac{1}{2}$.
The embeddings $H^{\check{k} - \frac{1}{2},r}  \hookrightarrow L^{\frac{2(n+1)}{n-2} + 2 \epsilon} $ and
$H^{\check{k}} \hookrightarrow  H^{\check{k} - \frac{1}{2}, \frac{2n}{n-1}} $,  and interpolation show that
there exists $ 0 < \theta <  1 $ such that

\begin{equation}
\begin{array}{ll}
\| u \|_{L_{t}^{\frac{2(n+1)}{n-2} + 2 \epsilon} L_{x}^{\frac{2(n+1)}{n-2} + 2 \epsilon}(I) } & \lesssim
\| u \|_{L_{t}^{\frac{2(n+1)}{n-2} + 2 \epsilon} H^{\check{k} - \frac{1}{2},r} (I)} \\
& \lesssim \| u \|^{1 - \theta}_{ L_{t}^{\frac{2(n+1)}{n-1}} H^{\check{k}- \frac{1}{2}, \frac{2(n+1)}{n-1}} (I)}
\| u \|^{\theta}_{ L_{t}^{\infty} H^{\check{k} - \frac{1}{2},\frac{2n}{n-1}} (I)}  \\
& \lesssim \| u \|^{1 - \theta}_{ L_{t}^{\frac{2(n+1)}{n-1}} H^{\check{k} - \frac{1}{2},\frac{2(n+1)}{n-1}} (I)}
\| u \|^{\theta}_{L_{t}^{\infty} H^{\check{k}} (I)} \\
& \lesssim Q \cdot
\end{array}
\nonumber
\end{equation}
Hence (\ref{Eqn:ConcavEst2}) holds.

\end{proof}

\subsection{Consequences}

We prove from the Jensen-type inequalities the following fractional Leibnitz rule:

\begin{prop}

Let $k' \in \{1,2,...\}$, and $\beta$, $\alpha$, $\check{k}$ such that $\beta >  k' - 1 $, $ 0 \leq \alpha < 1 $ and $\check{k} > 1$. Let $(\bar{q},\bar{r}, q, r) $ be such that $ (\bar{q},q) \in (1, \infty]^{2} $, $(\bar{r},r) \in (1, \infty]^{2} $, and $ \left( \frac{1}{\bar{q}}, \frac{1}{\bar{r}} \right) = \left( \frac{1}{q}, \frac{1}{r} \right) + \frac{(n-2) \beta}{2(n+1)} (1, 1)$. Let $I$ be an interval. Let $ F: \mathbb{R}^{+} \rightarrow \mathbb{R}^{+} $ be a $ \mathcal{C}^{k^{'}}  \left( \mathbb{R}^{+} \right)- $ function that is nondecreasing and let $ G: \mathbb{R}^{2} \rightarrow \mathbb{R}$ be a $ \mathcal{C}^{k'}(\mathbb{R}^{n})- $ function that satisfy the following properties:

\begin{equation}
\begin{array}{l}
(a): \; \forall \mu > 0, \exists B > 0 \; \text{s.t} \; F^{\mu} \; \text{is concave on} \; ( B, \infty)  \\
(b): \; \forall \mu > 0, \forall \epsilon>0, \exists B > 0 \; \text{s.t} \left[ x > B \Longrightarrow F^{\mu}(x^{\epsilon})  \geq \frac{1}{10} F^{\mu} (x) \right] \\
(c): \; \forall \mu > 0, \; \forall \nu > 0, \; F^{\mu}(x^{\nu}) \lesssim F^{\mu}(x)
\end{array}
\label{Eqn:FProp}
\end{equation}

\begin{equation}
\begin{array}{ll}
F^{[i]}(x) & = O \left( \frac{F(x)}{x^{i}}  \right), \; \tau \in [0,1]: \; F(|\tau x + (1- \tau)y|^{2})  \lesssim F(|x|^{2}) + F(|y|^{2}), \; \text{and}
\end{array}
\label{Eqn:HypF}
\end{equation}

\begin{equation}
\begin{array}{ll}
G^{[i]}(x,\bar{x}) & = O(|x|^{ \beta + 1 - i})
\end{array}
\label{Eqn:HypG}
\end{equation}
for $ 0 \leq i \leq k' $. Here $F^{[i]}$ and $G^{[i]}$ denote the $i^{th}$ derivative of $F$ and $G$ respectively. Assume that there exists $ Q > 0$ such that $X_{\check{k}}(I,u) \leq Q$. Then

\begin{equation}
\begin{array}{ll}
\| \langle D \rangle^{k' - 1 + \alpha} ( G(u,\bar{u}) F(|u|^{2}) ) \|_{L_{t}^{\bar{q}} L_{x}^{\bar{r}}(I)} &
\lesssim \| \langle D  \rangle^{k' - 1 + \alpha} u \|_{L_{t}^{q} L_{x}^{r}(I)}
\| u \|^{\beta}_{L_{t}^{\frac{2(n+1)}{n-2}} L_{x}^{\frac{2(n+1)}{n-2}} (I)} \check{F}(Q) \cdot
\end{array}
\label{Eqn:FracLeibnk}
\end{equation}
More generally, let $\tilde{F} := \mathbb{R}^{+} \rightarrow \mathbb{R}^{+} $ be a nondecreasing function. Remove the assumption ``$F$ is nondecreasing''. Replace $F$ with $\tilde{F}$ on the right-hand side of the equality of (\ref{Eqn:HypF}), in the inequality of (\ref{Eqn:HypF}), and on the right-hand side of (\ref{Eqn:FracLeibnk}). With these
 substitutions made, if $F$, $\tilde{F}$, and $G$ satisfy (\ref{Eqn:HypF}) and (\ref{Eqn:HypG}), then $F$ and $G$ satisfy (\ref{Eqn:FracLeibnk}).

\label{Prop:LeibnJensen}
\end{prop}

\begin{proof}

Recall the usual product rule for fractional derivatives

\begin{equation}
\begin{array}{ll}
\left\| \langle D \rangle^{\alpha_{1}} (fg) \right\|_{L^{\tilde{q}}} & \lesssim
\left\| \langle D \rangle^{\alpha_{1}} f \right\|_{L^{\tilde{q}_{1}}} \| g \|_{L^{\tilde{q}_{2}}} +
\| f \|_{L^{\tilde{q}_{3}}} \left\| \langle D \rangle^{\alpha_{1}} g \right\|_{L^{\tilde{q}_{4}}},
\end{array}
\label{Eqn:LeibnProd}
\end{equation}
and the Leibnitz rule for fractional derivative

\begin{equation}
\begin{array}{ll}
\left\| \langle D \rangle^{\alpha_{2}} H(f) \right\|_{L^{\tilde{q}}} & \lesssim \| \langle D  \rangle^{\alpha_{2}} f \|_{L^{\tilde{q}_{1}}}  \| \tilde{H}(f) \|_{L^{\tilde{q}_{2}}},
\end{array}
\label{Eqn:LeibnCompos}
\end{equation}
if $ 0 \leq \alpha_{1} < \infty $, $0 < \alpha_{2} \leq 1 $, $\tilde{q} \in (1,\infty)$, $ (\tilde{q}_{1}, \tilde{q}_{4}) \in (1,\infty)^{2}$,
$(\tilde{q}_{2}, \tilde{q}_{3}) \in (1,\infty]^{2}$, $\frac{1}{\tilde{q}}= \frac{1}{\tilde{q}_{1}} + \frac{1}{\tilde{q_{2}}} = \frac{1}{\tilde{q}_{3}}
+ \frac{1}{\tilde{q}_{4}}$, and $H$ is a $\mathcal{C}^{1}$ function that satisfies the following property: $H(0)=0$, there exists a function $\tilde{H}$ such that for
all $ \tau \in [0,1] $ we have $ \left| H^{'} \left( \tau x + (1- \tau)y  \right) \right| \lesssim \tilde{H} (x) + \tilde{H}(y) $ (see e.g Taylor \cite{taylor} and references therein ) \footnote{Notation abuse: $\tilde{H}(x)$, $H(x)$, and $H^{'}(x)$ mean
$\tilde{H}(x,\bar{x})$, $H(x,\bar{x})$, and $H^{'}(x,\bar{x})$ respectively. } \\
Let $ 1 < p < \infty$. We recall the following facts that we use throughout the proof. \\
\underline{Fact $1$}: the multiplier $D \langle D \rangle^{-1}$ is bounded as an operator from $L^{p}$ to $L^{p}$. This follows from the H\"ormander-Mikhlin multiplier theorem. \\
\underline{Fact $2$}:  $\| \langle D \rangle f \|_{L^{p}} \lesssim  \| f \|_{L^{p}} + \| D f \|_{L^{p}} $. This follows from the decomposition $f = P_{0} f + (Id - P_{0})f$, the triangle inequality, and the previous fact. \\
\underline{Fact $3$}: $\| D f \|_{L^{p}} \approx \| \nabla f \|_{L^{p}}$: this follows from the boundedness of the Riesz transforms $R_{j}$ defined by $ \widehat{R_{j} f}(\xi)    := -i \frac{ \xi_{j} }{ |\xi | }  \hat{f}(\xi) $  for $ j \in \{ 1,2,...,n \} $ \footnote{Indeed we have $\|\nabla f \|_{L^{p}} \lesssim \sum \limits_{j=1}^{n} \| \p_{x_{j}} f  \|_{L^{p}} \lesssim  \| D f \|_{L^{p}} $. We also have $ f =  \sum \limits_{j=1}^{n} R_{j} D^{-1}  \p_{x_{j}}  f $  which implies that $ \| f \|_{L^{p}} \lesssim \| D^{-1} \nabla f \|_{L^{p}} $ (in other words $\| D f \|_{L^{p}} \lesssim  \| \nabla f \|_{L^{p}}$ ).} \\
\\
Assume that $k' = 1 $. Proposition \ref{Prop:JensenTypeIneq} shows that

\begin{equation}
\begin{array}{ll}
\| \langle D \rangle^{\alpha} \left( G(u,\bar{u}) F(|u|^{2}) \right) \|_{L_{t}^{\bar{q}} L_{x}^{\bar{r}}(I)} &
\lesssim \| \langle D  \rangle^{\alpha} u \|_{L_{t}^{q} L_{x}^{r}(I)}
\| u^{\beta} F(|u|^{2}) \|_{L_{t}^{\frac{2(n+1)}{(n-2) \beta}} L_{x}^{\frac{2(n+1)}{(n-2) \beta}}(I)} \\
& \lesssim \| \langle D  \rangle^{\alpha} u \|_{L_{t}^{q} L_{x}^{r}(I)} \| u F^{\beta^{-1}} (|u|^{2}) \|^{\beta}_{L_{t}^{\frac{2(n+1)}{n-2}}
L_{x}^{\frac{2(n+1)}{n-2}}(I)} \\
& \lesssim \| \langle D  \rangle^{\alpha} u \|_{L_{t}^{q} L_{x}^{r}(I)} \| u \|^{\beta}_{L_{t}^{\frac{2(n+1)}{n-2}} L_{x}^{\frac{2(n+1)}{n-2}}(I)} \check{F}(Q)
\end{array}
\nonumber
\end{equation}
The more general statement follows exactly the same steps and its proof is left to the reader.\\
\\
Assume that the results holds for $k' > 1 $. Let us prove that it also holds for $k'+ 1$. We have

\begin{equation}
\begin{array}{ll}
\left\| \langle D \rangle^{k' + \alpha} \left( G(u,\bar{u}) F(|u|^{2})  \right) \right\|_{L_{t}^{\bar{q}} L_{x}^{\bar{r}}(I)} \\
\lesssim  \| |u|^{\beta + 1} F( |u|^{2}) ) \|_{L_{t}^{\bar{q}} L_{x}^{\bar{r}}(I)} +
\left\| \langle D \rangle^{k' - 1 + \alpha} \nabla ( G(u,\bar{u}) F(|u|^{2}) ) \right\|_{L_{t}^{\bar{q}} L_{x}^{\bar{r}}(I)} \\
\lesssim \| u \|_{L_{t}^{q} L_{x}^{r}(I)}  \| |u|^{\beta}  F(|u|^{2}) \|_{L_{t}^{\frac{2(n+1)}{(n-2) \beta}} L_{x}^{\frac{2(n+1)}{(n-2) \beta}}(I)}
+  \| \langle D \rangle^{k' - 1 + \alpha} ( \p_{z}G(u,\bar{u}) \nabla u  F(|u|^{2}) ) \|_{L_{t}^{\bar{q}} L_{x}^{\bar{r}}(I)} \\
+ \| \langle D \rangle^{k' - 1 + \alpha} ( \p_{\bar{z}} G(u,\bar{u}) \overline{\nabla u}  F(|u|^{2}) ) \|_{L_{t}^{\bar{q}} L_{x}^{\bar{r}}(I)}
+ \left\| \langle D \rangle^{k' -1 + \alpha} \left( G(u,\bar{u}) F^{'}(|u|^{2}) \Re ( \bar{u} \nabla u ) \right) \right\|_{L_{t}^{\bar{q}} L_{x}^{\bar{r}} (I)} \\
\lesssim Y_{1} + Y_{2} + Y_{3} + Y_{4}
\end{array}
\nonumber
\end{equation}
Let $\theta := \frac{1}{k' + \alpha}$. Let $(\bar{q}_{1},\bar{r}_{1})$ be such that
$\left( \frac{1}{\bar{q}_{1}}, \frac{1}{\bar{r}_{1}} \right) = (1 - \theta) \left( \frac{1}{q}, \frac{1}{r} \right) + \theta \frac{(n-2)}{2(n+1)} (1,1)$. Let
$(q_{1},q_{2},r_{1},r_{2})$ be such that $ \left( \frac{1}{q_{1}}, \frac{1}{r_{1}} \right) = (\beta -1)  \frac{(n-2)}{2(n+1)} (1,1)
+ \left( \frac{1}{\bar{q}_{1}}, \frac{1}{\bar{r}_{1}} \right) $ and
$ \left( \frac{1}{\bar{q}}, \frac{1}{\bar{r}} \right) = \left( \frac{1}{q_{1}}, \frac{1}{r_{1}} \right) + \left( \frac{1}{q_{2}}, \frac{1}{r_{2}} \right)$.
Observe that  $\left(  \frac{1}{q_{2}}, \frac{1}{r_{2}}\right) = \theta \left( \frac{1}{q}, \frac{1}{r} \right)
+ (1- \theta) \frac{ (n-2)}{2(n+1)} (1,1) $. \\
\\
We have

\begin{equation}
\begin{array}{ll}
Y_{1} & \lesssim \| u \|_{L_{t}^{q} L_{x}^{r}(I)} \| u F^{\beta^{-1}}(|u|^{2}) \|^{\beta}_{L_{t}^{\frac{2(n+1)}{n-2}} L_{x}^{\frac{2(n+1)}{n-2}} (I)} \\
& \lesssim \| \langle D  \rangle^{k'+ \alpha} u \|_{L_{t}^{q} L_{x}^{r}(I)}   \| u \|^{\beta}_{L_{t}^{\frac{2(n+1)}{n-2}} L_{x}^{\frac{2(n+1)}{n-2}}(I)} \check{F}(Q)
\end{array}
\nonumber
\end{equation}

\begin{equation}
\begin{array}{ll}
Y_{2} & \lesssim \| \langle D \rangle^{k' - 1 + \alpha} ( \p_{z}G(u,\bar{u}) F(|u|^{2}) ) \|_{L_{t}^{q_{1}} L_{x}^{r_{1}}(I)}
\| \nabla u \|_{L_{t}^{q_{2}} L_{x}^{r_{2}}(I)}  \\
& + \| \langle D \rangle^{k' + \alpha} u \|_{ L_{t}^{q} L_{x}^{r}(I)}  \| \p_{z} G(u,\bar{u}) F(|u|^{2}) \|_{L_{t}^{\frac{2(n+1)}{(n-2) \beta}} L_{x}^{\frac{2(n+1)}{(n-2) \beta}}(I)} \\
& \lesssim Y_{2,1} + Y_{2,2} \cdot
\end{array}
\label{Eqn:X2}
\end{equation}
We first estimate $Y_{2,2}$. We have

\begin{equation}
\begin{array}{ll}
Y_{2,2} & \lesssim  \| \langle D \rangle^{k' + \alpha} u \|_{ L_{t}^{q} L_{x}^{r}(I)}
\| u^{\beta} F (|u|^{2}) \| _{L_{t}^{\frac{2(n+1)}{n-2} \beta} L_{x}^{\frac{2(n+1)}{(n-2) \beta}}(I)} \\
& \lesssim \| \langle D  \rangle^{k' + \alpha} u \|_{L_{t}^{q} L_{x}^{r}(I)}
\| u \|^{\beta}_{L_{t}^{\frac{2(n+1)}{n-2}} L_{x}^{\frac{2(n+1)}{n-2}} (I)} \check{F}(Q)  \cdot
\end{array}
\label{Eqn:EstX22Est}
\end{equation}
We then estimate $Y_{2,1}$.  By interpolation we get

\begin{equation}
\begin{array}{ll}
\| D u \|_{L_{t}^{q_{2}} L_{x}^{r_{2}}(I)} & \lesssim
\| D^{k' + \alpha} u \|^{\theta}_{L_{t}^{q} L_{x}^{r} (I)} \| u \|^{1- \theta}_{L_{t}^{\frac{2(n+1)}{n-2}} L_{x}^{\frac{2(n+1)}{n-2}}(I)}  \\
\| \langle D \rangle^{k' - 1 + \alpha} u \|_{L_{t}^{\bar{q}_{1}} L_{x}^{\bar{r}_{1}}(I)} & \lesssim
\| \langle D \rangle^{k' + \alpha} u \|^{1- \theta}_{L_{t}^{q} L_{x}^{r} (I)}  \| u \|^{\theta}_{L_{t}^{\frac{2(n+1)}{n-2}} L_{x}^{\frac{2(n+1)}{n-2}}(I)}
\end{array}
\nonumber
\end{equation}
Hence

\begin{equation}
\begin{array}{ll}
Y_{2,1} & \lesssim \| \langle D \rangle^{k'-1 + \alpha} u  \|_{L_{t}^{\bar{q}_{1}} L_{x}^{\bar{r}_{1}} (I)}
\| u \|^{\beta - 1}_{L_{t}^{\frac{2(n+1)}{n-2}} L_{x}^{\frac{2(n+1)}{n-2}}(I)} \check{F}(Q)  \| D u \|_{L_{t}^{q_{2}} L_{x}^{r_{2}} (I)} \\
& \lesssim \| \langle D  \rangle^{k' + \alpha} u \|_{L_{t}^{q} L_{x}^{r}(I)}
\| u \|^{\beta}_{L_{t}^{\frac{2(n+1)}{n-2}} L_{x}^{\frac{2(n+1)}{n-2}} (I)} \check{F}(Q) \cdot
\end{array}
\nonumber
\end{equation}
Similarly $ Y_{3} \lesssim  \| \langle D  \rangle^{k' + \alpha} u \|_{L_{t}^{q} L_{x}^{r}(I)}
\| u \|^{\beta}_{L_{t}^{\frac{2(n+1)}{n-2}} L_{x}^{\frac{2(n+1)}{n-2}} (I)} \check{F}(Q) $. Let
$ \breve{F} (x) := x F^{'}(x) $, $ G_{1}(x,\bar{x}) :=  \frac{G(x,\bar{x})}{x} $, and
$G_{2}(x,\bar{x}) :=  \frac{G(x,\bar{x})}{\bar{x}}$ \footnote{Notation abuse: ``$\frac{G(x,\bar{x})}{x}$'' means $\frac{G(x,\bar{x})}{x}$ if $x \neq 0$ and it means ``$0$''
if $x=0$. We use a similar notation regarding ``$\frac{G(x,\bar{x})}{\bar{x}}$''}. From the induction assumption and elementary estimates of
$\reallywidecheck{\breve{F}}$  we see that

\begin{equation}
\begin{array}{ll}
Y_{4} & \lesssim
\left\| \langle D \rangle^{k' -1 + \alpha} \left( G_{1}(u,\bar{u}) \breve{F}(|u|^{2}) \nabla u \right)  \right\|_{L_{t}^{\bar{q}} L_{x}^{\bar{r}} (I)}
+  \left\| \langle D \rangle^{k' -1 + \alpha} \left( G_{2}(u,\bar{u}) \breve{F}(|u|^{2})  \nabla \bar{u} \right)  \right\|_{L_{t}^{\bar{q}} L_{x}^{\bar{r}} (I)} \\
& \lesssim  \sum \limits_{i \in \{ 1,2 \} }  \left\| \langle D \rangle^{k'-1 + \alpha} \left( G_{1}(u,\bar{u}) \breve{F}(|u|^{2}) \right) \right\|_{L_{t}^{q_{1}} L_{x}^{r_{1}}(I)}
\| \nabla u \|_{L_{t}^{q_{2}} L_{x}^{r_{2}}(I)} \\
& + \left\| \langle D \rangle^{k' + \alpha} u  \right\|_{L_{t}^{q} L_{x}^{r} (I)} \sum \limits_{i \in \{ 1,2 \} }  \| G_{i}(u,\bar{u}) \breve{F}(|u|^{2})  \|_{L_{t}^{\frac{2(n+1)}{(n-2) \beta}} L_{t}^{\frac{2(n+1)}{(n-2) \beta}}(I) } \\
& \lesssim \left\| \langle D \rangle^{k' -1 + \alpha} u \right\|_{L_{t}^{\bar{q}_{1}} L_{x}^{\bar{r}_{1}}(I)} \| u \|^{\beta-1}_{L_{t}^{\frac{2(n+1)}{n-2}}  L_{x}^{\frac{2(n+1)}{n-2}}(I)} \check{F}(Q)  \| \nabla u \|_{L_{t}^{q_{2}} L_{x}^{r_{2}} (I)}  \\
& + \left\| \langle D \rangle^{k' + \alpha} u  \right\|_{L_{t}^{q} L_{x}^{r} (I)}
\sum \limits_{i \in \{ 1,2 \} }  \| G_{i}(u,\bar{u}) \breve{F}(|u|^{2})  \|_{L_{t}^{\frac{2(n+1)}{(n-2) \beta}} L_{t}^{\frac{2(n+1)}{(n-2) \beta}}(I) }  \\
& \lesssim Y_{4,1} + Y_{4,2} \cdot
\end{array}
\nonumber
\end{equation}
We estimate $Y_{4,1}$ (resp. $Y_{4,2}$ ) in a similar way as $Y_{2,1}$ (resp. $Y_{2,2}$), taking into account the pointwise
estimate $ \left| G_{i}(u,\bar{u}) \breve{F}(|u|^{2}) \right| \lesssim \left| F(|u|^{2}) u^{\beta} \right| $ for $i \in \{ 1,2 \}$. We get

\begin{equation}
\begin{array}{ll}
Y_{4} & \lesssim
\| \langle D  \rangle^{k' + \alpha} u \|_{L_{t}^{q} L_{x}^{r}(I)}
\| u \|^{\beta}_{L_{t}^{\frac{2(n+1)}{n-2}} L_{x}^{\frac{2(n+1)}{n-2}} (I)} \check{F}(Q) \cdot
\end{array}
\end{equation}
The more general statement follows exactly the same steps and its proof is left to the reader.

\end{proof}

\section{Proof of Proposition \ref{Prop:GlobWellPosedCrit}}
\label{Sec:CritBlowUp}

Assume that $ \| u \|_{L_{t}^{\frac{2(n+1)}{n-2}} L_{x}^{\frac{2(n+1)}{n-2}} (I_{max})} < \infty $. \\
Let $ \tilde{k}  = \min ( k, \bar{k} ) $ with $\bar{k}$ defined as follows:

\begin{equation}
\left\{
\begin{array}{l}
n=3: \; \bar{k} := \frac{9}{2}-, \\
n=4: \; \bar{k} := \frac{5}{2}-, \; \text{and} \\
n=5: \; \bar{k}:= \frac{5}{2}-
\end{array}
\right.
\nonumber
\end{equation}
The number $\tilde{k}$ will allow us to use Proposition \ref{Prop:LeibnJensen} \footnote{in particular the assumption
$\beta > k' -1 $ with $ \beta:= \frac{4}{n-2} $ will be satisfied }.\\ Let $ 0 <  \epsilon \ll 1$ be a constant small enough such that all the estimates below are true.
Let $K:= [0,a] \subset I_{\max}$. Let $ t \in K $. We see from Proposition \ref{Prop:LeibnJensen} and (\ref{Eqn:StrichEst}) that there exist $ C \gtrsim 1 $ and
$ C' \gtrsim  1$ such that

\begin{equation}
\begin{array}{ll}
X_{\tilde{k}} \left( [0,t],u \right) & \leq C \left\| (u_{0},u_{1}) \right\|_{H^{k} \times H^{k-1} } + C \left\| \langle D \rangle^{\tilde{k} - \frac{1}{2}} \left( |u|^{1_{2}^{*}-2} u g(|u|) \right)
\right\|_{L_{t}^{\frac{2(n+1)}{n+3}} L_{x}^{\frac{2(n+1)}{n+3}}([0,t])  }  \\
& \leq C'  \| (u_{0},u_{1}) \|_{H^{k} \times H^{k-1}}
+ C' \| \langle D \rangle^{\tilde{k} - \frac{1}{2}} u \|_{L_{t}^{\frac{2(n+1)}{n-1}} L_{x}^{\frac{2(n+1)}{n-1}}([0,t])}
\| u \|^{1_{2}^{*}-2}_{L_{t}^{\frac{2(n+2)}{n-2}} L_{x}^{\frac{2(n+1)}{n-2}}([0,t])}  \\
& g \left( X_{\tilde{k}}([0,t],u) \right)
\end{array}
\label{Eqn:EstXkp}
\end{equation}
Assume now that

\begin{equation}
\begin{array}{ll}
\| u \|_{L_{t}^{\frac{2(n+1)}{n-2}} L_{x}^{\frac{2(n+1)}{n-2}} (K)} & \leq \frac{\epsilon}{g^{\frac{n-2}{4}} \left( 2 C' \| (u_{0},u_{1}) \|_{H^{k}}  \right)} \cdot
\end{array}
\label{Eqn:PropK}
\end{equation}
Then we see from (\ref{Eqn:EstXkp}), (\ref{Eqn:PropK}), and a continuity argument that $ X_{\tilde{k}} ( K ,u ) \leq 2 C' \| (u_{0},u_{1}) \|_{H^{k} \times H^{k-1}} $. Since

\begin{equation}
\begin{array}{ll}
\sum \limits_{j=1}^{\infty} \frac{\epsilon^{\frac{2(n+1)}{n-2}}}{g^{\frac{n+1}{2}} \left( (2 C')^{j} \| (u_{0},u_{1}) \|_{H^{k} \times H^{k-1}} \right) } = \infty,
\end{array}
\nonumber
\end{equation}
we see that we can partition $ I_{\max} \cap [0,\infty) $ into subintervals $ (K_{j})_{1 \leq j \leq J }$ (with $J < \infty$) such that
$ \| u \|_{L_{t}^{\frac{2(n+1)}{n+2}} L_{x}^{\frac{2(n+1)}{n+2}} (K_{j})}
= \frac{\epsilon}{g^{\frac{n-2}{4}} \left( (2 C')^{j} \| (u_{0},u_{1}) \|_{H^{k} \times H^{k-1}} \right)} $ and
$ \| u \|_{L_{t}^{\frac{2(n+1)}{n+2}} L_{x}^{\frac{2(n+1)}{n+2}} (K_{J})}
\leq \frac{\epsilon}{g^{\frac{n-2}{4}} \left( (2 C')^{j} \| (u_{0},u_{1}) \|_{H^{k} \times H^{k-1}} \right)} $. By iteration over $j$ we get
$X_{\tilde{k}} \left( I_{\max} \cap [0,\infty), u \right) < \infty$. Proceeding similarly on $I_{max} \cap (- \infty, 0]$ we also get
$ X_{\tilde{k}} \left( I_{\max} \cap (- \infty , 0], u \right) < \infty $. Hence  $ X_{\tilde{k}} ( I_{\max}, u ) < \infty $. \\
We then prove that $X_{k}(I_{max}, u) < \infty $ by using nonlinear estimates (see Lemma \ref{lem:NonlinearControl} below)
and an induction process on $k$ : see Appendix $B$ and Appendix $C$ in \cite{triroyrad} for a similar argument.

\begin{lem}
Let $ 1 \geq \delta' \geq 0 $. Assume that $\| u \|_{L_{t}^{\frac{2(n+1)}{n-2}} L_{x}^{\frac{2(n+1)}{n-2}} (I) } \leq \delta'$.
Let

\begin{equation}
A_{k}(I,u) :=
\left\{
\begin{array}{l}
\left\| \langle D \rangle^{k - \frac{1}{2}} \left( |u|^{1_{2}^{*}-2} u g(|u|) \right) \right\|_{L_{t}^{\frac{2(n+1)}{n+3}} L_{x}^{\frac{2(n+1)}{n+3}} (I)} \; \text{if} \; n \in \{ 3,4 \} \\
\left\| \langle D \rangle^{k-1} \left( |u|^{1_{2}^{*}-2} u g(|u|) \right) \right\|_{L_{t}^{1} L_{x}^{2} (I)} \; \text{if} \; n=5
\end{array}
\right.
\nonumber
\end{equation}
Then there exist $\bar{c} > 0 $ and $\bar{C} > 0$ such that

\begin{equation}
\begin{array}{l}
A_{k}(I,u) \lesssim (\delta')^{\bar{c}} X_{k} (I,u) \left( \langle  X_{\tilde{k}} (I,u) \rangle^{\bar{C}}
+ \langle X_{ k- \frac{1}{4}} (I,u) \rangle^{\bar{C}} \right)  \cdot
\end{array}
\label{Eqn:PropNonLin1}
\end{equation}

\label{lem:NonlinearControl}
\end{lem}

We postpone the proof of Lemma \ref{lem:NonlinearControl} to Appendix $A$. Assume that $ \tilde{k} < k < \tilde{k} + \frac{1}{4} $.
Let $ \delta^{'} > 0 $ be small enough such that all the statements below hold. Let $K := [0,a] \subset I_{max}$ be such that $\| u \|_{L_{t}^{\frac{2(n+1)}{n-2}} L_{x}^{\frac{2(n+1)}{n-2}} (K)} \leq \delta^{'}$. Let $t \in K$. Let $\bar{M}$ be a large constant such that $X_{\tilde{k}}(I_{max},u) \leq \bar{M}$. We see from
(\ref{Eqn:StrichEst}) and (\ref{Eqn:StrichEst2}) that  $X_{k}([0,t],u) \lesssim \left\| (u_{0},u_{1}) \right\|_{H^{k} \times H^{k-1} } + A_{k}([0,t],u) $. Hence, by applying Lemma \ref{lem:NonlinearControl}, there exists $C' \gtrsim 1$ such that

\begin{equation}
\begin{array}{ll}
X_{k} \left( [0,t],u \right) & \leq C' \left\| (u_{0},u_{1}) \right\|_{H^{k} \times H^{k-1} }  +
C' (\delta^{'})^{\bar{c}}  \langle \bar{M} \rangle^{\bar{C}}  X_{k} \left( [0,t],u \right)
\end{array}
\label{Eqn:EstXkp}
\end{equation}
A continuity argument shows that $ X_{k} ( K,u ) \leq 2 C' \left\| (u_{0},u_{1}) \right\|_{H^{k} \times H^{k-1} }  $. More generally let
$ j \in \mathbb{N} $ and let $ \tilde{k} + \frac{j}{4} \leq k <  \tilde{k} + \frac{j+1}{4}  $. Assume that
$X_{k'} ( K, u ) < \infty$ for $ \tilde{k} \leq  k' < \tilde{k} + \frac{j}{4}  $. Then by using a similar procedure as above we see that $ X_{k}( K,u ) < \infty $. By using again an iteration procedure (see proof of
$ X_{\tilde{k}} ( I_{\max}, u ) < \infty $) we see that $X_{k} (I_{max},u) < \infty$. \\
\\
We write $I_{max} = (a_{max},b_{max})$.  Observe from Remark \ref{Rem:Delta} and the conclusion above that
$ \inf \limits_{ t \in I_{max}} \delta \left( \left\| \left( u(t), \partial_{t} u(t) \right) \right\|_{H^{k} \times H^{k-1} }  \right) > 0 $. We still denote this infimum by $\delta$ for the sake of simplicity. We may assume WLOG that $\delta > 0 $ is small enough such that all the statements below are true. Choose $ \bar{t} < b_{max} $ close enough to $b_{max}$ so that $ \| u \|_{L_{t}^{\frac{2(n+1)}{n-2}} L_{x}^{\frac{2(n+1)}{n-2}}([\bar{t}, b_{max}))} \ll \delta $ and
$ \left\| \langle D \rangle^{\tilde{k} - \frac{1}{2}} u \right\|_{ L_{t}^{\frac{2(n+1)}{n-1}} L_{x}^{\frac{2(n+1)}{n-1}} ([\bar{t}, b_{max}))} \ll  \delta $ such all the estimates below are true. Let $r$ be such that $\frac{n-2}{2(n+1)} + \frac{n}{r} = \frac{n}{2} - \frac{1}{2}$. By the Sobolev embedding
$H^{\tilde{k} - \frac{1}{2}, r} \hookrightarrow L^{\frac{2(n+1)}{n-2}} $ we see that

\begin{equation}
\begin{array}{ll}
\| u_{nl,\bar{t}} \|_{L_{t}^{\frac{2(n+1)}{n-2}} L_{x}^{\frac{2(n+1)}{n-2}} ( [\bar{t}, b_{max} ) )} & \lesssim
\| \langle D \rangle^{\tilde{k} - \frac{1}{2}} u_{nl,\bar{t}} \|_{L_{t}^{\frac{2(n+1)}{n-2}} L_{x}^{r}([\bar{t}, b_{max} ))} \\
& \lesssim \left\| \langle D \rangle^{\tilde{k} - \frac{1}{2}} \left( |u|^{1_{2}^{*}-2} u g(|u|) \right) \right\|_{L_{t}^{\frac{2(n+1)}{n+3}} L_{x}^{\frac{2(n+1)}{n+3}}
 ([ \bar{t}, b_{max} ))} \\
& \lesssim \left\| \langle D \rangle^{\tilde{k} - \frac{1}{2}} u \right\|_{L_{t}^{\frac{2(n+1)}{n-1}} L_{x}^{\frac{2(n+1)}{n-1}}([ \bar{t}, b_{max} ))}
\| u \|^{1_{2}^{*} -2}_{L_{t}^{\frac{2(n+1)}{n+2}}  L_{x}^{\frac{2(n+1)}{n+2}}  ([ \bar{t}, b_{max} ))} \\
& g \left( X_{\tilde{k}} \left( [ \bar{t}, b_{max}) ,u \right) \right)   \\
& \\
& \ll \delta^{\frac{4}{n-2}} \cdot
\end{array}
\nonumber
\end{equation}
 Let $ T(t) := \cos \left( (t - \bar{t}) \langle D \rangle \right) u(\bar{t})  +
\frac{ \sin \left( (t - \bar{t}) \langle D \rangle \right)  \partial_{t} u(\bar{t}) }{ \langle D \rangle}
$. The triangle inequality yields

\begin{equation}
\begin{array}{ll}
\| T \|_{L_{t}^{\frac{2(n+1)}{n+2}}  L_{x}^{\frac{2(n+1)}{n+2}} \left( [\bar{t}, b_{max} ) \right) }
& \lesssim \| u \|_{L_{t}^{\frac{2(n+1)}{n+2}} L_{x}^{\frac{2(n+1)}{n+2}} \left( [\bar{t},b_{max}) \right) }  +
\| u_{nl} \|_{L_{t}^{\frac{2(n+1)}{n-2}} L_{x}^{\frac{2(n+1)}{n-2}} \left( [ \bar{t}, b_{max}) \right)} \\
& \leq \frac{\delta}{2} \cdot
\end{array}
\nonumber
\end{equation}
We also have $ \| T \|_{L_{t}^{\frac{2(n+1)}{n+2}}  L_{x}^{\frac{2(n+1)}{n+2}} \left( [ \bar{t}, \infty ) \right) }  \lesssim
\left\| \left( u(\bar{t}), \partial_{t} u(\bar{t}) \right) \right\|_{H^{k} \times H^{k-1}} < \infty $. These facts combined
with elementary considerations \footnote{i.e the continuity of $ s  \rightarrow  \| T \|_{\left( [\bar{t}, b_{max} + s] \right)} $, that is a
consequence of the monotone convergence theorem } show that there exists $\epsilon > 0 $ such that
$ \| T \|_{L_{t}^{\frac{2(n+1)}{n+2}}  L_{x}^{\frac{2(n+1)}{n+2}} \left( [ \bar{t}, b_{max} + \epsilon ] \right) }  \leq \delta $.
Hence contradiction with Proposition \ref{Prop:LocalWell}.

\section{Proof of Theorem \ref{thm:main}}   % \blc CHy \bc
\label{Sec:Thmmain}

Let $\tilde{k}$ be the number defined in Section \ref{Sec:CritBlowUp}. Again the number $\tilde{k}$ will allow us to
use Proposition \ref{Prop:LeibnJensen}. Our goal is to find a finite bound of $ X_{\tilde{k}} \left( [0,\epsilon_{0}), u \right) $ of the form
$ X_{\tilde{k}} \left( [0,\epsilon_{0}), u  \right) \leq f \left( \| (u_{0},u_{1}) \|_{H^{k} \times H^{k-1}} \right)$  with $f$ a function
that has finite values: this implies a finite bound of $ \| u \|_{L_{t}^{\frac{2(n+1)}{n-2}} L_{x}^{\frac{2(n+1)}{n-2}} ([0,\epsilon_{0}))} $ of the same form
(and hence global well-posedness, by Remark \ref{Rem:GlobalCrit}) from the estimates below that are consequences of interpolation and
the Sobolev embeddings $H^{\tilde{k}- \frac{1}{2},r} \hookrightarrow L^{\frac{2(n+1)}{n-2}}$ and
$ H^{\tilde{k}} \hookrightarrow H^{\tilde{k}- \frac{1}{2}, \frac{2n}{n-1}} $

\begin{equation}
\begin{array}{ll}
\| u \|_{L_{t}^{\frac{2(n+1)}{n-2}} L_{x}^{\frac{2(n+1)}{n-2}} ([0,\epsilon_{0}))} & \lesssim
\| u \|_{ L_{t}^{\frac{2(n+1)}{n-2}} H^{\tilde{k} - \frac{1}{2}, r} ([0, \epsilon_{0}))} \\
& \lesssim  \| u \|^{\theta}_{L_{t}^{\frac{2(n+1)}{n-1}} H^{\tilde{k} - \frac{1}{2}, \frac{2(n+1)}{n-1}} ([0, \epsilon_{0}))}
\| u \|^{1- \theta}_{L_{t}^{\infty} H^{\tilde{k} - \frac{1}{2}, \frac{2n}{n-1} } ([0, \epsilon_{0}))} \\
& \lesssim  X_{\tilde{k}} \left( [0,\epsilon_{0}),u \right) \cdot
\end{array}
\nonumber
\end{equation}
Here $r$ and $\theta$ are defined by $ \frac{n-2}{2(n+1)} + \frac{n}{r} = \frac{n}{2} - \frac{1}{2} $ and
$ \theta = \frac{n-2}{n-1}$. \\
Let $ M := M \left( \| (u_{0},u_{1}) \|_{H^{k} \times H^{k-1}} \right)  \gg 1 $ (resp. $ 0 <  \epsilon \ll 1 $)  be a positive constant large enough (resp. small enough) such that all the estimates and the statements below are true. \\
We claim that $ X_{\tilde{k}}( [0,\epsilon_{0}),u) < M $. If not elementary considerations show that there exists $ \bar{t} > 0 $ such that
$ X_{\tilde{k}}( [0,\bar{t}],u) = M $ and for all $t \in [0,\bar{t}]$, we have $ X_{\tilde{k}}([0,t],u) \leq M $. So we see that (\ref{Eqn:BoundStrichLong})
holds with $ K := [0,\bar{t}] $. Let $J$ be a subinterval of the form $ [a,\cdot) $ or $[a,\cdot]$ such that

\begin{equation}
\| u \|_{L_{t}^{\frac{2(n+1)}{n-2}} L_{x}^{\frac{2(n+1)}{n-2}} (J)} \leq  \frac{\epsilon}{g^{\frac{n-2}{4}}(M)} \cdot
\nonumber
\end{equation}
If $t \in J $ then the Strichartz-type estimates (\ref{Eqn:StrichEst}) and Proposition \ref{Prop:LeibnJensen} show that there exists $C \gtrsim 1$ such that

\begin{equation}
\begin{array}{l}
X_{\tilde{k}}([a,t],u) \\
\lesssim \left\| (u(a), \p_{t} u(a)) \right\|_{H^{k} \times H^{k-1}} +
\left\| \langle D \rangle^{\tilde{k} - \frac{1}{2}} \left( |u|^{\frac{4}{n-2}} u  g(|u|) \right) \right\|_{L_{t}^{\frac{2(n+1)}{n+3}} L_{x}^{\frac{2(n+1)}{n+3}} ([a,t])}  \\
\leq C \left\| (u(a), \p_{t} u(a)) \right\|_{H^{k} \times H^{k-1}}
+ C \| \langle D  \rangle^{\tilde{k} - \frac{1}{2}} u \|_{L_{t}^{\frac{2(n+1)}{n-1}} L_{x}^{\frac{2(n+1)}{n-1}}([a,t])}
\| u \|^{1_{2}^{*}-2}_{L_{t}^{\frac{2(n+1)}{n-2}} L_{x}^{\frac{2(n+1)}{n-2}} ([a,t])} g(M) \cdot
\end{array}
\nonumber
\end{equation}
Hence a continuity argument shows that $ X_{\tilde{k}}(J,u) \leq 2 C \left\| (u(a), \p_{t} u(a)) \right\|_{H^{k} \times H^{k-1}} $. In view of (\ref{Eqn:BoundStrichLong}) we can construct a partition of $K$ into subintervals $(K_{j})_{1 \leq j \leq J}$ such that
$ \| u \|_{L_{t}^{\frac{2(n+1)}{n-2}} L_{x}^{\frac{2(n+1)}{n-2}} (K_{j})} = \frac{\epsilon}{g^{\frac{n-2}{4}}(M)} $ and
$ \| u \|_{L_{t}^{\frac{2(n+1)}{n-2}} L_{x}^{\frac{2(n+1)}{n-2}} (K_{J})} \leq \frac{\epsilon}{g^{\frac{n-2}{4}}(M)} $. Moreover there exists a constant
$\bar{C} \gg 1 $ such that $ J \lesssim \bar{C}^{\bar{C}g^{b_{n}^{+}}(M)} $. By iteration we have

\begin{equation}
\begin{array}{l}
1 \leq j \leq J: \; X_{\tilde{k}}(K_{j},u)  \leq (2C)^{J} \left\| (u(a), \p_{t} u(a)) \right\|_{H^{k} \times H^{k-1}} \cdot
\end{array}
\nonumber
\end{equation}
Hence by increasing the value of $\bar{C}$ if necessary we see from the triangle inequality that

\begin{equation}
\begin{array}{ll}
X_{\tilde{k}}([0,\bar{t}],u) \leq \bar{C}^{\bar{C}^{\bar{C}g^{b_{n}^{+}}(M)}} < M,
\end{array}
\nonumber
\end{equation}
which is a contradiction.

\section{Proof of Proposition \ref{Prop:BoundLong}}
\label{Sec:ProofProp}

The proof relies upon two lemmas that we prove in the next subsections. These lemmas rely on
concentration techniques introduced in \cite{bourg} (see also \cite{nakimrn} for an application of these
techniques to the energy-critical Klein-Gordon equations).

\subsection{A first lemma}

We prove the following lemma:

\begin{lem}
Let $u$ be an $H^{k}-$ solution of (\ref{Eqn:KgBar}). There exists $ 0 <  c  \ll 1$ such that if
$ \eta = c g^{-\frac{1}{1_{2}^{*}-2} } (M) $ and if $ I := [a',b'] \subset K $ is an interval such that
$ \| u \|_{ L_{t}^{\frac{2(n+1)}{n-2}} L_{x}^{\frac{2(n+1)}{n-2}}(I)} = \eta $ then there exist $x_{0} \in \mathbb{R}^{n}$ and a subinterval $J \subset I$
such that for all $t \in J$

\begin{equation}
\begin{array}{ll}
\int_{|x-x_{0}| \leq R} |u(t,x)|^{1_{2}^{*}} \; dx &  \gtrsim \eta^{1_{2}^{*}(n-1)},
\end{array}
\label{Eqn:Concentration}
\end{equation}
with $R$ radius such that  $ 0 <  R \lesssim g^{ \frac{(n+1)(n^{2}-3n + 6)}{8}+}(M)  |J|$.

\label{lem:Concentration}
\end{lem}

\begin{proof}

Let $ 0 <  c \ll 1$ be a constant small enough such that if $ \eta = c g^{-\frac{1}{1_{2}^{*}-2} } (M) $ then all the estimates and statements below are true. \\
Let $I'$ be such that $ a^{'} \in I^{'} \subset I$. Assume that $n=3$. The Strichartz estimates (\ref{Eqn:StrichEst}), the Plancherel theorem, and Proposition \ref{Prop:LeibnJensen} show that

\begin{equation}
\begin{array}{ll}
\| \langle D \rangle^{0+} u \|_{L_{t}^{2+} L_{x}^{\infty-} (I') } &
\lesssim \left\| ( \langle D \rangle u(a') , \p_{t} u(a')) \right\|_{L^{2}}
+ \left\| \langle D \rangle^{0+} ( |u|^{1_{2}^{*} - 2} u g(|u|) ) \right\|_{L_{t}^{1+} L_{x}^{2-}(I')} \\
& \lesssim E^{\frac{1}{2}} + \| \langle D \rangle^{0+} u \|_{L_{t}^{2+} L_{x}^{\infty-} (I')}
\| u \|^{1_{2}^{*} - 2}_{L_{t}^{\frac{2(n+1)}{n-2}} L_{x}^{\frac{2(n+1)}{n-2}}(I')} g(M) \cdot
\end{array}
\nonumber
\end{equation}
Assume now that $n \in \{ 4,5 \}$. Then we get similarly

\begin{equation}
\begin{array}{l}
\|  \langle D \rangle^{\frac{n-3}{2(n-1)}} u \|_{L_{t}^{2} L_{x}^{\frac{2(n-1)}{n-3}}(I') }  \\
\lesssim \left\| ( \langle D \rangle u(a'), \p_{t} u(a') ) \right\|_{L^{2}}
+ \left\|   \langle D \rangle^{\frac{n-3}{2(n-1)}} \left(  |u|^{1_{2}^{*} - 2} u g(|u|) \right) \right\|_{L_{t}^{\frac{2(n+1)}{n+5}} L_{x}^{\frac{2(n-1)(n+1)}{n^{2} + 2 n -7}}(I')} \\
\lesssim E^{\frac{1}{2}} + \| \langle D \rangle^{\frac{n-3}{2(n-1)}} u \|_{L_{t}^{2} L_{x}^{\frac{2(n-1)}{n-3}}(I')}
\| u \|^{1_{2}^{*} - 2}_{L_{t}^{\frac{2(n+1)}{n-2}} L_{x}^{\frac{2(n+1)}{n-2}}(I')} g(M) \cdot
\end{array}
\nonumber
\end{equation}
Hence a continuity argument  shows that

\begin{equation}
\begin{array}{l}
n=3: \; \| \langle D \rangle^{0+} u \|_{L_{t}^{2+} L_{x}^{\infty-} (I) }  \lesssim E^{\frac{1}{2}}  \lesssim 1 \\
n \in \{ 4,5 \}: \; \| \langle D \rangle^{\frac{n-3}{2(n-1)}} u \|_{L_{t}^{2} L_{x}^{\frac{2(n-1)}{n-3}}(I) }
\lesssim E^{\frac{1}{2}} \lesssim 1 \cdot
\end{array}
\nonumber
\end{equation}
By interpolation we see that there exists $ \theta := \theta(n) \in [0,1] $  such that

\begin{equation}
\begin{array}{lll}
n = 3: & \| \langle D \rangle^{\frac{1}{2}} u \|_{L_{t}^{\frac{2(n+1)}{n-1}} L_{x}^{\frac{2(n+1)}{n-1}}(I)} &  \lesssim
\| \langle D \rangle u \|_{L_{t}^{\infty} L_{x}^{2} (I)}^{\theta}
\| \langle D \rangle^{0+} u \|^{1- \theta}_{L_{t}^{2+} L_{x}^{\infty-} (I) }  \lesssim 1 \\
&  &  \\
n \in \{ 4,5 \}: &  \| \langle D \rangle^{\frac{1}{2}} u \|_{L_{t}^{\frac{2(n+1)}{n-1}} L_{x}^{\frac{2(n+1)}{n-1}}(I)}  & \lesssim
\| \langle D \rangle u \|_{L_{t}^{\infty} L_{x}^{2} (I)}^{\theta} \| \langle D \rangle^{\frac{n-3}{2(n-1)}}  u \|^{1- \theta}_{L_{t}^{2} L_{x}^{\frac{2(n-1)}{n-3}}(I) }
\lesssim 1 \cdot
\end{array}
\nonumber
\end{equation}
Next we use the refined Sobolev inequality

\begin{equation}
\begin{array}{ll}
 \| f \|_{L^{\frac{2(n+1)}{n-2}}} & \lesssim \| f \|^{\frac{1}{n-1}}_{B^{1-\frac{n}{2}}_{\infty,\infty}}
\left\|  \langle D \rangle^{\frac{1}{2}} f \right\|^{1 - \frac{1}{n-1}}_{L^{\frac{2(n+1)}{n-1}}} :
\end{array}
\label{Eqn:RefSob}
\end{equation}
this estimate belongs to the well-known class of refined Sobolev inequalities (see e.g \cite{bahchem} and references therein): a proof is given in Appendix $C$. \\
Assume that $|I| \gtrsim \eta^{\frac{2(n+1)}{n-2}} $. Then we see from (\ref{Eqn:RefSob}) that there exist $(t_{0},x_{0}) \in I \times \mathbb{R}^{n}$ and $ N \in \left\{ 0, 2^{\mathbb{N}} \right\}$ such that

\begin{equation}
\begin{array}{ll}
| P_{N} u(t_{0},x_{0})| \gtrsim \eta^{n-1} \langle N \rangle^{\frac{n}{2}-1}
\end{array}
\label{Eqn:LowerBd}
\end{equation}
Now assume that $|I| \ll \eta^{\frac{2(n+1)}{n-2}} $. Let $ Q \in 2^{\mathbb{N}} $. From  H\"older inequality w.r.t time and Bernstein inequality we get

\begin{equation}
\begin{array}{ll}
 \| P_{< Q} u \|_{L_{t}^{\frac{2(n+1)}{n-2}} L_{x}^{\frac{2(n+1)}{n-2}}(I) }  &
\lesssim (|I| Q  )^{\frac{n-2}{2(n+1)}}   \| u \|_{L_{t}^{\infty} H^{1} (I)} \\
& \lesssim ( |I|  Q  )^{\frac{n-2}{2(n+1)}} \cdot
\end{array}
\label{Eqn:BoundLowFreq}
\end{equation}
Let $c_{1}$ and $c_{2}$ be two small positive constants such that all the estimates below are true. Let  $Q \in 2^{\mathbb{N}}$ be such that
$ c_{1} \eta^{\frac{2(n+1)}{n-2}} |I|^{-1} \leq Q \leq c_{2}  \eta^{\frac{2(n+1)}{n-2}} |I|^{-1} $. From (\ref{Eqn:BoundLowFreq}) we see that $ \| P_{\geq Q} u \|_{L_{t}^{\frac{2(n+1)}{n-1}} L_{x}^{\frac{2(n+1)}{n-1}} (I)} \gtrsim \eta $. We have

\begin{equation}
\begin{array}{ll}
\| P_{\geq Q} u \|_{L_{t}^{\frac{2(n+1)}{n-2}} L_{x}^{\frac{2(n+1)}{n-2}}(I)}
& \lesssim \| P_{\geq Q} u \|^{\frac{1}{n-1}}_{L_{t}^{\infty} B_{\infty,\infty}^{1 - \frac{n}{2}} (I)}
\| \langle D \rangle^{\frac{1}{2}} u \|^{1 - \frac{1}{n-1}}_{L_{t}^{\frac{2(n+1)}{n-1}} L_{x}^{\frac{2(n+1)}{n-1}} (I)} \cdot
\end{array}
\nonumber
\end{equation}
Hence there exist $(t_{0},x_{0}) \in I \times \mathbb{R}^{n}$  and $N \in 2^{\mathbb{N}} $ such that $N \geq Q$ and
(\ref{Eqn:LowerBd}) holds. We have

\begin{equation}
\begin{array}{ll}
\| P_{N} u(t) - P_{N} u(t_{0}) \|_{H^{1}} & \lesssim X_{1} + X_{2} + X_{3}, \; \text{with}
\end{array}
\nonumber
\end{equation}

\begin{equation}
\begin{array}{ll}
X_{1} & :=  \left\| \left( \cos ( (t-t_{0}) \langle D \rangle ) - 1 \right) P_{N} u(t_{0}) \right\|_{H^{1}}, \\
X_{2} & := \left\| \frac{ \sin  \left( (t-t_{0}) \langle D \rangle \right) }{\langle D \rangle} P_{N} \p_{t} u(t_{0}) \right\|_{H^{1}}, \; \text{and} \\
& \\
X_{3} & := \left\| \int_{t_{0}}^{t} \frac{\sin \left( (t-t_{0}) \langle D \rangle \right) }{\langle D \rangle}  P_{N} \left( |u(t')|^{1_{2}^{*}-2} u(t') g(|u(t')|) \right) \; d t' \right\|_{H^{1}}
\end{array}
\nonumber
\end{equation}
The Plancherel theorem and elementary considerations show that

\begin{equation}
\begin{array}{ll}
X_{1} + X_{2} & \lesssim |t-t_{0}| \langle N \rangle  \left( \| u(t_{0}) \|_{H^{1}} + \|\p_{t} u(t_{0}) \|_{L^{2}} \right) \\
& \lesssim |t-t_{0}| \langle N \rangle \cdot
\end{array}
\nonumber
\end{equation}
The embeddings $ H^{1} \hookrightarrow H^{1- \frac{2}{n+1}, \frac{2n(1+n)}{n^{2} + n  +4}} $ and $H^{1} \hookrightarrow L^{1_{2}^{*}} $, and Proposition
\ref{Prop:LeibnJensen} show that

\begin{equation}
\begin{array}{ll}
X_{3} & \lesssim \left\| \langle D \rangle \left( P_{N} ( |u|^{1_{2}^{*} -2} u g(|u|) ) \right) \right\|_{L_{t}^{\frac{n+1}{2}} L_{x}^{\frac{2(1+n)}{n + 5}} ([t_{0},t])} \\
&  \lesssim |t-t_{0}|^{\frac{2}{n+1}} \langle N \rangle^{\frac{2}{n+1}}
\left\| \langle D \rangle^{1 - \frac{2}{n+1}} \left( |u|^{1_{2}^{*} -2} u g(|u|)  \right) \right\|_{L_{t}^{\infty} L_{x}^{\frac{2(1+n)}{n+5}} ([t_{0},t])} \\
& \lesssim
\left( |t-t_{0}| \langle N \rangle \right)^{\frac{2}{n+1}} \left\| \langle D \rangle^{1- \frac{2}{n+1}} u \right\|_{L_{t}^{\infty} L_{x}^{\frac{2n(1+n)}{n^{2} + n -4}} ([t_{0},t]) } \| u \|^{1_{2}^{*}-2}_{L_{t}^{\infty} L_{x}^{1_{2}^{*}} ([t_{0},t])} g(M) \\
& \lesssim \left( |t - t_{0}| \langle N \rangle \right)^{\frac{2}{n+1}} g(M)
\end{array}
\nonumber
\end{equation}
Let $c'$ be a positive constant that is small enough for the estimates and statements below to be true. Since $\langle N \rangle^{1 - \frac{n}{2}} \left| P_{N} \left( u(t,x_{0}) \right) - P_{N}  \left( u(t_{0},x_{0}) \right) \right| \lesssim \left\| P_{N} u(t) - P_{N} u(t_{0}) \right\|_{H^{1}} $, we see from the estimates above that if $t$ is an element of

\begin{equation}
\begin{array}{ll}
J  & := \left\{ t \in I: \; |t- t_{0}| = c' \eta^{\frac{(n-1)(n+1)}{2}} \left( g^{\frac{n+1}{2}}(M) \langle N \rangle \right)^{-1} \right\} \\
& = \left\{ t \in I: \; |t-t_{0}| = c' g^{- \frac{(n+1)(n^{2}-3n + 6)}{8}}(M) \langle N \rangle^{-1} \right\},
\end{array}
\nonumber
\end{equation}
then (\ref{Eqn:LowerBd}) holds if $t_{0}$ is replaced with $t$. Hence (\ref{Eqn:LowerBd}) holds on a subinterval $ J \subset I $ such that  $ |J | =  c' g^{-\frac{(n+1)(n^{2}-3n + 6)}{8}}(M)  \langle N \rangle^{-1} $. \\
\\
Let $f$ be a function and let $ N \in \{ 0,2^{\mathbb{N}} \} $. We define $\bar{N}$ and $ \rho $ in the following fashion: $ (\bar{N}, \rho) := ( 1, \phi) $ if $ N=0 $
and $(\bar{N}, \rho) := (N, \psi) $ if $N \in 2^{\mathbb{N}}$. We write  $ P_{N} f (x) = \int_{\mathbb{R}^{n}} \bar{N}^{n}
\check{\rho} \left(  \bar{N}  (x-y) \right) f(y) \; dy $ . Here $\check{\rho}$ is the inverse Fourier transform of $\rho$.
 We consider a number $\eta^{0-}$. Let $\infty-$ be a large constant that is allowed to change from one line to another one and such that the estimates below are true. The fast decay of $\check{\rho}$ implies that there exists $C' := C'_{\infty-}  > 1  $ such that $ | \check{\rho} (x) | \leq \frac{C'}{|x|^{\infty-}} $ if $ |x| \geq 1 $. Let $ C := \eta^{0-} C^{'} $. Hence we see from the H\"older inequality that

\begin{equation}
\begin{array}{ll}
\eta^{n-1} \bar{N}^{\frac{n}{2} - 1} & \lesssim  \| u(t) \|_{L^{1_{2}^{*}}( \bar{N} |y-x_{0}| \leq C )}
\left\| \bar{N}^{n} \check{\rho} \left(  \bar{N} (y - x_{0}) \right) \right\|_{L^{\left( 1_{2}^{*} \right)^{'}}} \\
& \\
& + \| u(t) \|_{L^{1_{2}^{*}}}
\left\| \bar{N}^{n} \check{\rho} \left(  \bar{N} (y - x_{0}) \right)  \right\|_{L^{\left( 1_{2}^{*} \right)^{'}} ( \bar{N} |y -x_{0}| \geq C )} \\
& \\
& \lesssim \bar{N}^{ \frac{n}{2} - 1}
\left(
\| u(t) \|_{L^{1_{2}^{*}} \left( \bar{N} |y-x_{0}| \leq C \right)} + \frac{C^{'}}{C^{\infty-}}
\right) \\
& \\
& \lesssim \bar{N}^ {\frac{n}{2} - 1} \left( \| u(t) \|_{L^{1_{2}^{*}} \left(  \langle N \rangle |y-x_{0}| \leq C  \right)} + o ( \eta^{n-1})  \right) \cdot
\end{array}
\nonumber
\end{equation}
Here $(1_{2}^{*})^{'}$ is the conjugate of $1_{2}^{*}$, i.e $ \frac{1}{(1_{2}^{*})^{'}} + \frac{1}{1_{2}^{*}} = 1 $.
Hence (\ref{Eqn:Concentration}) holds.

\end{proof}

\subsection{A second lemma}

We prove the following lemma:

\begin{lem}
The following estimates hold:

\begin{enumerate}

\item Let $\bar{x} \in \mathbb{R}^{n}$. Then

\begin{equation}
\begin{array}{ll}
\int_{K} \int_{\mathbb{R}^{n}} \frac{|u(t,x)|^{1_{2}^{*}} g(|u(t,x)|)}{|x - \bar{x}|} \; dx \; dt & \lesssim 1 \cdot
\end{array}
\label{Eqn:MorawetzEst}
\end{equation}

\item Let $ \alpha > 1$, $\bar{t} \in K$, and $ \bar{x} \in \mathbb{R}^{n}$. Then

\begin{equation}
\begin{array}{ll}
\left\| \int_{ |t - \bar{t}| \geq \alpha^{-1}|x - \bar{x}| } |u(t,x)|^{1_{2}^{*}} \; dx \right\|_{l^{1} L^{\infty}} & \lesssim 1 \cdot
\end{array}
\label{Eqn:MorawEstCone}
\end{equation}
Here $ \| f \|_{l^{1} L^{\infty}}   :=  \sum \limits_{j=1}^{\infty} \sup \limits_{2^{-j} (a - \bar{t}) \leq |t - \bar{t}| \leq 2^{-j+1} (a - \bar{t})} |f(t)| $.

\end{enumerate}
\label{lem:EstDecay}
\end{lem}

\begin{proof}

We first prove (\ref{Eqn:MorawetzEst}) \footnote{ The proof involves some computations. Strictly speaking, the computations below only hold for $H^{k} - $ solutions ( i.e solutions $ ( u, \partial_{t} u ) \in \mathcal{C} \left( [S,T], H^{k} \right) \times \mathcal{C} \left( [S,T], H^{k-1} \right) $, with $k$ large). In the case where $ n \in \{ 3,4 \}$, one can then show that (\ref{Eqn:MorawetzEst}) also holds for $H^{k}-$ solutions with $  k_{n} > k > 1 $ by a standard approximation argument. In
the case where $n=5$ the nonlinearity is not that smooth: it is not even $\mathcal{C}^{3}$. So one has to smooth out the nonlinearity, get similar estimates as those below for smooth solutions and then take limit by a standard approximation argument to prove that (\ref{Eqn:MorawetzEst}) holds.}. This estimate belongs to the class of well-known Morawetz-type estimates that play an essential role in the proof of scattering of subcritical and critical nonlinear Klein-Gordon equations. The proof is well-known in the literature (see e.g \cite{morstr,strauss} ). Hence we only sketch the proof of these estimates in the framework of barely supercritical Klein-Gordon equations for the convenience of the reader. \\
\\
Throughout the proof we use the following notations. If $f$ is a function depending on $t$, $x_{1}$,.., and $x_{n}$, then
$\partial_{0} f :=  \partial_{t} f $, $\partial^{0} f :=  - \partial_{t} f $, $ \partial_{1} f := \partial_{x_{1}} f $,
$ \partial^{1} f := \partial_{x_{1}} f $, ..., $ \partial_{n} f := \partial_{x_{n}} f $, and
$ \partial^{n} f := \partial_{x_{n}} f $. We define $r := |x|$,
$\p_{r} f := \frac{\nabla f \cdot x}{|x|}$ and $\p_{\theta} f := \nabla f - \frac{\nabla f \cdot x}{|x|} \frac{x}{|x|} $. If two mathematical symbols $A_{i}$ and
$B_{i}$ are indexed by the same variable $i$, then $A_{i} B_{i}$ means that we perform the summation from $i=0$ to $n$. Let
$\Box u := - \partial_{i} \partial^{i} u  $. Then $ \Box u = \partial_{tt} u- \triangle u $. \\
\\
By using the space translation invariance, we may assume WLOG that $\bar{x} = 0$. Let $m := p_{i} \partial_{x_{i}} u  + u q $. Let
$f(u) : = |u|^{1_{2}^{*} -2} u g(|u|)$. Let $u$ be a solution of $ \Box u = -f(u)$. Let $G(z,\bar{z}) := \Re \left( \bar{z} f(z) - F(z,\bar{z}) \right) $. Then we see
from (\ref{Eqn:ExpF}) that  $ G(z,\bar{z}) = |z|^{1_{2}^{*}} g( |z| ) -  \int_{0}^{|z|} t^{1_{2}^{*} -1} g(t) \; dt $.
We get from (\ref{Eqn:IntPartsFz}) and elementary estimates

\begin{equation}
\begin{array}{ll}
G(z,\bar{z}) & = \frac{n+2}{2n} |z|^{1_{2}^{*}} g(|z|) + \int_{0}^{|z|} t^{1_{2}^{*}- 1} g'(t) \; dt \\
& \approx |z|^{1_{2}^{*}} g(|z|) \cdot
\end{array}
\label{Eqn:EstG}
\end{equation}
Recall the well-known formula (see e.g \cite{nakcpde} and references therein)

\begin{equation}
\begin{array}{ll}
\Re \left(  \left( \Box u + f(u) \right)   \bar{m} \right) & =
\partial_{i} \Re \left( - \partial^{i} u \bar{m}  + l(u) p_{i}  + \frac{|u|^{2}}{2} \partial^{i} q \right)
+ \Re \left( \partial_{i} u \partial^{i} p_{j} \overline{\partial_{j} u} \right)  + \frac{|u|^{2}}{2} \Box q \\
& + G(u,\bar{u}) q + ( 2 q - \partial_{i} p_{i} ) l(u) \cdot
\end{array}
\nonumber
\end{equation}
Here $ l(u)  :=  - \frac{1}{2} | \partial_{t} u |^{2} + \frac{1}{2} | \nabla u |^{2} + F(u,\bar{u}) $.  \\
Define $\vec{p} := (0,p_{1},...,p_{n})$ with $p_{j} :=  \frac{x_{j}}{|x|}$ for $ j \in \{1,...,n \} $.  Let $ q  := \frac{\nabla \cdot p}{2} $. Hence after some computations we get
the well-known Morawetz-type estimate (see e.g  \cite{strauss})

\begin{equation}
\begin{array}{ll}
\left| \left[ \int_{\mathbb{R}^{n}} \Re ( \partial_{t} u \bar{m} )  \; dx \right]_{t=S}^{t=T} \right| & \gtrsim
\int_{S}^{T}  \int_{\mathbb{R}^{n}} \frac{|\nabla u|^{2}}{|x|} - \frac{ \left| \nabla u \cdot \frac{x}{|x|}  \right|^{2} }{|x|}  \; dx \, dt
+ (n-1)(n-3) \int_{S}^{T} \int_{\mathbb{R}^{n}}  \frac{|u|^{2}}{|x|^{3}} \; dx \, dt \\
&  + \frac{n-1}{2} \int_{S}^{T} \int_{\mathbb{R}^{n}} \frac{G(u,\bar{u})}{|x|} \; dx \, dt
\end{array}
\nonumber
\end{equation}
Hence we see from the Cauchy-Schwartz inequality and the conservation of energy that (\ref{Eqn:MorawetzEst}) holds. \\
\\
Then we turn our attention to (\ref{Eqn:MorawEstCone}). By using the time translation invariance, the space translation invariance, and the
time reversal invariance we may replace WLOG  $ | t- \bar{t} | $  with $ t- \bar{t}$ in (\ref{Eqn:MorawEstCone}) and in the definition of
$\| f \|_{l^{1} L^{\infty}}$, and assume that $(\bar{x},\bar{t}) = (0,0)$. We then use an argument in \cite{nakimrn}. Recall the following result:

\begin{res}(see \cite{nakimrn} \footnote{The statement of this lemma is actually slightly different from that of Lemma 7.1
in \cite{nakimrn}. Nevertheless the proof is a straightforward modification of that of Lemma 7.1 in \cite{nakimrn}: therefore it is omitted.})
Let $ s > 0$. Let $f$ and $g$ be two functions such that for any $ 0 \leq S <  T \leq a $

\begin{equation}
\begin{array}{ll}
\left[ t^{s} f(t) \right]_{S}^{T} \lesssim T^{s} + \int_{S}^{T} t^{s} |g(t)| \; dt \cdot
\end{array}
\nonumber
\end{equation}
Then we have

\begin{equation}
\begin{array}{ll}
\| f \|_{l^{1} L^{\infty}} \lesssim 1 + \| g \|_{L^{1} \left( [0,a] \right)}
\end{array}
\nonumber
\end{equation}
Here $ \left[ f \right]_{S}^{T} := f(T) - f(S) $ and $\| f \|_{l^{1} L^{\infty}}$ is defined in Lemma \ref{lem:EstDecay}.
\label{Res:TechLem}
\end{res}
Let $H(z,\bar{z}) := \frac{n-1}{2} G(z,\bar{z}) - F(z,\bar{z})$. Proceeding similarly as in (\ref{Eqn:EstG}) we get
$H(z,\bar{z}) \approx |x|^{1_{2}^{*}} g(|z|) $. \\
\\
Let $m(u) := 2(t^{2} + r^{2}) u_{t} + 4 t r \p_{r} u + 2(n-1) t u $. We have

\begin{equation}
\begin{array}{ll}
\Re
\left[
\overline{m(u)}
\left( \Box u + u + |u|^{\frac{4}{n-2}} u g(|u|)  \right)
\right]
& = \p_{t} \left( t^{2} Q_{0}(u) +(t^{2} + r^{2} ) F(u,\bar{u})  -(n-1) \nabla \cdot (x |u|^{2}) \right) \\
& + \nabla \cdot \left( - m(u) \nabla u + 2 t x \left( e(u) - 2 | \p_{t} u|^{2} \right)  \right) + 4 t \left( H (u,\bar{u}) - |u|^{2} \right)
\end{array}
\label{Eqn:InversIdent}
\end{equation}
Here $t^{2} Q_{0} (u,\bar{u}) := | t \p_{t} u + r \p_{r} u +(n-1) u |^{2} + |r \p_{t} u + t \p_{r} u|^{2} + (t^{2} +r^{2}) ( |u_{\theta}^{2}| + |u|^{2}) $,
$ F(u,\bar{u}) := 2 \int_{0}^{|u|} s^{1_{2}^{*} -1} g(s) $, $ e(u) := \frac{1}{2} |\p_{t}u|^{2} + \frac{1}{2} |\nabla u|^{2} + \frac{1}{2} |u|^{2} + F( u,\bar{u}) $, and $H(u,\bar{u}) := \frac{n-1}{2} G(u,\bar{u}) - F(u,\bar{u})$. Proceeding similarly as in (\ref{Eqn:EstG}) we get $H(z,\bar{z}) \approx |z|^{1_{2}^{*}} g(|z|)$. \\
Let $K_{S}^{T} := \{ (t,x): \,  t \in [S, T] \, \text{and} \,  t \geq \alpha^{-1} |x|  \} $.  In the sequel $a' \lesssim b' $ means that there exists a constant $C:= C(E,\alpha)$ such that $a' \leq C b'$. Integrating (\ref{Eqn:InversIdent}) over $K_{S}^{T}$, we get from the Green formula

\begin{equation}
\begin{array}{ll}
\left[ \int_{ t \geq \alpha^{-1}|x| } t^{2} Q_{0}(u,\bar{u}) + (t^{2} + r^{2}) F(u,\bar{u}) \; dx  \right]_{S}^{T}
- \frac{1}{\sqrt{1 + \alpha^{-2}}} X +  \frac{1}{\sqrt{1 + \alpha^{-2}}} Y + 4 \int_{K_{S}^{T}} t \left( H(u,\bar{u}) - |u|^{2} \right) \; dx \; dt  & = 0  \cdot
\end{array}
\nonumber
\end{equation}
Here

\begin{equation}
\begin{array}{l}
X  := \int_{\alpha S \leq |x|  \leq \alpha T}
\left[
\begin{array}{l}
\left| \alpha^{-1} r \p_{t} u + r \p_{r} u  + (n-1) u \right|^{2}
+ \left| r \p_{t} u + \alpha^{-1} r \p_{r} u \right|^{2}  \\
+ (\alpha^{-2} + 1 ) r^{2} \left( u_{\theta}^{2} + |u|^{2} \right)
\end{array}
\right] ( \alpha^{-1} r,x )
\; dx \; \text{and}
\end{array}
\nonumber
\end{equation}

\begin{equation}
\begin{array}{ll}
Y & := \int_{\alpha  \leq |x| \leq \alpha T}
 \left( \frac{x}{r} \cdot \left( - m (u) \nabla u
+ 2  \alpha^{-1} r x \left( e(u) - 2 |\p_{t} u|^{2} \right) \right) \right) (\alpha^{-1} r,x) \; dx
\end{array}
\nonumber
\end{equation}
We have $|X| + |Y| \lesssim \int_{\alpha S \leq |x| \leq \alpha T} r^{2} (Z_{1} + Z_{2}) \; dx $ with
$ Z_{1} :=  e (u) (\alpha^{-1} r, x) $ and  $Z_{2} := \frac{|u|^{2}}{r^{2}} \left( \alpha^{-1} r, x \right) $. \\
\\
We first estimate $\int_{\alpha S \leq |x| \leq \alpha T} r^{2} Z_{1} \;  dx $. Recall the energy identity $\p_{t} e(u) - \Re \left( \nabla \cdot \left( \nabla u \overline{\partial_{t} u }  \right) \right) = 0  $. Let $\bar{S},\bar{T}$ be two arbitrary numbers such that $0 \leq \bar{S} < \bar{T} \leq a$. Integrating this identity over the cone $K_{\bar{S}}^{\bar{T}} $, we get

\begin{equation}
\begin{array}{ll}
 \left[ \int_{ t \geq \alpha^{-1}|x|  } e(u) (t, x ) \; dx \right]_{\bar{S}}^{\bar{T}} & =
\left( 1 + \alpha^{-2} \right)^{-1}  \int_{\bar{S} \leq |x| \leq \bar{T}}  \left( e(u) + \alpha^{-1} \Re  \left( \frac{ \overline{ \partial_{t} u} \nabla u \cdot x }{r} \right)
\right) \left( \alpha^{-1} r, x \right)  \; dx
\end{array}
\nonumber
\end{equation}
Hence, using also the Young inequality $ \left| \frac{ \overline{\partial_{t} u} \nabla u \cdot x}{|x|} \right|  \leq \frac{1}{2} \left( |\p_{t} u|^{2} + |\nabla u|^{2} \right)$
and the conservation of the energy we see that $ \int_{\alpha S \leq |x| \leq \alpha T} r^{2} Z_{1} \;  dx  \lesssim T^{2} $. \\
We then estimate $\int_{\alpha S \leq |x| \leq \alpha T} r^{2} Z_{2} \;  dx $. From the Hardy-type inequality (see e.g \cite{shat}) below

\begin{equation}
\begin{array}{ll}
\int_{|x| \leq \alpha T} \frac{|f(x)|^{2}}{|x|^{2}} \; dx & \lesssim \int_{|x| \leq \alpha T} |\p_{r} f|^{2} \; dx  + \left( \int_{|x| \leq \alpha T}  |f(x)|^{1_{2}^{*}} \; dx \right)^{\frac{2}{1_{2}^{*}}},
\end{array}
\nonumber
\end{equation}
we see that  $ \int_{\alpha S \leq |x| \leq \alpha T} r^{2} Z_{2} \;  dx \lesssim T^{2} $. Hence $ |X| + |Y| \lesssim T^{2} $.  \\
The triangle inequality, the Hardy inequality and (\ref{Eqn:MorawetzEst}) show that

\begin{equation}
\begin{array}{l}
\int \limits_{\substack{t \in [0, a] \\ t \geq \alpha^{-1} |x| }} \frac{|H(u,\bar{u}) - u^{2}|}{t} \; dx \, dt  \lesssim
\int \limits_{\substack{t \in [0, a] \\ t \geq \alpha^{-1} |x| }}  t \frac{|u|^{2}}{|x|^{2}} \; dx  \; dt
+ \int \limits_{\substack{t \in [0, a] \\ t \geq \alpha^{-1} |x| }}   \frac{|u|^{1_{2}^{*}} g(|u|)}{|x|} \; dx \; dt \lesssim 1 \cdot
\end{array}
\nonumber
\end{equation}
Hence, by applying Result \ref{Res:TechLem} we get

\begin{equation}
\begin{array}{ll}
\left\| \int_{t \geq \alpha^{-1} |x|}  Q_{0}(u) \; dx \right\|_{l^{1}L^{\infty}} \lesssim 1
\end{array}
\label{Eqn:EstQZero}
\end{equation}
Integrating the equality below over the truncated cone $K_{S}^{T}$

\begin{equation}
\begin{array}{ll}
\p_{t} |u|^{2} & = 2 \Re \left( u \left(  \overline{\p_{t} u + r \p_{r} u + (n-1) u  } \right) \right) \\
& - \nabla \cdot \left( |u|^{2} x \right) - (n-2)|u|^{2},
\end{array}
\nonumber
\end{equation}
we get, after applying the Young inequality $AB \leq \frac{A^{2}}{2} + \frac{B^{2}}{2} $ to
$A:= |u| \sqrt{t}$ and $ B := \frac{|r \partial_{t} u + t \partial_{r} u + (n-1) u|}{\sqrt{t}}$

\begin{equation}
\begin{array}{ll}
\left[ \int_{ t \geq \alpha^{-1} |x| } |u(t,x)|^{2} \; dx \right]_{S}^{T} & \lesssim
\int_{K_{S}^{T}} \frac{ | r \p_{t} u  + t \p_{r} u + (n-1) u|^{2} }{t} \; dx \; dt +
\int_{ \alpha S \leq |x| \leq \alpha T} |x| |u|^{2} (\alpha^{-1} r, x ) \; dx \; dt \cdot
\end{array}
\nonumber
\end{equation}
% COMM
%We may write $\int_{K_{S}^{T}} t |u|^{2} \; dx \; dt \lesssim \int_{S}^{T} t^{2} g(t) \; dt $ with
% $g(t) := \int_{ t \geq \alpha^{-1} |x|} \frac{|u|^{2}}{|x|} \; dx $. By placing $|u|^{2}$ in $ L^{\frac{1_{2}^{*}}{2}}$  we see from Holder inequality
% that $\| g \|_{ L^{1} \left( [ 0,\epsilon_{0}) \right)} \lesssim 1 $.
Hence we see from (\ref{Eqn:EstQZero}) and the above estimate  $ \int_{\alpha S \leq |x| \leq \alpha T} r^{2} Z_{2} \;  dx \lesssim T^{2} $ that

\begin{equation}
\begin{array}{ll}
\left\| \int_{ t \geq \alpha^{-1} |x|} \frac{|u|^{2}}{t^{2}} \; dx \right\|_{l^{1} L^{\infty}} & \lesssim 1 \cdot
\end{array}
\label{Eqn:MassHardyTime}
\end{equation}
From the Hardy-type inequality (see e.g \cite{nakcpde})

\begin{equation}
\begin{array}{ll}
\int_{ t \geq \alpha^{-1} |x| } \frac{|u|^{2}}{|x|^{2}} \; dx & \lesssim  \int_{ t \geq \alpha^{-1} |x|} \frac{|u|^{2}}{t^{2}} + \left( \frac{t-|x|}{t} \right)^{2}
|\p_{r} u|^{2} \; dx,
\end{array}
\nonumber
\end{equation}
the inequality

\begin{equation}
\begin{array}{ll}
\left( \frac{t -|x|}{t} \right)^{2} |\p_{r} u|^{2} & \lesssim \frac{1}{t^{2}}
\left( | t \p_{t} u + |x| \p_{r} u + (n-1) u | ^{2} +
\left| |x| \p_{t} u + t \p_{r} u \right|^{2} + \frac{|u|^{2}}{t^{2}}
\right)
\end{array}
\nonumber
\end{equation}
coming from the equality $(t^{2} - |x|^{2}) \p_{r} u := t \left( |x| \p_{t} u + t \p_{r} u \right) -
|x| \left( t \p_{t} u  + |x| \p_{r} u \right) $ and elementary estimates, from (\ref{Eqn:MassHardyTime}), and from the following Hardy-type inequality
(see e.g \cite{nakcpde})

\begin{equation}
\begin{array}{ll}
\int_{r < R  } |f|^{1_{2}^{*}} \; dx & \lesssim \| \nabla f \|^{1_{2}^{*} - 2} _{L^{2}} \int_{r < R}  |f_{\theta}|^{2} + \frac{|f|^{2}}{r^{2}} \; dx
\end{array}
\nonumber
\end{equation}
we get (\ref{Eqn:MorawEstCone}).

\end{proof}

\subsection{The proof}

In this subsection we prove Proposition \ref{Prop:BoundLong}. We use an argument in \cite{nakimrn}. Divide $K$ into subintervals $(K_{j})_{ 1  \leq j \leq l  }$ such that
$\| u \|_{L_{t}^{\frac{2(n+1)}{n-2}} L_{x}^{\frac{2(n+1)}{n-2}} (K_{j})} = c g^{-\frac{1}{1_{2}^{*}-1}}(M)$ for  $1 \leq j < l$  and
$\| u \|_{L_{t}^{\frac{2(n+1)}{n-1}} L_{x}^{\frac{2(n+1)}{n-1}}(K_{l})} \leq c g^{-\frac{1}{1_{2}^{*}-1}}(M) $, with $c$ constant defined
in Lemma \ref{lem:Concentration}. In view of (\ref{Eqn:BoundStrichLong}) and the triangle inequality, we may assume WLOG that
$\| u \|_{L_{t}^{\frac{2(n+1)}{n-1}} L_{x}^{\frac{2(n+1)}{n-1}}(K_{j})} = cg^{-\frac{1}{1_{2}^{*}-1}}(M) $ for all $1 \leq  j \leq l$ and it suffices to prove that there
exists $C \gg 1$ such that

\begin{equation}
\begin{array}{ll}
l & \lesssim C^{C g^{b_{n}^{+}}(M)}
\end{array}
\label{Eqn:Estl}
\end{equation}
In the sequel we say that $\bar{C}$ (resp. $\bar{c}$) is a constant associated to $a' \lesssim b' $ (resp. $a' \ll b'$) if the constant $\bar{C} > 0$ (resp. the constant
$ 0 <  \bar{c} \ll 1$)
satisfies $a' \leq \bar{C} b' $ (resp. $a' \leq \bar{c} b'$). In the sequel we choose constants associated to expressions of the type
$ a \lesssim b $  or the type  $ a \ll b $ in (\ref{Eqn:DyadExtr2}) and (\ref{Eqn:ChoicePQR}) in such a way that all the estimates and statements below are true. Recall the following result:

\begin{res}{(straightforward modification of Lemma $4.2$ in \cite{nakimrn})}
Let $N \in \{1,2,... \}$. Let $S  \subset \mathbb{R}^{n+1} $ be a set. There exists a constant $\bar{C}  \gg 1$ such that if
$\card{(S)} \geq  \bar{C}^{\bar{C} N } $ then one can find at least $N$ distinct points $ z_{1},z_{2},...,z_{N} \in  S $ such that for all
for all $j \in \{ 2,...,N \}$

\begin{equation}
\begin{array}{l}
| z_{j} - z_{N} | \ll | z_{j-1} - z_{N}|
\end{array}
\label{Eqn:DyadExtr}
\end{equation}

\label{Res:DyadicExt}
\end{res}
By Lemma \ref{lem:Concentration} there exist $x_{j} \in \mathbb{R}^{n}$, $J_{j} \subset K_{j}$ and
$0 < R_{j} \lesssim g^{\frac{(n+1)(n^{2}-3n +6)}{8}} (M) |J_{j}| $ such that for all $t \in J_{j}$

\begin{equation}
\begin{array}{ll}
\int_{|x-x_{j}| \leq R_{j}} |u(t,x)|^{1_{2}^{*}} \; dx  & \gtrsim \eta^{1_{2}^{*}(n-1)} \cdot
\end{array}
\label{Eqn:ConcJj}
\end{equation}
Let $t_{j} \in J_{j}$. \\
We may assume WLOG that $l \gg 1$. Hence we can choose $N \in \{1,2,... \}$ such that
$ l \geq \bar{C}^{\bar{C}(N+1)} $ and $ l \approx \bar{C}^{\bar{C}(N+1)} $. From Result \ref{Res:DyadicExt} there exist $N$ distinct
points $y_{1} :=(t_{1},x_{1})$,..., and $y_{N}:=(t_{N},c_{N})$ such that for all $j \in \{ 2,...,N \}$

\begin{equation}
\begin{array}{ll}
| y_{j} - y_{N} | & \ll |y_{j-1} - y_{N}|
\end{array}
\label{Eqn:DyadExtr2}
\end{equation}
Define

\begin{equation}
\begin{array}{ll}
S' := \left\{ y_{1},...,y_{N} \right\}, // 
P := \{ j \in S': \; |y_{j}  - y_{N}| \lesssim  R_{j}  \}, \\
Q := \{ j \in S'/P: \; |x_{j} - x_{N}| \lesssim  |t_{j} - t_{N}| \}, \, \text{and} \\ 
R := S' / (P \cup Q) \cdot
\end{array}
\label{Eqn:ChoicePQR}
\end{equation}
We first estimate $\card{(P)}$. We get from (\ref{Eqn:MorawetzEst}) and (\ref{Eqn:ConcJj})

\begin{equation}
\begin{array}{lll}
1 & \gtrsim \int_{[0,\epsilon_{0})} \int \frac{|u(t,x)|^{1_{2}^{*}}}{| x - x_{N} |} \; dx \; dt
  \gtrsim  \sum \limits_{j \in P} \int_{J_{j}} \int \frac{|u(t,x)|^{1_{2}^{*}}}{| x - x_{N} |} \; dx \; dt \\
  & \gtrsim   \sum \limits_{j \in P}   \frac{\eta^{1_{2}^{*} (n-1) } |J_{j}| }{R_{j}} \\
  & \gtrsim \card{(P)} \frac{ \eta^{1_{2}^{*} (n-1)}} {g^{\frac{(n+1)(n^{2}-3n + 6)}{8}+}(M)}
\end{array}
\nonumber
\end{equation}
We then estimate $\card{(Q)}$. Let $ j \in Q $. Let  $ B_{j} := \{ x \in \mathbb{R}^{n}: \; |x - x_{j}| \leq R_{j} \}$.
Let $ x \in  B_{j}$. Then $| t_{j} - t_{N}| \gtrsim R_{j} + |x_{j} - x_{N}| \gtrsim | x - x_{N}| $. Hence we see from
the application of (\ref{Eqn:MorawEstCone}) with $\alpha \gg 1$ large enough  that

\begin{equation}
\begin{array}{ll}
1 \gtrsim \card{(Q)} \eta^{1_{2}^{*}(n-1)} \cdot
\end{array}
\nonumber
\end{equation}
It remains to estimate  $\card{(R)}$. Let $ j \in R $. We define
$\tilde{B}_{j} := \left\{ x \in \mathbb{R}^{n}: \; |x - x_{j} | \leq R_{j} + |t_{j} - t_{N}| \right\}$. Let $k \in R$ such that
$j \neq k$. Observe that  $ \tilde{B}_{j} \cap \tilde{B}_{k} = \emptyset$. Define

\begin{equation}
\begin{array}{ll}
e(u) & := \frac{1}{2} |\p_{t} u|^{2}  + \frac{1}{2} |\nabla u |^{2} + \frac{1}{2} |u|^{2} + F(u,\bar{u})
\end{array}
\nonumber
\end{equation}
If $ t_{N} \geq t_{j} $ then by integrating the well-known energy identity
$ \Re \left( \overline{\partial_{t} u }  \left( \p_{tt} u - \triangle u + u + |u|^{1_{2}^{*}-2} u g(|u|) \right) \right) = \partial_{t} e(u) - \Re \left( \nabla \cdot \left( \overline{\partial_{t} u} \nabla u \right) \right) $ on the forward cone
$ \bar{K}_{j,f} := \left\{ (t,x): \;  t_{N} \geq  t \geq t_{j}, \; t > t_{j} + |x -x_{j}| - R_{j} \right\} $ we get

\begin{equation}
\begin{array}{ll}
\int_{\tilde{B}_{j}} e(u(t_{N})) \; dx & \geq \int_{B_{j}} e(u(t_{j})) \; dx .
\end{array}
\label{Eqn:NrjEst}
\end{equation}
If $ t_{N} \leq t_{j} $ then integrating the same identity in the backward cone
$ \bar{K}_{j,b} := \left\{ (t,x): \;  t_{j} \geq  t \geq t_{N}, \; t - t_{j} < R_{j} - |x -x_{j}| \right\} $ we get
(\ref{Eqn:NrjEst}). Hence

\begin{equation}
\begin{array}{ll}
E  = \int_{\mathbb{R}^{n}} e(u(t_{N})) \; dx & \geq \sum \limits_{j \in R} \int_{\tilde{B}_{j}} e(u(t_{N})) \; dx  \\
& \geq \sum  \limits_{j \in R} \int_{B_{j}} e(u(t_{j})) \; dx \\
& \geq \card{(R)} \eta^{1_{2}^{*}(n-1)}
\end{array}
\nonumber
\end{equation}
Hence $ N  \lesssim g^{b_{n}+}(M) $ and (\ref{Eqn:Estl}) holds.

\section{Appendixes}

Unless otherwise stated, let

\begin{itemize}

\item  $c > 0$  be a constant allowed to change from one line to another one
and that is small enough
\item $C > 0$ be a constant allowed to change from one line to another one
\item $\theta \in (0,1)$  be a constant allowed to change from one line to another one
\item  $x+$ (resp. $x-$) be a number allowed to changed from one line to another one and slightly larger (resp. slightly smaller ) than $x$
\item $x++$ be a number allowed to changed from one line to another one and slightly larger than $x+$
\item $x+++$ be a number  allowed to changed from one line to another one and slightly larger than $x++$
\item $\infty-$ be a finite constant allowed to changed from one line to another one and very large
\end{itemize}
such that all the estimates (and statements) in Appendix $A$ and in Appendix $B$ are true. We recommend that the reader plots all the points
$ \left( \frac{1}{a}, \frac{1}{b} \right)$ wherever $L_{t}^{a} L_{x}^{b}$ appears
on the coordinate plane $Oxy$ with $Ox$ (resp. $Oy$) representing the $x-$axis (resp. the $y-$ axis ).

\subsection{Appendix $A$ }

In this appendix we prove Lemma \ref{lem:NonlinearControl}.\\
\\
Let $r$ be such that $\frac{n-2}{2(n+1)} + \frac{n}{r} = \frac{n}{2} - \frac{1}{2} $ . \\
\\
Assume that $n=5$. Observe that $ \left( \frac{2n}{n-2}, \frac{2(n+1)}{n-2}, \frac{2(n+1)}{n-1}, \frac{2n}{n-3} \right) =
\left( \frac{10}{3} ,4, 3, 5 \right)$. \\
\\
Assume also that $ 1<k<2$. We see from (\ref{Eqn:LeibnCompos}) and $g^{'}(|f|) |f| + g \left( |f| \right) \lesssim 1 + |f|^{0+}$ that

\begin{equation}
\begin{array}{l}
\left\| \langle D \rangle^{k - 1} \left( |u|^{1_{2}^{*}-2} u g(|u|) \right) \right\|_{ L_{t}^{1} L_{x}^{2} (I)} \\
\lesssim
\left\| \langle D \rangle^{k-1} u \right\|_{L_{t}^{2} L_{x}^{5} (I) }
\left(
\| u \|^{\frac{4}{3}}_{L_{t}^{\frac{8}{3}} L_{x}^{\frac{40}{9}} (I) } +
\| u \|^{\frac{4}{3}}_{L_{t}^{\frac{8}{3}} L_{x}^{\frac{40}{9}+} (I) } \| u \|^{C}_{L_{t}^{\infty} L_{x}^{1_{2}^{*}} (I)}
\right) \\
\lesssim  \text{R.H.S of (\ref{Eqn:PropNonLin1})}
\end{array}
\label{Eqn:App0}
\end{equation}
Here we used $H^{\tilde{k}} \hookrightarrow L^{1_{2}^{*}}$,

\begin{equation}
\begin{array}{ll}
\| u \|_{L_{t}^{\frac{8}{3}} L_{x}^{\frac{40}{9}} (I)} & \lesssim
\| u \|^{\theta}_{L_{t}^{2} L_{x}^{5} (I)}
\| u \|^{1- \theta}_{L_{t}^{4} L_{x}^{4} (I)}, \\
\| u \|_{L_{t}^{\frac{8}{3}} L_{x}^{\frac{40}{9}+} (I)} & \lesssim \| u \|^{\theta}_{L_{t}^{\frac{8}{3}} L_{x}^{\frac{40}{9}} (I)}
\| u \|^{1- \theta}_{L_{t}^{\frac{8}{3}} L_{x}^{\frac{40}{9}++} (I)},
\end{array}
\nonumber
\end{equation}
followed by the embedding $H^{0+,\tilde{r}} \hookrightarrow L^{\frac{40}{9}++} $ ( with $\tilde{r}$ such that
$ \frac{1}{\frac{8}{3}++} + \frac{5}{\tilde{r}} = \frac{5}{2} - 1 $ ) and

\begin{equation}
\begin{array}{ll}
\left\| \langle D \rangle^{0+} u \right\|_{L_{t}^{\frac{8}{3}} L_{x}^{\tilde{r}} (I)} & \lesssim
\left\| \langle D \rangle^{0+} u \right\|^{\theta}_{L_{t}^{2} L_{x}^{5}(I)}
\left\|  \langle D \rangle^{0+} u \right\|^{1- \theta}_{L_{t}^{\infty} L_{x}^{1_{2}^{*}}(I)} \\
& \lesssim \left\| \langle D \rangle^{\tilde{k}-1} u \right\|^{\theta}_{L_{t}^{2} L_{x}^{5}(I)}
\left\| \langle D \rangle^{\tilde{k}} u \right\|^{1 - \theta}_{L_{t}^{\infty} L_{x}^{2}(I)} \cdot
\end{array}
\label{Eqn:App1}
\end{equation}
In the expression above we used the embedding $ \left\|  \langle D \rangle^{0+} u \right\|_{L_{t}^{\infty} L_{x}^{1_{2}^{*}}(I)} \lesssim
\left\|  \langle D \rangle^{\tilde{k}} u \right\|_{L_{t}^{\infty} L_{x}^{2}(I)} $. \\
Assume now that $ 2 \leq k < \frac{7}{3} $. Then by using the Plancherel theorem and by expanding the gradient we see that
$ \left\| \langle D \rangle^{k - 1} \left( |u|^{1_{2}^{*}-2} u g(|u|) \right) \right\|_{ L_{t}^{1} L_{x}^{2} (I)}$ is bounded by terms of the form

\begin{equation}
\begin{array}{ll}
Y_{0} & := \left\| |u|^{1_{2}^{*}-2} u g(|u|) \right\|_{ L_{t}^{1} L_{x}^{2} (I)}, \; \\
Y_{1} & := \left\| \langle D \rangle^{k-2} \left( \nabla u G(u,\bar{u})  g(|u|) \right) \right\|_{ L_{t}^{1} L_{x}^{2} (I)}, \\
Y_{2} & :=  \left\| \langle D \rangle^{k-2} \left( \nabla u  G(u,\bar{u}) g^{'}(|u|) |u| \right) \right\|_{ L_{t}^{1} L_{x}^{2} (I)}, \; \text{and}
\end{array}
\nonumber
\end{equation}
terms that are similar to $Y_{1}$ and $Y_{2}$. Here $G$ is a
$\mathcal{C}^{1}(\mathbb{R}^{2})-$ function such that $|G(f,\bar{f})| \approx |f|^{\frac{4}{3}}$. We have

\begin{equation}
\begin{array}{ll}
Y_{0} & \lesssim  \|  u \|_{L_{t}^{2} L_{x}^{5} (I) }
\left(
\| u \|^{\frac{4}{3}}_{L_{t}^{\frac{8}{3}} L_{x}^{\frac{40}{9}} (I) } +
\| u \|^{\frac{4}{3}}_{L_{t}^{\frac{8}{3}} L_{x}^{\frac{40}{9}+} (I) } \| u \|^{C}_{L_{t}^{\infty} L_{x}^{1_{2}^{*}}(I)}
\right) \\
& \lesssim  \text{R.H.S of (\ref{Eqn:PropNonLin1})} \cdot
\end{array}
\nonumber
\end{equation}
We then only estimate $Y_{1}$ since $Y_{2}$ is estimated similarly. We see from (\ref{Eqn:LeibnProd}) that $ Y_{1} \lesssim A + B $ with

\begin{equation}
\begin{array}{ll}
A & :=  \left\| \langle D \rangle^{k-1} u \right\|_{L_{t}^{2} L_{x}^{5} (I)}
\left\| G(u,\bar{u})  g(|u|) \right\|_{L_{t}^{2} L_{x}^{\frac{10}{3}} (I)}, \; \text{and} \\
B & :=  \left\| \langle D \rangle^{k-2} \left( G(u,\bar{u}) g(|u|) \right) \right\|_{L_{t}^{\frac{8}{5}} L_{x}^{\frac{40}{11}} (I)}
\| \nabla u \|_{L_{t}^{\frac{8}{3}} L_{x}^{\frac{40}{9}} (I)} \cdot
\end{array}
\nonumber
\end{equation}
We have

\begin{equation}
\begin{array}{ll}
A & \lesssim  \left\| \langle D \rangle^{k-1} u \right\|_{L_{t}^{2} L_{x}^{5} (I)}
\left(
\| u \|^{\frac{4}{3}}_{L_{t}^{\frac{8}{3}} L_{x}^{\frac{40}{9}} (I)}
+ \| u \|^{\frac{4}{3}}_{L_{t}^{\frac{8}{3}} L_{x}^{\frac{40}{9}+} (I)} \| u \|^{C}_{L_{t}^{\infty} L_{x}^{1_{2}^{*}}(I)}
\right) \\
& \lesssim \text{R.H.S of (\ref{Eqn:PropNonLin1})} \cdot
\end{array}
\nonumber
\end{equation}
We see from (\ref{Eqn:LeibnCompos}) that

\begin{equation}
\begin{array}{ll}
B & \lesssim  \left\| \langle D \rangle^{k-2} u \right\|_{L_{t}^{2} L_{x}^{5} (I)}
\left(
\| u \|^{\frac{1}{3}}_{L_{t}^{\frac{8}{3}} L_{x}^{\frac{40}{9}} (I) }
+ \| u \|^{\frac{1}{3}}_{L_{t}^{\frac{8}{3}} L_{x}^{\frac{40}{9}+} (I) }
\| u \|^{C}_{L_{t}^{\infty} L_{x}^{1_{2}^{*}}(I)}
\right)
\| \nabla u \|_{L_{t}^{\frac{8}{3}} L_{x}^{\frac{40}{9}} (I) } \\
& \lesssim \text{R.H.S of (\ref{Eqn:PropNonLin1})},
\end{array}
\nonumber
\end{equation}
using at the last line the embedding  $H^{\tilde{k}} \hookrightarrow L^{1_{2}^{*}} $ and the embedding $H^{k-1} \hookrightarrow L^{1_{2}^{*}} $ to get

\begin{equation}
\begin{array}{l}
\| \nabla u \|_{L_{t}^{\frac{8}{3}} L_{x}^{\frac{40}{9}} (I)}  \lesssim \| \nabla u \|^{\theta}_{L_{t}^{\infty} L_{x}^{1_{2}^{*}} (I)}
\| \nabla u \|^{1 - \theta}_{L_{t}^{2} L_{x}^{5} (I)}  \lesssim X_{k} \left( I,u \right) \cdot
\end{array}
\nonumber
\end{equation}
Hence $ Y_{1} \lesssim  \text{R.H.S of (\ref{Eqn:PropNonLin1})} $. \\
\\
Assume now that $n \in \{ 3, 4 \}$.  \\
\\
We write $ k - \frac{1}{2} = m + \alpha $ with $ 0 \leq \alpha < 1 $ and $m$ nonnegative integer.  From
$ \| \langle D \rangle^{m} f \|_{L^{\frac{2(n+1)}{n-1}}} \lesssim  \| f \|_{L^{\frac{2(n+1)}{n-1}}} + \| D^{m} f \|_{L^{\frac{2(n+1)}{n-1}}} $
and $ \| D^{m} f \|_{L^{\frac{2(n+1)}{n-1}}} \lesssim  \sum \limits_{\gamma \in \mathbb{N}^{n}: |\gamma|=m }
\| \partial^{\gamma} f \|_{L^{\frac{2(n+1)}{n-1}}} $ \footnote{if $m=1$ then one can prove the estimate by using the identity
$f = \sum \limits_{j=1}^{n} R_{j} D^{-1} \partial_{x_{j}} f $ with $R_{j}$ Riesz transform defined by $ \widehat{R_{j} f} (\xi) := - i \frac{\xi_{j}}{|\xi|} \hat{f}(\xi) $
and the boundedness of the $R_{j} \, s$; if $m>1$ then one can prove the estimate by induction.} we see that

\begin{equation}
\begin{array}{ll}
\left\| \langle D \rangle^{k- \frac{1}{2}} \left( |u|^{1_{2}^{*} - 2} u g(|u|) \right) \right\|_{L_{t}^{\frac{2(n+1)}{n+3}} L_{x}^{\frac{2(n+1)}{n+3}} (I)}
& \lesssim \left\| \langle D \rangle^{\alpha} X \right\|_{L_{t}^{\frac{2(n+1)}{n+3}} L_{x}^{\frac{2(n+1)}{n+3}} (I) } \\
& + \sum \limits_{\gamma\in \mathbb{N}^{n}: |\gamma|=m } \left\| \langle D \rangle^{\alpha} \partial^{\gamma} X  \right\|_{L_{t}^{\frac{2(n+1)}{n+3}} L_{x}^{\frac{2(n+1)}{n+3}} (I)} ,
\end{array}
\label{Eqn:InterpHr}
\end{equation}
with $X  := |u|^{1_{2}^{*}-2} u g(|u|) $. We get from (\ref{Eqn:LeibnCompos}) and $g^{'}(|f|) |f| + g(|f|) \lesssim  1 + |f|^{0+}$

\begin{equation}
\begin{array}{l}
\left\| \langle D \rangle^{\alpha} X \right\|_{L_{t}^{\frac{2(n+1)}{n+3}} L_{x}^{\frac{2(n+1)}{n+3}}(I)} \\
\lesssim \left\| \langle D \rangle^{\alpha} u \right\|_{L_{t}^{\frac{2(n+1)}{n-1}} L_{x}^{\frac{2(n+1)}{n-1}}(I)}
\left( \| u \|^{1_{2}^{*}-2}_{L_{t}^{\frac{2(n+1)}{n-2}} L_{x}^{\frac{2(n+1)}{n-2}} (I)}
+ \| u \|^{1_{2}^{*} -2}_{L_{t}^{\frac{2(n+1)}{n-2}} L_{x}^{\frac{2(n+1)}{n-2}+} (I)} \| u \|^{C}_{L_{t}^{\infty} L_{x}^{1_{2}^{*}}([0,T_{l}])}
\right) \\
\lesssim \text{R.H.S of (\ref{Eqn:PropNonLin1})} \cdot
\end{array}
\label{Eqn:InitContDalphX}
\end{equation}
In the expression above we used the embedding $ H^{\tilde{k}} \hookrightarrow L^{1_{2}^{*}} $,

\begin{equation}
\begin{array}{ll}
\| u \|_{L_{t}^{\frac{2(n+1)}{n-2}} L_{x}^{\frac{2(n+1)}{n-2}+} (I)} & \lesssim \| u \|^{\theta}_{L_{t}^{\frac{2(n+1)}{n-2}} L_{x}^{\frac{2(n+1)}{n-2}} (I)}
\| u \|^{1- \theta}_{L_{t}^{\frac{2(n+1)}{n-2}} L_{x}^{\frac{2(n+1)}{n-2}++} (I)}  \\
& \lesssim (\delta')^{c} \langle X_{\tilde{k}} \left( I ,u \right)  \rangle^{C},
\end{array}
\label{Eqn:InitInterpStrichLittle}
\end{equation}
where at the last line we used the embedding $H^{\tilde{k} - \frac{1}{2},r} \hookrightarrow L^{\frac{2(n+1)}{n-2}++}$ followed by

\begin{equation}
\begin{array}{ll}
\left\| \langle D \rangle^{ \tilde{k}- \frac{1}{2}} u \right\|_{L_{t}^{\frac{2(n+1)}{n-2}} L_{x}^{r} (I)} &
\lesssim \left\| \langle D \rangle^{ \tilde{k}- \frac{1}{2}} u \right\|^{\theta}_{L_{t}^{\frac{2(n+1)}{n-1}} L_{x}^{\frac{2(n+1)}{n-1}} (I)}
\left\| \langle D \rangle^{ \tilde{k}- \frac{1}{2}} u \right\|^{1 - \theta}_{L_{t}^{\infty} L_{x}^{\frac{2n}{n-1}} (I)} \\
& \lesssim \left\| \langle D \rangle^{ \tilde{k}- \frac{1}{2}} u \right\|^{\theta}_{L_{t}^{\frac{2(n+1)}{n-1}} L_{x}^{\frac{2(n+1)}{n-1}} (I)}
\| u \|^{1 - \theta}_{L_{t}^{\infty} H^{\tilde{k}} (I)} \cdot
\end{array}
\label{Eqn:InitInterpOneHalf}
\end{equation}
In the expression above we used the embedding $H^{\tilde{k}} \hookrightarrow H^{\tilde{k} - \frac{1}{2}, \frac{2n}{n-1}}$. Hence if $m =0$ then $\left\| \langle D \rangle^{k- \frac{1}{2}} \left( |u|^{1_{2}^{*} -2} u g(|u|) \right) \right\|_{L_{t}^{\frac{2(n+1)}{n+3}} L_{x}^{\frac{2(n+1)}{n+3}} (I)} \lesssim \text{R.H.S of
(\ref{Eqn:PropNonLin1})} $.  \\
\\
\underline{Note}: we may assume WLOG that $m >0$. \\
\\
We have to estimate  $ \left\| \langle  D \rangle^{\alpha} \partial^{\gamma} X \right\|_{L_{t}^{\frac{2(n+1)}{n+3}} L_{x}^{\frac{2(n+1)}{n+3}} (I) }$.
Let $\tilde{g}(x):= \log^{\gamma} \left( \log (10 + x) \right)$. Then $g(x) = \tilde{g}(x^{2})$. \\
If $n=3$ (resp. $n=4$) and $m \leq 5$ (resp. $ m \leq 3 $) then by expanding  $\partial^{\gamma} X $ we see that it is a finite sum of terms of the form

\begin{equation}
\begin{array}{l}
 X^{'} := \partial^{\gamma^{'}} \tilde{g}(|u|^{2}) S_{\gamma^{'}} (u,\bar{u})  (\partial^{\gamma_{1}} u)^{\alpha_{1}}...( \partial^{\gamma_{p}} u)^{\alpha_{p}}
( \partial^{\bar{\gamma}_{1}} \bar{u} )^{\bar{\alpha}_{1}}... ( \partial^{\bar{\gamma}_{p'}} \bar{u} )^{\bar{\alpha}_{p'}} \cdot
\end{array}
\nonumber
\end{equation}
Here $\gamma^{'} \in \mathbb{N}$ and $S_{\gamma'}(u,\bar{u})$ is of the form $C^{'} u^{p_{1}} \bar{u}^{p_{2}}$  for some $C^{'} \in \mathbb{R}$
and some $(p_{1}, p_{2}) \in \mathbb{N}^{2}$ such that $p_{1} + p_{2} = \gamma^{'}$.
Here $ p $, $ p^{'} $, $ \gamma_{1} $,..., $\gamma_{p}$, $\bar{\gamma}_{1} $,...,$\bar{\gamma}_{p'}$, $\alpha_{1}$,..., $\alpha_{p}$, $\bar{\alpha}_{1}$,..., $\bar{\alpha}_{p'} $ satisfy the following properties: $ p \neq 0 $ or $p^{'} \neq  0$,
$ \left( \gamma_{1},..., \gamma_{p},\bar{\gamma}_{1},.., \bar{\gamma}_{p'} \right) \in \mathbb{N}^{n} \times ... \times \mathbb{N}^{n} $,
there exists  $i \in \{ 1,...,p \}$ such that $\gamma_{i} \neq (0,...,0)$ or there exists $i^{'} \in \{1,..,p' \}$ such that
$\gamma_{i'} \neq (0,...,0)$, $ \left( \alpha_{1},...\alpha_{p},
\bar{\alpha}_{1},..., \bar{\alpha}_{p'} \right) \in \mathbb{N}^{*} \times ... \times \mathbb{N}^{*} $, $ \alpha_{1} |\gamma_{1}| ... + \alpha_{p} |\gamma_{p}|
+ \bar{\alpha}_{1} |\bar{\gamma}_{1}| + ... + \bar{\alpha}_{p'} |\bar{\gamma}_{p'}| = m $ and $ \alpha_{1} + ... + \alpha_{p} +
\bar{\alpha}_{1} + ... + \bar{\alpha}_{p'}=  1_{2}^{*} - 1 $. \\
If $n=3$ (resp. $n=4$) and $ m \leq 4 $ (resp. $m  \leq 2$ ) then $X^{'}$ satisfies
either $(A)$ or $(B)$ with $(A)$, $(B)$ defined by

\begin{equation}
\begin{array}{l}
(A): \exists i \in \{ 1,...,p \} \; \text{s.t} \; \gamma_{i}=(0,...,0) \; \text{and} \; \alpha_{i} \geq  1 \\
(B): \exists i' \in \{ 1,..,p^{'} \} \; \text{s.t} \; \bar{\gamma}_{i'} = (0,...,0) \; \text{and} \; \bar{\alpha}_{i'} \geq 1
\end{array}
\nonumber
\end{equation}
If $n=3$ (resp. $n=4$) and $ m \geq 5 $ (resp. $ m \geq 3 $) then $ \partial^{\gamma} X $ can be written as a finite sum of terms of the form
$ X_{1}^{'} $ or $ X_{2}^{'} $

\begin{equation}
\begin{array}{l}
X_{1}^{'} := X^{'}, \; \text{and} \\
X_{2}^{'} := X^{''} \partial^{\tilde{\gamma}^{'}} \tilde{g}(|u|^{2}) (\partial^{\gamma_{1}} u)^{\alpha_{1}}...( \partial^{\gamma_{p}} u)^{\alpha_{p}}  (\partial^{\bar{\gamma}_{1}} \bar{u})^{\bar{\alpha}_{1}}... (\partial^{\bar{\gamma}_{p'}} \bar{u})^{\bar{\alpha}_{p'}} \cdot
\end{array}
\label{Eqn:WriteXDer}
\end{equation}
In the definition of $X_{2}^{'}$ above, $\tilde{\gamma}^{'} \in \mathbb{N}^{*}$; $p$, $ p^{'} $, $ \gamma_{1} $,..., $\gamma_{p}$, $\bar{\gamma}_{1} $,...,$\bar{\gamma}_{p'}$, $\alpha_{1}$,..., $\alpha_{p}$, $\bar{\alpha}_{1}$,..., $\bar{\alpha}_{p'} $ satisfy the same properties as those stated for the case $n=3$ (resp. $n=4$) and $m \leq 5 $ (resp. $ m \leq 3 $ ) except that
`` $\alpha_{1} |\gamma_{1}| + ... + \alpha_{p} |\gamma_{p}|+ \bar{\alpha}_{1} |\bar{\gamma}_{1}|+ ... + \bar{\alpha}_{p'} |\bar{\gamma}_{p'}| = m $ '' is replaced with
`` $ \alpha_{1} |\gamma_{1}| + ... + \alpha_{p} |\gamma_{p}|+ \bar{\alpha}_{1} |\bar{\gamma}_{1}|+ ... + \bar{\alpha}_{p'} |\bar{\gamma}_{p'}| \leq m $ ''; $X^{''}$ is
a product of terms of the form $( \partial^{\delta_{1}} u )^{\tilde{\alpha}_{1}} ...
(\partial^{\delta_{\tilde{p}}} u)^{\tilde{\alpha}_{\tilde{p}}}
( \partial^{\bar{\delta}_{1}} \bar{u} )^{\overline{\tilde{\alpha}}_{1}} ...
(\partial^{\bar{\delta}_{p'}} \bar{u})^{\overline{\tilde{\alpha}}_{\tilde{p}^{'}}} $ with $\tilde{p} \neq 0$ or
$\tilde{p}^{'} \neq 0$, and $(\tilde{\alpha}_{1},..., \tilde{\alpha}_{\tilde{p}},
\overline{\tilde{\alpha}}_{1},..., \overline{\tilde{\alpha}}_{\tilde{p}^{'}} ) \in \mathbb{N}^{*} \times ... \times \mathbb{N}^{*} $. These terms contain a small number of derivatives compare with $k$. More precisely if $ \delta_{\max} := \max \left( |\delta_{1}|,..., |\delta_{\tilde{p}}|, |\bar{\delta}_{1}|,..., |\bar{\delta}_{\tilde{p}^{'}}| \right)  $ then

\begin{equation}
n=3: \; \delta_{max} \leq m - 5; \; n=4: \; \delta_{max} \leq m - 3:
\label{Eqn:Estm}
\end{equation}
this fact will allow to use embeddings of the type $ H^{k} \hookrightarrow L^{\infty} $.
The proof of (\ref{Eqn:WriteXDer}) follows from an induction process applied to $m$, taking into account that for $n=3$ (resp. $n=4$)  and $\gamma$ such that $|\gamma|= 5$ ( resp. $|\gamma| = 3 $ ) $ \partial^{\gamma} X $ can be written as a finite sum of terms of the form $ X^{'}$.\\
\\
Assume that $n=3$ (resp. $n=4$) and $ m \leq 4 $ (resp. $ m \leq 2 $). We may assume WLOG that $(A)$ holds. Reordering the $\gamma_{i}$ s if necessary, we may assume WLOG
that $|\gamma_{1}|  := \min \left( |\gamma_{1}|,...,|\gamma_{p}| \right)$. Hence $|\gamma_{1}|=0 $ and we see  from (\ref{Eqn:LeibnProd}), the boundedness of the Riesz transforms
and Fact $1$ (see Section \ref{Sec:JensenIneq}), and $\overline{\langle D \rangle^{\alpha'} f} = \langle D \rangle^{\alpha'} \bar{f}$ for $\alpha^{'} \in \mathbb{R}$, that $ \| \langle D \rangle^{\alpha} \partial^{\gamma} X \|_{L_{t}^{\frac{2(n+1)}{n+3}} L_{x}^{\frac{2(n+1)}{n+3}} (I)} $ is bounded by a finite sum of terms of the form:

\begin{equation}
\begin{array}{ll}
Y & := \left\| \langle D \rangle^{\alpha} \left( u \partial^{\gamma^{'}} \tilde{g}(|u|^{2}) S_{\gamma^{'}} (u,\bar{u}) \right) \right\|_{L_{t}^{Q}
L_{x}^{R} (I)}
\prod \limits_{  s \in [1,...,p+p'] }  \left\| \langle D \rangle^{|\gamma_{s}^{'}|} u \right\|^{\alpha^{'}_{s}}_{L_{t}^{Q_{s}} L_{x}^{R_{s}} (I)},  \; \text{or} \\
Z^{'}_{j} & :=   \|  u \partial^{\gamma^{'}} \tilde{g}(|u|^{2}) S_{\gamma^{'}} (u,\bar{u}) \|_{L_{t}^{Q'} L_{x}^{R'} (I)}
\left\|  \langle D \rangle^{|\gamma^{'}_{j}|} u  \right\|^{\alpha^{'}_{j}-1}_{L_{t}^{Q^{'}_{j,1}} L_{x}^{R^{'}_{j,1}}(I)}
\left\|  \langle D \rangle^{|\gamma^{'}_{j}| + \alpha} u  \right\|_{L_{t}^{Q^{'}_{j,2}} L_{x}^{R^{'}_{j,2}}(I)}
\\
&  \prod \limits_{ \substack{ s \in [1,..,p+p']  \\ s \neq j } } \left\| \langle D \rangle^{|\gamma^{'}_{s}|} u  \right\|^{\alpha_{s}^{'}}_{
L_{t}^{Q^{'}_{s}} L_{x}^{R^{'}_{s}} (I)} \cdot
\end{array}
\label{Eqn:DefYZPrimej}
\end{equation}
Here $j \in \{ 1,...,p + p^{'} \}$. In the expression above we define $\gamma^{'}_{s}$ (resp. $\alpha_{s}^{'}$) in the following fashion:
$\gamma^{'}_{1} := (0,...,0)$, $ \alpha^{'}_{1} := \alpha_{1} - 1  $,
$ 2 \leq s \leq p $: $ \left( \gamma^{'}_{s}, \alpha_{s}^{'} \right)  := \left( \gamma_{s}, \alpha_{s} \right) $ ; $ p + p^{'} \geq s \geq  p+1 $:
$  \left( \gamma^{'}_{s}, \alpha_{s}^{'} \right) := \left( \bar{\gamma}_{s-p}, \bar{\alpha}_{s-p} \right) $. Here $Q$, $R$, $Q_{s}$, $R_{s}$, $Q^{'}$, $R^{'}$, $Q^{'}_{j,1}$, $R^{'}_{j,1}$, $Q^{'}_{j,2}$, $R^{'}_{j,2}$, $Q^{'}_{s}$, $R^{'}_{s}$ are numbers to be chosen that satisfy the following constraints: $R \neq \infty$,
$R_{s} \neq \infty$, $R' \neq \infty$, $R_{j,1} \neq \infty$, $R^{'}_{j,2} \neq \infty$, $R^{'}_{s} \neq \infty$, $ \left( \frac{1}{Q}, \frac{1}{R} \right) + \sum \limits_{ s \in [1,..,p+p^{'}] }   \alpha^{'}_{s}  \left( \frac{1}{Q_{s}}, \frac{1}{R_{s}} \right) = \frac{n+3}{2(n+1)} (1,1) $, and
$\left( \frac{1}{Q'} + \frac{\alpha_{j}^{'} -1}{Q_{j,1}^{'}} + \frac{1}{Q_{j,2}^{'}}, \frac{1}{R'} + \frac{\alpha_{j}^{'} -1}{R_{j,1}^{'}} + \frac{1}{R_{j,2}^{'}} \right)
+ \sum \limits_{ \substack{ s \in [1,..,p+p^{'}] \\ s \neq j,}} \alpha^{'}_{s} \left( \frac{1}{Q^{'}_{s}}, \frac{1}{R^{'}_{s}} \right)
=  \frac{n+3}{2(n+1)} (1,1) $. Observe that  $ \left( \gamma^{'},\gamma^{'}_{1},..., \gamma^{'}_{p+p^{'}}  \right) $ and $ \left( \alpha^{'}_{1},...\alpha^{'}_{p+p^{'}} \right) $ satisfy the following properties: $ \alpha^{'}_{1} |\gamma^{'}_{1}| ... + \alpha^{'}_{p+p^{'}} |\gamma^{'}_{p+p^{'}}| \leq m  $
and $ \alpha^{'}_{1} + ... + \alpha^{'}_{p+p^{'}} = 1_{2}^{*} - 2 $. \\
By collecting the $ \gamma^{'}_{s} $ that have the same length $|.|$ we see that there exists $ 1 \leq t \leq p+p^{'} $ and that there exist numbers $ n_{i} $, $ i \in \{ 1,..,t \} $, such that the following properties hold:  $ n_{i} \in \{ 1,...,p+p' \}$,
$ | \gamma^{'}_{n_{1}} | < | \gamma^{'}_{n_{2}} | < ... < | \gamma^{'}_{n_{t}} | $, and for all $s$ there exists $i$ such that
$ |\gamma^{'}_{s}| = |\gamma^{'}_{n_{i}}|$. Let $ \bar{\gamma}_{i} := |\gamma^{'}_{n_{i}} | $ and $ \bar{\alpha}_{i} := \sum \limits_{s: |\gamma^{'}_{s}| = |\gamma^{'}_{n_{i}}| } \alpha^{'}_{s} $ \footnote{The reader should keep in mind that the value of $t$, that of $\bar{\alpha}_{i}$, and that of $\bar{\gamma}_{i}$ depend on the expression we estimate. For example if we estimate $Y$, then their value depends on $Y$. If we estimate $Z^{'}_{j}$  then their value depends on $Z_{j}^{'}$. Nevertheless, in the sequel, in order to avoid too much notation, we use the same $t$, the same $\bar{\alpha}_{i}$, and the same $\bar{\gamma}_{i}$ for all the expressions that we estimate.}. Then

\begin{equation}
\begin{array}{ll}
Y & \lesssim \left\| \langle  D \rangle^{\alpha} \left( u \partial^{\gamma^{'}} \tilde{g}(|u|^{2}) S_{\gamma^{'}} (u,\bar{u}) \right)  \right\|
_{ L_{t}^{Q} L_{x}^{R} (I)} \prod \limits_{ i \in [1,...,t-1]}  \left\| \langle D \rangle^{\bar{\gamma}_{i}} u \right\|^{\bar{\alpha}_{i}}_{L_{t}^{\bar{Q}_{i}}
L_{x}^{\bar{R}_{i}}(I)} \\
& \left\| \langle D \rangle^{\bar{\gamma}_{t}} u \right\|^{\bar{\alpha}_{t} - 1}_{L_{t}^{\bar{Q}_{t,1}} L_{x}^{\bar{R}_{t,1}}(I)}
  \left\| \langle D \rangle^{\bar{\gamma}_{t}} u \right\|_{L_{t}^{\bar{Q}_{t,2}} L_{x}^{\bar{R}_{t,2}}(I)} , \; \text{and}
\end{array}
\label{Eqn:EstY}
\end{equation}

\begin{equation}
\begin{array}{ll}
Z_{j}^{'} & \lesssim \| u \partial^{\gamma^{'}} \tilde{g}(|u|^{2}) S_{\gamma^{'}} (u,\bar{u})  \|_{L_{t}^{Q'} L_{x}^{R'} (I)}
\left\| \langle D \rangle^{ | \gamma^{'}_{j} |  } u \right\|^{\alpha_{j}^{'}-1}_{L_{t}^{\bar{Q}^{'}_{j,1}} L_{x}^{\bar{R}^{'}_{j,1}} (I)}
\left\| \langle D \rangle^{ | \gamma^{'}_{j} | + \alpha} u \right\|_{L_{t}^{\bar{Q}^{'}_{j,2}} L_{x}^{\bar{R}^{'}_{j,2}} (I)} \\
& \prod \limits_{ i \in [1,..,t-1] } \left\| \langle D \rangle^{\bar{\gamma}_{i}} u \right\|^{\bar{\alpha}_{i}}_{L_{t}^{\bar{Q}^{'}_{i}} L_{x}^{\bar{R}^{'}_{i}} (I)}
\left\| \langle D \rangle^{\bar{\gamma}_{t}} u \right\|^{\bar{\alpha}_{t} -1}_{L_{t}^{\bar{Q}^{'}_{t,1}} L_{x}^{\bar{R}^{'}_{t,1}} (I) }  \left\| \langle D \rangle^{\bar{\gamma}_{t}} u \right\|_{L_{t}^{\bar{Q}^{'}_{t,2}} L_{x}^{\bar{R}^{'}_{t,2}} (I)} \cdot
\end{array}
\label{Eqn:EstZjPrime}
\end{equation}
Here $Q$, $R$, $\bar{Q}_{i}$,  $\bar{R}_{i}$, $\bar{Q}_{t,1}$, $\bar{R}_{t,1}$, $\bar{Q}_{t,2}$, $\bar{R}_{t,2}$, $Q'$, $R'$, $\bar{Q}^{'}_{j,1}$, $\bar{R}^{'}_{j,1}$,
 $\bar{Q}^{'}_{j,2}$, $\bar{R}^{'}_{j,2}$, $\bar{Q}^{'}_{i}$, $\bar{R}^{'}_{i}$, $\bar{Q}^{'}_{t,1}$, $\bar{R}^{'}_{t,1}$, $\bar{Q}^{'}_{t,2}$ and $\bar{R}^{'}_{t,2}$ are numbers  such that $R \neq \infty$, $\bar{R}_{i} \neq \infty$, $\bar{R}_{t,1} \neq \infty$, $\bar{R}_{t,2} \neq \infty$, $R^{'} \neq \infty$, $\bar{R}^{'}_{j,1} \neq \infty$,
 $\bar{R}^{'}_{j,2} \neq \infty$,$\bar{R}_{i}^{'} \neq \infty$,  $\bar{R}_{t,1} \neq \infty$, $\bar{R}_{t,2} \neq \infty$,
 $ \left( \frac{1}{Q}, \frac{1}{R} \right) + \sum \limits_{ i \in [1,...,t-1] } \bar{\alpha}_{i} \left( \frac{1}{_{\bar{Q}_{i}}}, \frac{1}{_{\bar{R}_{i}}} \right)
+ \left( \frac{\bar{\alpha}_{t}- 1}{_{\bar{Q}_{t,1}}} + \frac{1}{_{\bar{Q}_{t,2}}}, \frac{\bar{\alpha}_{t}- 1}{_{\bar{R}_{t,1}}} + \frac{1}{_{\bar{R}_{t,2}}} \right)
 = \frac{n+3}{2(n+1)} (1,1) $ and $  \left( \frac{1}{Q'}, \frac{1}{R'} \right)    +  \left( \frac{\alpha_{j}^{'}-1}{_{\bar{Q}^{'}_{j,1}}} + \frac{1}{_{\bar{Q}^{'}_{j,2}}},
\frac{\alpha_{j}^{'}-1}{_{\bar{R}^{'}_{j,1}}} + \frac{1}{_{\bar{R}^{'}_{j,2}}} \right)
+ \sum \limits_{ i \in [1,..,t-1] } \bar{\alpha}_{i} \left( \frac{1}{_{\bar{Q}^{'}_{i}}}, \frac{1}{_{\bar{R}^{'}_{i}}} \right)
+ \left( \frac{\bar{\alpha}_{t}- 1} {_{\bar{Q}^{'}_{t,1}}} + \frac{1}{_{\bar{Q}^{'}_{t,2}}}, \frac{\bar{\alpha}_{t}- 1}{ _{\bar{R}^{'}_{t,1}}}
+ \frac{1}{_{_{\bar{R}^{'}_{t,2}}}} \right)
= \frac{n+3}{2(n+1)} (1,1) $. Moreover the following properties hold:

\begin{equation}
\text{R.H.S of} \; (\ref{Eqn:EstY}): \;
\left\{
\begin{array}{l}
(a):  \;   \bar{\gamma}_{1} < \bar{\gamma}_{2} < ... < \bar{\gamma}_{t} \\
(b):  \;  \bar{\alpha}_{1} \bar{\gamma}_{1} ... + \bar{\alpha}_{t} \bar{\gamma}_{t} \leq m \\
(c):  \;  \bar{\alpha}_{1} + ... + \bar{\alpha}_{t} = 1_{2}^{*} - 2, \; \text{and}
\end{array}
\right. \text{and}
\label{Eqn:PropY}
\end{equation}

\begin{equation}
\text{R.H.S of }\; (\ref{Eqn:EstZjPrime}): \;
\left\{
\begin{array}{l}
(a):  \;  \bar{\gamma}_{1} < \bar{\gamma}_{2} < ... < \bar{\gamma}_{t}  \\
(b):  \;  \alpha_{j}^{'} |\gamma_{j}^{'}| + \bar{\alpha}_{1} \bar{\gamma}_{1} ... + \bar{\alpha}_{t} \bar{\gamma}_{t} \leq m \\
(c):  \; \alpha_{j}^{'} + \bar{\alpha}_{1} + ... + \bar{\alpha}_{t} = 1_{2}^{*} - 2 .
\end{array}
\right.
\label{Eqn:PropZjPrime}
\end{equation}
Hence $ \bar{\gamma}_{t} \leq \frac{m}{\bar{\alpha}_{t}} $ and $|\gamma^{'}_{j}| \leq \frac{m}{\alpha_{j}^{'}} $. Consequently the following conclusions hold.  Regarding  (\ref{Eqn:PropY}): either $\bar{\gamma}_{t} = m $ and in this case, $t=2$, $\bar{\alpha}_{2}=1 $, and $ \bar{\gamma}_{1} = 0 $; or
$\bar{\gamma}_{i} \leq m-1$ for $1 \leq i \leq t$. Regarding (\ref{Eqn:PropZjPrime}): either $\bar{\gamma}_{t} = m $ and in this case $t \in \{1,2 \}$, $\bar{\alpha}_{t}=1$, and $|\gamma_{j}^{'}| = 0 $; or $ |\gamma^{'}_{j}| = m $ and in this case $t=1$, $\bar{\gamma}_{1}=0$, and $\alpha_{j}^{'} = 1$; or $\bar{\gamma}_{i} \leq m-1$ for $1 \leq i \leq t$ and $ |\gamma_{j}^{'}| \leq m-1 $. \\
\\
\underline{Note}: In the sequel we will implicitly use the conclusions above to prove that some estimates and some embeddings hold \footnote{such as the ones
between ``Then the'' and ``if  $\bar{\gamma}_{t} \neq m $'' below (\ref{Eqn:InterpOneHalf}) }. \\
\\
We first estimate $Z_{j}^{'}$. The elementary estimate
$ \partial^{\gamma^{'}} \tilde{g}(|f|^{2}) S_{\gamma^{'}} (f,\bar{f}) \lesssim g(|f|) \lesssim 1 + |f|^{0+} $ shows that

\begin{equation}
\begin{array}{l}
\| u \partial^{\gamma^{'}} \tilde{g}(|u|^{2}) S_{\gamma^{'}} (u,\bar{u})  \|_{L_{t}^{\frac{2(n+1)}{n-2}} L_{x}^{\frac{2(n+1)}{n-2}} (I)} \\
\lesssim  \| u \|_{L_{t}^{\frac{2(n+1)}{n-2}} L_{x}^{\frac{2(n+1)}{n-2}} (I)}
+ \| u \|_{L_{t}^{\frac{2(n+1)}{n-2}} L_{x}^{\frac{2(n+1)}{n-2}+} (I)} \| u \|^{C}_{L_{t}^{\infty} L_{x}^{1_{2}^{*}} (I) }  \\
\lesssim (\delta')^{c} \langle X_{k - \frac{1}{4}} \left( I,u \right)  \rangle^{C} \cdot
\end{array}
\label{Eqn:LowNormg}
\end{equation}
In the expression above we used the embedding $ H^{k- \frac{1}{4}} \hookrightarrow L^{1_{2}^{*}} $,

\begin{equation}
\begin{array}{ll}
\| u \|_{L_{t}^{\frac{2(n+1)}{n-2}} L_{x}^{\frac{2(n+1)}{n-2}+} (I)} & \lesssim \| u \|^{\theta}_{L_{t}^{\frac{2(n+1)}{n-2}} L_{x}^{\frac{2(n+1)}{n-2}} (I)}
\| u \|^{1- \theta}_{L_{t}^{\frac{2(n+1)}{n-2}} L_{x}^{\frac{2(n+1)}{n-2}++} (I)}  \\
& \lesssim (\delta')^{c} \langle X_{k - \frac{1}{4}} \left( I ,u \right)  \rangle^{C},
\end{array}
\label{Eqn:InterpStrichLittle}
\end{equation}
where at the last line we used the embedding $H^{k - \frac{3}{4},r} \hookrightarrow L^{\frac{2(n+1)}{n-2}++}$ followed by

\begin{equation}
\begin{array}{ll}
\left\| \langle D \rangle^{ k- \frac{3}{4}} u \right\|_{L_{t}^{\frac{2(n+1)}{n-2}} L_{x}^{r} (I)} &
\lesssim \left\| \langle D \rangle^{ k- \frac{3}{4}} u \right\|^{\theta}_{L_{t}^{\frac{2(n+1)}{n-1}} L_{x}^{\frac{2(n+1)}{n-1}} (I)}
\left\| \langle D \rangle^{ k- \frac{3}{4}} u \right\|^{1 - \theta}_{L_{t}^{\infty} L_{x}^{\frac{2n}{n-1}} (I)} \\
& \lesssim \left\| \langle D \rangle^{k- \frac{3}{4}} u \right\|^{\theta}_{L_{t}^{\frac{2(n+1)}{n-1}} L_{x}^{\frac{2(n+1)}{n-1}} (I)}
\| u \|^{1 - \theta}_{L_{t}^{\infty} H^{k- \frac{1}{4}} (I)} \cdot
\end{array}
\label{Eqn:InterpOneHalf}
\end{equation}
In the expression above we used the embedding $H^{k - \frac{1}{4}} \hookrightarrow H^{k - \frac{3}{4}, \frac{2n}{n-1}}$.
Assume that $\bar{\gamma}_{t} \geq |\gamma_{j}^{'}| $. Then the embeddings $ H^{k - \frac{1}{2}, \frac{2(n+1)}{n-1}} \hookrightarrow  H^{\bar{\gamma}_{t} , \frac{2(n+1)}{n-1} } $,$  H^{k - \frac{3}{4}, r} \hookrightarrow H^{\bar{\gamma}_{i},\frac{2(n+1)}{n-2}} $, and
$ H^{k - \frac{3}{4}, r} \hookrightarrow   H^{ |\gamma_{j}^{'}| + \alpha, \frac{2(n+1)}{n-2} }  \hookrightarrow    H^{|\gamma_{j}^{'}|, \frac{2(n+1)}{n-2}} $ hold.
We also have $ H^{k - \frac{3}{4}, r} \hookrightarrow H^{\bar{\gamma}_{t},\frac{2(n+1)}{n-2}} $ if $\bar{\gamma}_{t} \neq m $. Let $( Q',R') := \frac{2(n+1)}{n-2} (1,1) $, $(\bar{Q}^{'}_{j,1}, \bar{R}^{'}_{j,1},\bar{Q}^{'}_{j,2}, \bar{R}^{'}_{j,2},
 \bar{Q}^{'}_{i}, \bar{R}^{'}_{i} ) := \frac{2(n+1)}{n-2} (1,...,1)$, $ \left( \bar{Q}^{'}_{t,1}, \bar{R}^{'}_{t,1} \right) := \frac{2(n+1)}{n-2} (1,1)$ , and
 $ \left( \bar{Q}^{'}_{t,2}, \bar{R}^{'}_{t,2} \right) := \frac{2(n+1)}{n-1} (1,1) $. Hence we get from (\ref{Eqn:EstZjPrime}) that $ Z_{j}^{'} \lesssim \text{R.H.S of} \;  (\ref{Eqn:PropNonLin1}) $. Assume now that $\bar{\gamma}_{t} \leq |\gamma_{j}^{'}| $. Let $( Q',R') := \frac{2(n+1)}{n-2} (1,1) $, $(\bar{Q}^{'}_{j,1}, \bar{R}^{'}_{j,1},
 \bar{Q}^{'}_{i}, \bar{R}^{'}_{i} ) := \frac{2(n+1)}{n-2} (1,...,1)$, $ ( \bar{Q}^{'}_{j,2}, \bar{R}^{'}_{j,2} ) := \frac{2(n+1)}{n-1} (1,1)$, and
 $( Q^{'}_{t,1}, R^{'}_{t,1}, Q^{'}_{t,2}, R^{'}_{t,2} ) := \frac{2(n+1)}{n-2} (1,1,1,1)$. Then $ Z_{j}^{'} \lesssim \text{R.H.S of} \; (\ref{Eqn:PropNonLin1}) $. \\
\\
We then estimate $Y$. We have

\begin{equation}
\begin{array}{l}
\left\| \langle D \rangle^{\alpha} \left( u \partial^{\gamma^{'}} \tilde{g}(|u|^{2}) S_{\gamma^{'}} (u,\bar{u}) \right) \right\|_{L_{t}^{\frac{2(n+1)}{n-2}} L_{x}^{\frac{2(n+1)}{n-2}}(I)} \\
\lesssim  \left\| u \partial^{\gamma^{'}} \tilde{g}(|u|^{2}) S_{\gamma^{'}} (u,\bar{u}) \right\|^{\theta}_{L_{t}^{\frac{2(n+1)}{n-2}} L_{x}^{\frac{2(n+1)}{n-2}} (I)}
\left\| \langle D \rangle^{\alpha +}  \left( u \partial^{\gamma^{'}} \tilde{g}(|u|^{2}) S_{\gamma^{'}} (u,\bar{u}) \right) \right\|^{1- \theta}_{L_{t}^{\frac{2(n+1)}{n-2}} L_{x}^{\frac{2(n+1)}{n-2}} (I)}  \\
\lesssim  (\delta')^{c} \langle X_{k - \frac{1}{4}} \left( I ,u \right) \rangle^{C} \cdot
\end{array}
\nonumber
\end{equation}
In the expression above we used (\ref{Eqn:LowNormg}) and the estimate \\
$ (*): \; \left\| \langle D \rangle^{\alpha +}  \left( u \partial^{\gamma^{'}} \tilde{g}(|u|^{2}) S_{\gamma^{'}} (u,\bar{u}) \right) \right\|_{L_{t}^{\frac{2(n+1)}{n-2}} L_{x}^{\frac{2(n+1)}{n-2}} (I)} \lesssim \langle X_{k- \frac{1}{4}} \left( I,u \right) \rangle^{C} $. In order to derive $(*)$ we proceed as follows. First observe from the definition of the Besov norms in terms of the Paley-Littlewood projectors, a Paley-Littlewood decomposition
into low frequencies and high frequencies, and the H\"older inequality for sequences that \\
$ \left\| \langle D \rangle^{\alpha +}  \left( u \partial^{\gamma^{'}} \tilde{g}(|u|^{2}) S_{\gamma^{'}} (u,\bar{u})  \right)
\right\|_{L_{t}^{\frac{2(n+1)}{n-2}} L_{x}^{\frac{2(n+1)}{n-2}} (I)} \lesssim A + B $ with

\begin{equation}
\begin{array}{l}
A :=  \| u \partial^{\gamma^{'}} \tilde{g}(|u|^{2}) S_{\gamma^{'}} (u,\bar{u}) \|_ {L_{t}^{\frac{2(n+1)}{n-2}} L_{x}^{\frac{2(n+1)}{n-2}} (I)},\; \text{and} \\
B:=  \left\| u \partial^{\gamma^{'}} \tilde{g}(|u|^{2}) S_{\gamma^{'}} (u,\bar{u}) \right\|_{L_{t}^{\frac{2(n+1)}{n-2}} \dot{B}^{\alpha ++}_{\frac{2(n+1)}{n-2},\frac{2(n+1)}{n-2}} (I)} \cdot
\end{array}
\label{Eqn:DefADefB}
\end{equation}
We already know that $ A \lesssim \langle X_{k - \frac{1}{4}} \left( I,u \right) \rangle^{C} $.
So we just need to estimate $B$. Let $ 1 > s > 0$, $p \geq 1$, and $q \geq 1$. Recall (see e.g \cite{bahchem}) that
$\| f \|^{q}_{\dot{B}^{s}_{p,q}} \approx  \int_{\mathbb{R}^{n}} \frac{\| f(x+h) - f(x) \|^{q}_{L^{p}}}{|h|^{n+sq}} \; dh $. By applying the
fundamental theorem of calculus and by proceeding similarly  as in (\ref{Eqn:LowNormg}) we see that
$ \left\| f \partial^{\gamma^{'}} \tilde{g}(|f|^{2}) S_{\gamma^{'}} (f,\bar{f})(x+h) - f \partial^{\gamma^{'}} \tilde{g}(|f|^{2}) S_{\gamma^{'}} (f,\bar{f})(x) \right\|_{L^{p}} \lesssim \left\| f(x+h) - f(x) \right\|_{L^{p}}  + \left\| f(x+h) - f(x) \right\|_{L^{p+}} \| f \|^{C}_{L^{1_{2}^{*}}} $. Hence

\begin{equation}
\begin{array}{ll}
B & \lesssim \left( \| u \|_{L_{t}^{\frac{2(n+1)}{n-2}} \dot{B}^{\alpha ++}_{\frac{2(n+1)}{n-2},\frac{2(n+1)}{n-2}} (I)}
+ \| u \|_{L_{t}^{\frac{2(n+1)}{n-2}} \dot{B}^{\alpha ++ }_{\frac{2(n+1)}{n-2}+,\frac{2(n+1)}{n-2}} (I)} \right)
\langle X_{k - \frac{1}{4}} \left( I ,u \right) \rangle^{C} \\
& \lesssim  \left( \left\| \langle D \rangle^{\alpha +++} u \right\|_{L_{t}^{\frac{2(n+1)}{n-2}} L_{x}^{\frac{2(n+1)}{n-2}} (I)}
+ \left\| \langle D \rangle^{\alpha +++} u \right\|_{L_{t}^{\frac{2(n+1)}{n-2}} L_{x}^{\frac{2(n+1)}{n-2}+} (I)}
\right) \langle X_{k - \frac{1}{4}} \left( I,u \right) \rangle^{C}  \\
& \lesssim \langle X_{k - \frac{1}{4}} \left( I,u \right) \rangle^{C},
\end{array}
\label{Eqn:BesovEst}
\end{equation}
where at the second line we used again a decomposition into low frequencies and high frequencies and at the third line we used the embeddings
$H^{\alpha++,\frac{2(n+1)}{n-2}}, \, H^{\alpha +++, \frac{2(n+1)}{n-2}} \hookrightarrow  H^{k- \frac{3}{4},r} $, and  (\ref{Eqn:InterpOneHalf}). Let $(Q,R) := \frac{2(n+1)}{n-2} (1,1) $
$(\bar{Q}_{i}, \bar{R}_{i}) : = \frac{2(n+1)}{n-2} (1,1)$, $(\bar{Q}_{t,1}, \bar{R}_{t,1}) := \frac{2(n+1)}{n-2}(1,1)$, and
$(\bar{Q}_{t,2}, \bar{R}_{t,2}) :=  \frac{2(n+1)}{n-1} (1,1) $. Then combining the estimates above with similar embeddings as those below
(\ref{Eqn:InterpOneHalf}) we see that $Y \lesssim  \text{R.H.S of} \;  (\ref{Eqn:PropNonLin1}) $.  \\ \\
Assume that $n=3$ (resp. $n=4$) and $ m \geq 5$ (resp. $ m \geq 3 )$. Let $ Min := \min \left( |\gamma_{1}|,..., |\gamma_{p}|, | \bar{\gamma}_{1}|,..., | \bar{\gamma}_{p'}|  \right) $. Either there exists $\tilde{i} \in \{1,...,p \} $ such that
$ |\gamma_{\tilde{i}}| = Min $ or there exists $\tilde{i} \in \{1,..., p' \} $ such that $ |\bar{\gamma}_{\tilde{i}}| =  Min $. We may assume WLOG that
the first scenario occurs. Reordering the $\gamma_{i}$ s and the $\bar{\gamma}_{i}$ s if necessary, we may assume WLOG that $|\gamma_{1}| = Min $. Observe from $\alpha_{1} |\gamma_{1}| + ... + \alpha_{p} |\gamma_{p}| + \bar{\alpha}_{1} |\bar{\gamma}_{1}| +...
 + \bar{\alpha}_{p'} |\bar{\gamma}_{p'}| \leq m $ and from $\alpha_{1} +... + \alpha_{p} + \bar{\alpha}_{1}+...+ \bar{\alpha}_{p'} = 1_{2}^{*} -1 $ that

\begin{equation}
\begin{array}{l}
|\gamma_{1} | \leq \frac{m}{1_{2}^{*} - 1}:
\end{array}
\label{Eqn:Boundgamma1}
\end{equation}
we will use (\ref{Eqn:Boundgamma1}) in the sequel to control norms involving $\partial^{\gamma_{1}} u $. Then we see from (\ref{Eqn:LeibnProd}) that $ \| \langle D \rangle^{\alpha} \partial^{\gamma} X \|_{L_{t}^{\frac{2(n+1)}{n+3}} L_{x}^{\frac{2(n+1)}{n+3}} (I)} $ is bounded by a finite sum of terms of the form

\begin{equation}
\begin{array}{ll}
Y_{a} & := \left\| \langle D \rangle^{\alpha} \left( \partial^{\gamma^{'}} \tilde{g}(|u|^{2}) S_{\gamma^{'}} (u,\bar{u}) \partial^{\gamma_{1}} u \right)  \right\|_{L_{t}^{\frac{2(n+1)}{n-2}} L_{x}^{\frac{2(n+1)}{n-2}} (I)}
\prod \limits_{  s \in [1,...,p+p'] }  \left\| \langle D \rangle^{|\gamma_{s}^{'}|} u \right\|^{\alpha^{'}_{s}}_{L_{t}^{Q_{s}} L_{x}^{R_{s}} (I)},  \\
Z^{'}_{j,a} & :=   \left\|  \partial^{\gamma^{'}} \tilde{g}(|u|^{2}) S_{\gamma^{'}} (u,\bar{u}) \partial^{\gamma_{1}} u  \right\|_{L_{t}^{\frac{2(n+1)}{n-2}}
L_{x}^{\frac{2(n+1)}{n-2}} (I)}
\left\|  \langle D \rangle^{|\gamma^{'}_{j}|} u  \right\|^{\alpha^{'}_{j}-1}_{L_{t}^{Q^{'}_{j,1}} L_{x}^{R^{'}_{j,1}}(I)}
\left\|  \langle D \rangle^{|\gamma^{'}_{j}| + \alpha} u  \right\|_{L_{t}^{Q^{'}_{j,2}} L_{x}^{R^{'}_{j,2}}(I)} \\
&  \prod \limits_{ \substack{ s \in [1,..,p+p']  \\ s \neq j } } \left\| \langle D \rangle^{|\gamma^{'}_{s}|} u  \right\|^{\alpha_{s}^{'}}_{
L_{t}^{Q^{'}_{s}} L_{x}^{R^{'}_{s}} (I)},
\end{array}
\label{Eqn:EstYaPrimeaux}
\end{equation}

\begin{equation}
\begin{array}{ll}
Y_{b} & := \left\| \langle D \rangle^{\alpha} \left( X^{''} \partial^{\tilde{\gamma}^{'}} \tilde{g}(|u|^{2}) \right) \right\|_{L_{t}^{\infty}
L_{x}^{\infty-} (I)}
\prod \limits_{  s \in [1,...,p+p'] }  \left\| \langle D \rangle^{|\gamma_{s}^{'}|} u \right\|^{\alpha^{'}_{s}}_{L_{t}^{Q_{s}} L_{x}^{R_{s}} (I)}, \; \text{and} \\
Z^{'}_{j,b} & :=   \| X^{''} \partial^{\tilde{\gamma}^{'}} \tilde{g}(|u|^{2}) \|_{ L_{t}^{\infty} L_{x}^{\infty} (I)}
\left\|  \langle D \rangle^{|\gamma^{'}_{j}|} u  \right\|^{\alpha^{'}_{j}-1}_{L_{t}^{Q^{'}_{j,1}} L_{x}^{R^{'}_{j,1}}(I)}
\left\|  \langle D \rangle^{|\gamma^{'}_{j}| + \alpha} u  \right\|_{L_{t}^{Q^{'}_{j,2}} L_{x}^{R^{'}_{j,2}}(I)} \\
&  \prod \limits_{ \substack{ s \in [1,..,p+p']  \\ s \neq j } } \left\| \langle D \rangle^{|\gamma^{'}_{s}|} u  \right\|^{\alpha_{s}^{'}}_{
L_{t}^{Q^{'}_{s}} L_{x}^{R^{'}_{s}} (I)} \cdot
\end{array}
\label{Eqn:EstYbPrimejAux}
\end{equation}
Here $j \in \{ 1,...,p + p^{'} \}$. We rewrite the text starting with ``In the expression above we define $\gamma^{'}_{s}$ ""
and finishing with `` $  \left( \gamma^{'}_{s}, \alpha_{s}^{'} \right) := \left( \bar{\gamma}_{s-p}, \bar{\alpha}_{s-p} \right) $. '' just below
(\ref{Eqn:DefYZPrimej}), replacing ``$\gamma^{'}_{1} := (0,...,0)$ '' with `` $ \gamma^{'} := \gamma_{1} $ '' \footnote{Observe that $\gamma_{1}$ is not necessarily equal to
$(0,...,0)$}. Here the $Q$ and $R$ are numbers to be chosen that  satisfy the following constraints: $R_{s} \neq \infty$, $R^{'}_{j,1} \neq \infty$, $R^{'}_{j,2} \neq \infty $,
$R^{'}_{s} \neq \infty$, and

\begin{equation}
\begin{array}{l}
Y_{a} : \; \;  \frac{n-2}{2(n+1)} (1,1) + \sum \limits_{ s \in [1,..,p+p^{'}]} \alpha_{s}^{'} \left( \frac{1}{Q_{s}}, \frac{1}{R_{s}}\right)  = \frac{n+3}{2(n+1)} (1,1) \\
Y_{b} : \; \; \left( \frac{1}{\infty-}, 0 \right) + \sum \limits_{ s \in [1,..,p+p^{'}]} \alpha_{s}^{'} \left( \frac{1}{Q_{s}}, \frac{1}{R_{s}}\right)  = \frac{n+3}{2(n+1)} (1,1) \\
Z^{'}_{j,a} : \; \; \frac{n-2}{2(n+1)} (1,1) + (\alpha_{j}^{'} - 1) \left( \frac{1}{Q^{'}_{j,1}}, \frac{1}{R^{'}_{j,1}} \right)
+ \left( \frac{1}{Q^{'}_{j,2}}, \frac{1}{R^{'}_{j,2}} \right)
+ \sum \limits_{ \substack{ s \in [1,..,p+p']  \\ s \neq j } } \alpha_{s}^{'} \left( \frac{1}{Q_{s}^{'}}, \frac{1}{R_{s}^{'}} \right) = \frac{n+3}{2(n+1)} (1,1) \\
Z^{'}_{j,b} : \; \;   (\alpha_{j}^{'} - 1) \left( \frac{1}{Q^{'}_{j,1}}, \frac{1}{R^{'}_{j,1}} \right)
+ \left( \frac{1}{Q^{'}_{j,2}}, \frac{1}{R^{'}_{j,2}} \right)
+ \sum \limits_{ \substack{ s \in [1,..,p+p']  \\ s \neq j } } \alpha_{s}^{'} \left( \frac{1}{Q_{s}^{'}}, \frac{1}{R_{s}^{'}} \right) = \frac{n+3}{2(n+1)} (1,1)
\end{array}
\nonumber
\end{equation}
Observe that $\left(  \gamma_{1}^{'},..., \gamma^{'}_{p + p^{'}} \right)$ and $ \left( \alpha_{1}^{'},...,\alpha_{p+p'} \right)$ satisfy
$ \alpha_{1}^{'} |\gamma_{1}^{'}| + ... + \alpha^{'}_{p+p'} |\gamma_{p+p'}^{'}| \leq m $. Moreover $ \alpha_{1}^{'} +...+ \alpha^{'}_{p+p} = 1_{2}^{*} - 2 $ if we deal with $Y_{a}$ and $Z^{'}_{j,a}$, and $ \alpha_{1}^{'} +...+ \alpha^{'}_{p+p} = 1_{2}^{*} - 1 $ if we deal with $Y_{b}$ and $Z^{'}_{j,b}$. We rewrite the text starting with ``By collecting the $\gamma_{s}^{'}$ '' and finishing with ``Then'' just above (\ref{Eqn:EstY}) to get

\begin{equation}
\begin{array}{ll}
Y_{a} & \lesssim \left\| \langle D \rangle^{\alpha} \left( \partial^{\gamma^{'}} \tilde{g}(|u|^{2}) S_{\gamma^{'}} (u,\bar{u}) \partial^{\gamma_{1}} u \right)  \right\|_{L_{t}^{\frac{2(n+1)}{n-2}} L_{x}^{\frac{2(n+1)}{n-2}} (I)}
\prod \limits_{i \in [1,...,t-1]}  \left\| \langle D \rangle^{\bar{\gamma}_{i}} u \right\|^{\bar{\alpha}_{i}}_{L_{t}^{\bar{Q}_{i}}
L_{x}^{\bar{R}_{i}}(I)} \\
&  \left\| \langle D \rangle^{\bar{\gamma}_{t}} u \right\|^{\bar{\alpha}_{t} - 1}_{L_{t}^{\bar{Q}_{t,1}} L_{x}^{\bar{R}_{t,1}}(I)}
  \left\| \langle D \rangle^{\bar{\gamma}_{t}} u \right\|_{L_{t}^{\bar{Q}_{t,2}} L_{x}^{\bar{R}_{t,2}}(I)}, \\
&  \\
Y_{b} & \lesssim  \left\| \langle D \rangle^{\alpha} \left( X^{''} \partial^{\tilde{\gamma}^{'}} \tilde{g}(|u|^{2}) \right) \right\|_{L_{t}^{\infty}
L_{x}^{\infty-} (I)}
\prod \limits_{i \in [1,...,t-1]}  \left\| \langle D \rangle^{\bar{\gamma}_{i}} u \right\|^{\bar{\alpha}_{i}}_{L_{t}^{\bar{Q}_{i}}
L_{x}^{\bar{R}_{i}}(I)} \\
&  \left\| \langle D \rangle^{\bar{\gamma}_{t}} u \right\|^{\bar{\alpha}_{t} - 1}_{L_{t}^{\bar{Q}_{t,1}} L_{x}^{\bar{R}_{t,1}}(I)}
  \left\| \langle D \rangle^{\bar{\gamma}_{t}} u \right\|_{L_{t}^{\bar{Q}_{t,2}} L_{x}^{\bar{R}_{t,2}}(I)},
\end{array}
\label{Eqn:EstYAux}
\end{equation}

\begin{equation}
\begin{array}{ll}
Z^{'}_{j,a} & \lesssim   \left\|  \partial^{\gamma^{'}} \tilde{g}(|u|^{2}) S_{\gamma^{'}} (u,\bar{u}) \partial^{\gamma_{1}} u  \right\|_{L_{t}^{\frac{2(n+1)}{n-2}}
L_{x}^{\frac{2(n+1)}{n-2}} (I)}
\left\| \langle D \rangle^{ | \gamma^{'}_{j} |  } u \right\|^{\alpha_{j}^{'}-1}_{L_{t}^{\bar{Q}^{'}_{j,1}} L_{x}^{\bar{R}^{'}_{j,1}} (I)}
\left\| \langle D \rangle^{ | \gamma^{'}_{j} | + \alpha} u \right\|_{L_{t}^{\bar{Q}^{'}_{j,2}} L_{x}^{\bar{R}^{'}_{j,2}} (I)} \\
& \prod \limits_{i \in [1,..,t-1] } \left\| \langle D \rangle^{\bar{\gamma}_{i}} u \right\|^{\bar{\alpha}_{i}}_{L_{t}^{\bar{Q}^{'}_{i}} L_{x}^{\bar{R}^{'}_{i}} (I)}
\left\| \langle D \rangle^{|\bar{\gamma}_{t}|} u \right\|^{\bar{\alpha}_{t} -1}_{L_{t}^{\bar{Q}^{'}_{t,1}} L_{x}^{\bar{R}^{'}_{t,1}} (I) }  \left\| \langle D \rangle^{|\bar{\gamma}_{t}|} u \right\|_{L_{t}^{\bar{Q}^{'}_{t,2}} L_{x}^{\bar{R}^{'}_{t,2}} (I)}, \; \text{and} \\
Z^{'}_{j,b} & \lesssim  \left\|  X^{''} \partial^{\tilde{\gamma}^{'}} \tilde{g}(|u|^{2}) \right\|_{L_{t}^{\infty} L_{x}^{\infty} (I)}
\left\| \langle D \rangle^{ | \gamma^{'}_{j} |  } u \right\|^{\alpha_{j}^{'}-1}_{L_{t}^{\bar{Q}^{'}_{j,1}} L_{x}^{\bar{R}^{'}_{j,1}} (I)}
\left\| \langle D \rangle^{ | \gamma^{'}_{j} | + \alpha} u \right\|_{L_{t}^{\bar{Q}^{'}_{j,2}} L_{x}^{\bar{R}^{'}_{j,2}} (I)} \\
& \prod \limits_{i \in [1,..,t-1] } \left\| \langle D \rangle^{\bar{\gamma}_{i}} u \right\|^{\bar{\alpha}_{i}}_{L_{t}^{\bar{Q}^{'}_{i}} L_{x}^{\bar{R}^{'}_{i}} (I)} \left\| \langle D \rangle^{\bar{\gamma}_{t}} u \right\|^{\bar{\alpha}_{t} -1}_{L_{t}^{\bar{Q}^{'}_{t,1}} L_{x}^{\bar{R}^{'}_{t,1}} (I) }
\left\| \langle D \rangle^{\bar{\gamma}_{t}} u \right\|_{L_{t}^{\bar{Q}^{'}_{t,2}} L_{x}^{\bar{R}^{'}_{t,2}} (I)} \; \cdot
\end{array}
\label{Eqn:ZjPrimeAux}
\end{equation}
Here the $\bar{Q}$,  $\bar{R}$, $\bar{Q}^{'}$, and $\bar{R}^{'}$  numbers satisfy the following constraints: $\bar{R}_{i} \neq \infty$,
$\bar{R}_{t,1} \neq \infty$, $\bar{R}_{t,2} \neq \infty$, $\bar{R}^{'}_{j,1} \neq \infty$, $\bar{R}^{'}_{j,2} \neq \infty$,
$ \bar{R}^{'}_{i} \neq \infty$, $\bar{R}^{'}_{t,1} \neq \infty$, $\bar{R}^{'}_{t,2} \neq \infty$ and

\begin{equation}
\begin{array}{ll}
Y_{a}: & \frac{n-2}{2(n+1)} (1,1) + \sum \limits_{i=1}^{t-1} \bar{\alpha}_{i} \left(  \frac{1}{_{\bar{Q}_{i}}}, \frac{1}{_{\bar{R}_{i}}} \right)
+ (\bar{\alpha}_{t} - 1) \left( \frac{1}{_{\bar{Q}_{t,1}}}, \frac{1}{_{\bar{R}_{t,1}}} \right) + \left( \frac{1}{_{\bar{Q}_{t,2}}}, \frac{1}{_{\bar{R}_{t,2}}} \right) = \frac{n+3}{2(n+1)} (1,1), \\
Y_{b}: &  \left( 0, \frac{1}{\infty-} \right) + \sum \limits_{i=1}^{t-1} \bar{\alpha}_{i} \left( \frac{1}{_{\bar{Q}_{i}}}, \frac{1}{_{\bar{R}_{i}}} \right)
+ (\bar{\alpha}_{t} - 1) \left( \frac{1}{_{\bar{Q}_{t,1}}}, \frac{1}{_{\bar{R}_{t,1}}} \right) + \left( \frac{1}{_{\bar{Q}_{t,2}}}, \frac{1}{_{\bar{R}_{t,2}}} \right)
= \frac{n+3}{2(n+1)} (1,1), \\
Z^{'}_{j,a}: & \frac{n-2}{2(n+1)} (1,1) +  (\alpha_{j}^{'} - 1) \left( \frac{1}{_{\bar{Q}_{j,1}^{'}}},  \frac{1}{_{\bar{R}_{j,1}^{'}}}  \right)
+ \left(  \frac{1}{_{\bar{Q}_{j,2}^{'}}},  \frac{1}{_{\bar{R}_{j,2}^{'}}}  \right) + \sum \limits_{ i \in [1,..,t-1] }
\bar{\alpha}_{i} \left(  \frac{1}{_{\bar{Q}^{'}_{i}}}, \frac{1}{_{\bar{R}^{'}_{i}}} \right) \\
& + (\bar{\alpha}_{t} - 1) \left( \frac{1}{_{\bar{Q}^{'}_{t,1}}}, \frac{1}{_{\bar{R}^{'}_{t,1}}} \right)
+ \left( \frac{1}{_{\bar{Q}^{'}_{t,2}}}, \frac{1}{_{\bar{R}^{'}_{t,2}}} \right) = \frac{n+3}{2(n+1)} (1,1), \; \text{and} \\
Z^{'}_{j,b} : & (\alpha_{j}^{'} - 1) \left( \frac{1}{_{\bar{Q}_{j,1}^{'}}},  \frac{1}{_{\bar{R}_{j,1}^{'}}}  \right)
+ \left(  \frac{1}{_{\bar{Q}_{j,2}^{'}}},  \frac{1}{_{\bar{R}_{j,2}^{'}}}  \right) + \sum \limits_{i \in [1,..,t-1]}
\bar{\alpha}_{i} \left(  \frac{1}{_{\bar{Q}^{'}_{i}}}, \frac{1}{_{\bar{R}^{'}_{i}}} \right) \\
& + (\bar{\alpha}_{t} - 1) \left( \frac{1}{_{\bar{Q}^{'}_{t,1}}}, \frac{1}{_{\bar{R}^{'}_{t,1}}} \right)
+ \left( \frac{1}{_{\bar{Q}^{'}_{t,2}}}, \frac{1}{_{\bar{R}^{'}_{t,2}}} \right) = \frac{n+3}{2(n+1)} (1,1) \cdot
\end{array}
\label{Eqn:CdtionsYaYbZpj}
\end{equation}
We then rewrite the text just above (\ref{Eqn:PropY}) starting  with ``Moreover the following properties hold: '' and finishing with ``$ |\gamma^{'}_{j}| \leq m - 1$. '', replacing the condition  ``$(c): \bar{\alpha}_{1} + ... + \bar{\alpha}_{t} = 1_{2}^{*}-2$'' ( resp.``$(c): \alpha_{j}^{'} + \bar{\alpha}_{1} + ... + \bar{\alpha}_{t} = 1_{2}^{*}-2$ '') with ``$(c): \bar{\alpha}_{1} + ... + \bar{\alpha}_{t} = 1_{2}^{*}-1$'' (resp. ``$(c): \alpha_{j}^{'} + \bar{\alpha}_{1} + ... + \bar{\alpha}_{t} = 1_{2}^{*}-1$'') when we deal with $Y_{b}$ (resp. $Z^{'}_{j,b}$). In particular the same conclusions as those one below (\ref{Eqn:PropZjPrime}) hold. \\
\\
\underline{Note}: In the sequel we will implicitly use these conclusions (see Note above (\ref{Eqn:LowNormg})). \\
\\
We first estimate $Y_{a}$ and $Z^{'}_{j,a}$. We have

\begin{equation}
\begin{array}{l}
\left\| \langle D \rangle^{\alpha} \left( \partial^{\gamma^{'}} \tilde{g}(|u|^{2}) S_{\gamma^{'}} (u,\bar{u}) \partial^{\gamma_{1}} u \right)  \right\|_{L_{t}^{\frac{2(n+1)}{n-2}} L_{x}^{\frac{2(n+1)}{n-2}} (I)}  \\
\lesssim \left\|  \partial^{\gamma^{'}} \tilde{g}(|u|^{2}) S_{\gamma^{'}} (u,\bar{u}) \partial^{\gamma_{1}} u   \right\|^{\theta}_{L_{t}^{\frac{2(n+1)}{n-2}} L_{x}^{\frac{2(n+1)}{n-2}} (I)}  \left\| \langle D \rangle \left( \partial^{\gamma^{'}} \tilde{g}(|u|^{2}) S_{\gamma^{'}} (u,\bar{u}) \partial^{\gamma_{1}} u \right)  \right\|^{1- \theta}_{L_{t}^{\frac{2(n+1)}{n-2}} L_{x}^{\frac{2(n+1)}{n-2}} (I)}    \\
\lesssim A + A^{\theta} B^{1- \theta} \cdot
\end{array}
\nonumber
\end{equation}
In the expression above $ A :=  \left\| \partial^{\gamma^{'}} \tilde{g}(|u|^{2}) S_{\gamma^{'}} (u,\bar{u}) \partial^{\gamma_{1}} u  \right\|_{L_{t}^{\frac{2(n+1)}{n-2}} L_{x}^{\frac{2(n+1)}{n-2}} (I)} $, \\ and
$ B := \left\| \nabla \left( \partial^{\gamma^{'}} \tilde{g}(|u|^{2}) S_{\gamma^{'}} (u,\bar{u}) \partial^{\gamma_{1}} u  \right)
\right\|_{L_{t}^{\frac{2(n+1)}{n-2}} L_{x}^{\frac{2(n+1)}{n-2}} (I)} $. Proceeding similarly as in (\ref{Eqn:LowNormg}) we get

\begin{equation}
\begin{array}{l}
A  \lesssim \| \partial^{\gamma_{1}} u  \|_{L_{t}^{\frac{2(n+1)}{n-2}} L_{x}^{\frac{2(n+1)}{n-2}} (I)}
+  \| \partial^{\gamma_{1}} u  \|_{L_{t}^{\frac{2(n+1)}{n-2}} L_{x}^{\frac{2(n+1)}{n-2}+} (I)} \| u \|^{C}_{L_{t}^{\infty} L_{x}^{1_{2}^{*}} (I)} \lesssim (\delta')^{c} \langle X_{k- \frac{1}{4}} \left( I ,u \right) \rangle^{C} \cdot
\end{array}
\nonumber
\end{equation}
In the expression above we used the boundedness of the Riesz transforms to get

\begin{equation}
\begin{array}{ll}
\| \partial^{\gamma_{1}} u  \|_{L_{t}^{\frac{2(n+1)}{n-2}} L_{x}^{\frac{2(n+1)}{n-2}} (I)} & \lesssim
\| D^{|\gamma_{1}|} u  \|_{L_{t}^{\frac{2(n+1)}{n-2}} L_{x}^{\frac{2(n+1)}{n-2}} (I)} \\
& \lesssim  \| u \|^{\theta}_{L_{t}^{\frac{2(n+1)}{n-2}} L_{x}^{\frac{2(n+1)}{n-2}} (I)}
\| D^{ |\gamma_{1}| + } u  \|^{1- \theta}_{L_{t}^{\frac{2(n+1)}{n-2}} L_{x}^{\frac{2(n+1)}{n-2}} (I)} \\
& \lesssim (\delta')^{c} \langle X_{k- \frac{1}{4}} (I,u) \rangle^{C} \cdot
\end{array}
\label{Eqn:EstPartialGamma1}
\end{equation}
Here we used $ \| D^{|\gamma_{1}|+} f  \|_{L^{\frac{2(n+1)}{n-2}}}
\lesssim \left\| \langle D \rangle^{k- \frac{3}{4}} f \right\|_{L^{r}} $ $($\text{that follows from
(\ref{Eqn:Boundgamma1})}$)$ and (\ref{Eqn:InterpOneHalf}). We also used

\begin{equation}
\begin{array}{ll}
\| \partial^{\gamma_{1}} u  \|_{L_{t}^{\frac{2(n+1)}{n-2}} L_{x}^{\frac{2(n+1)}{n-2}+} (I)} & \lesssim
\| D^{|\gamma_{1}|} u  \|_{ L_{t}^{\frac{2(n+1)}{n-2}} L^{\frac{2(n+1)}{n-2}+} (I)} \\
& \lesssim \| D^{|\gamma_{1}|} u  \|^{\theta}_{ L_{t}^{\frac{2(n+1)}{n-2}} L_{x}^{\frac{2(n+1)}{n-2}} (I)}
\| D^{|\gamma_{1}|} u  \|^{1 - \theta}_{ L_{t}^{\frac{2(n+1)}{n-2}}  L_{x}^{\frac{2(n+1)}{n-2}++} (I)} \\
& \lesssim (\delta')^{c} \langle X_{k- \frac{1}{4}} (I,u) \rangle^{C},
\end{array}
\label{Eqn:EstPartialGamma2}
\end{equation}
using at the last line the embedding
$ \| D^{|\gamma_{1}|} f  \|_{L^{\frac{2(n+1)}{n-2}++}} \lesssim \left\| \langle D \rangle^{k- \frac{3}{4}} f \right\|_{L^{r}} $. \\
We have $ B \lesssim B_{1} + B_{2} + B_{3} $ with
$ B_{1} := \left\|  \partial^{\gamma^{'}} \tilde{g}(|u|^{2}) S_{\gamma^{'}}(u,\bar{u}) \nabla \partial^{\gamma_{1}} u
\right\|_{L_{t}^{\frac{2(n+1)}{n-2}} L_{x}^{\frac{2(n+1)}{n-2}} (I)} $, \\
$ B_{2} := \left\| \partial^{\gamma' + 1} \tilde{g}(|u|^{2}) S_{\gamma'}(u,\bar{u}) u \nabla u  \partial^{\gamma_{1}} u \right\|_{L_{t}^{\frac{2(n+1)}{n-2}} L_{x}^{\frac{2(n+1)}{n-2}} (I)}$, \\
 $ B_{3} :=  \left\| \partial^{\gamma'} \tilde{g}(|u|^{2}) S_{\gamma'-1}(u,\bar{u}) \nabla u  \partial^{\gamma_{1}} u \right\|_{L_{t}^{\frac{2(n+1)}{n-2}} L_{x}^{\frac{2(n+1)}{n-2}} (I)} $ if $\gamma^{'} \neq 1$, and terms that are similar to $B_{1}$, $B_{2}$, and $B_{3}$. Proceeding similarly as in
 (\ref{Eqn:LowNormg}) we get

\begin{equation}
\begin{array}{ll}
B_{1} & \lesssim \left\| \nabla \partial^{\gamma_{1}} u \right\| _{ L_{t}^{\frac{2(n+1)}{n-2}} L_{x}^{\frac{2(n+1)}{n-2}} (I)}
 + \left\| \nabla \partial^{\gamma_{1}} u \right\|_{L_{t}^{\frac{2(n+1)}{n-2}} L_{x}^{\frac{2(n+1)}{n-2}+} (I)}
 \| u \|^{C}_{L_{t}^{\infty} L_{x}^{1_{2}^{*}} (I)} \\
 & \lesssim \langle X_{k-\frac{1}{4}} (I,u) \rangle^{C} \cdot
\end{array}
\nonumber
\end{equation}
In the expression above we used at the last line the embedding
$ \left\| \nabla \partial^{\gamma_{1}} u \right\|_{L_{t}^{\frac{2(n+1)}{n-2}} L_{x}^{p} (I)} \lesssim  \left\| \langle D \rangle^{|\gamma_{1}| + 1}  u \right\|_{L_{t}^{\frac{2(n+1)}{n-2}} L_{x}^{p} (I)}
\lesssim  \left\| \langle D \rangle^{k - \frac{3}{4}}  u \right\|_{L_{t}^{\frac{2(n+1)}{n-2}} L_{x}^{r} (I)}
$ for $p \in \left\{ \frac{2(n+1)}{n-2}, \frac{2(n+1)}{n-2}+ \right\} $. We only estimate $B_{2}$: $B_{3}$ is estimated similarly. Embeddings and the estimate $\left| \partial^{\gamma^{'} +1} \tilde{g}(|u|^{2})  S_{\gamma'}(u,\bar{u}) u \right| \lesssim 1 $ show that

\begin{equation}
\begin{array}{ll}
B_{2} & \lesssim \| \nabla u \|_{L_{t}^{\infty} L_{x}^{\infty} (I)}
\left\| \partial^{\gamma_{1}} u  \right\|_{L_{t}^{\frac{2(n+1)}{n-2}} L_{x}^{\frac{2(n+1)}{n-2}} (I)} \\
 & \lesssim \| u \|_{L_{t}^{\infty} H^{k-\frac{1}{4}} (I)}
 \left\| \langle D \rangle^{k- \frac{3}{4}} u \right\|_{ L_{t}^{\frac{2(n+1)}{n-2}} L_{x}^{r} (I)} \\
 & \lesssim  \langle X_{k - \frac{1}{4}} (I,u) \rangle^{C} \cdot
\end{array}
\nonumber
\end{equation}
If we deal with $Z_{j,a}^{'}$ then
there are again two options: if $ | \gamma_{j}^{'} | \geq |\bar{\gamma}_{t}| $ then
 $ ( \bar{Q}^{'}_{i}, \bar{R}^{'}_{i} ) := \frac{2(n+1)}{n-2} (1,1) $,
 $ ( \bar{Q}^{'}_{t,1}, \bar{R}^{'}_{t,1},\bar{Q}^{'}_{t,2}, \bar{R}^{'}_{t,2}) := \frac{2(n+1)}{n-2} (1,1,1,1) $,
 $ ( \bar{Q}^{'}_{j,1}, \bar{R}^{'}_{j,1} ) := \frac{2(n+1)}{n-2} (1,1) $ and
 $ ( \bar{Q}^{'}_{j,2}, \bar{R}^{'}_{j,2} ) := \frac{2(n+1)}{n-1} (1,1) $; if
 $ | \gamma_{j}^{'} | \leq |\bar{\gamma}_{t}| $ then
 $ ( \bar{Q}^{'}_{j,1}, \bar{R}^{'}_{j,1}, \bar{Q}^{'}_{j,2}, \bar{R}^{'}_{j,2}) := \frac{2(n+1)}{n-2} (1,1,1,1) $,
 $ ( \bar{Q}^{'}_{i}, \bar{R}^{'}_{i}, \bar{Q}^{'}_{t,1}, \bar{R}^{'}_{t,1}  ) := \frac{2(n+1)}{n-2} (1,1,1,1) $ and
 $ ( \bar{Q}^{'}_{t,2}, \bar{R}^{'}_{t,2} ) := \frac{2(n+1)}{n-1} (1,1) $. If we deal with $Y_{a}$ then we choose
$ ( \bar{Q}_{i}, \bar{R}_{i} ) := \frac{2(n+1)}{n-2} (1,1) $, $(\bar{Q}_{t,1}, \bar{R}_{t,1}) :=  \frac{2(n+1)}{n-2} (1,1)$, and
$(\bar{Q}_{t,2}, \bar{R}_{t,2}) :=  \frac{2(n+1)}{n-1} (1,1)$. Hence, combining the above estimates with similar embeddings as those (\ref{Eqn:LowNormg}),  $Y_{a} + Z_{j,a}^{'} \lesssim  (\delta')^{c}
 \langle X_{k- \frac{1}{4}} \left( I,u \right) \rangle^{C} $. \\
We then estimate $Y_{b}$ and $Z^{'}_{j,b}$. Let $ \breve{r} := \infty- $ if $ \alpha > 0 $ and
$\breve{r} := \infty $ if $\alpha=0$. We claim that

\begin{equation}
\begin{array}{l}
\left\| \langle D \rangle^{\alpha} \left( X^{''} \partial^{\tilde{\gamma}^{'}} \tilde{g}(|u|^{2}) \right) \right\|_{L_{t}^{\infty} L_{x}^{\breve{r}} ([0,T_{l}])} \lesssim \langle X_{k- \frac{1}{4}} \left( I,u \right) \rangle^{C} \cdot
\end{array}
\nonumber
\end{equation}
Indeed, we may assume WLOG that $\tilde{p} \in \mathbb{N}^{*}$.  We see from (\ref{Eqn:LeibnProd}), and H\"older inequality that it suffices to prove that the three embeddings hold: $ i \in \{1,...,\tilde{p} \}: \;  \left\| \langle D \rangle^{\alpha} \partial^{\delta_{i}} f \right\|_{L^{\breve{r}}} \lesssim \| f \|_{H^{k- \frac{1}{4}}} $,
$ i \in \{1,...,\tilde{p}^{'} \}: \;  \left\| \langle D \rangle^{\alpha} \partial^{\bar{\delta}_{i}} f \right\|_{L^{\breve{r}}} \lesssim \| f \|_{H^{k- \frac{1}{4}}} $, and $ \left\| \langle D \rangle^{\alpha} \left(  \partial^{\tilde{\gamma}^{'}} \tilde{g} (|f|^{2}) \partial^{\delta_{1}} f  \right) \right\|_{L^{\breve{r}}} \lesssim \langle \| f \|_{H^{k- \frac{1}{4}}} \rangle^{C}$. Clearly the first two
embedding hold, in view of (\ref{Eqn:Estm}). If $\alpha = 0$ then the last embedding follows from the estimate
$ (\triangle): \; \left\| \partial^{\tilde{\gamma}^{'}} \tilde{g}(|f|^{2}) \partial^{\delta_{1}} f  \right\|_{L^{\infty}} \lesssim \| \partial^{\delta_{1}} f \|_{L^{\infty}} \lesssim
\| f \|_{H^{k- \frac{1}{4}}} $. Assume now that $\alpha \neq 0$. Then we use similar arguments as those between ``In the expression above'' and (\ref{Eqn:BesovEst}) to conclude. More precisely $ \left\| \langle D \rangle^{\alpha} \left( \partial^{\tilde{\gamma}^{'}} \tilde{g} (|f|^{2})  \right) \partial^{\delta_{1}} f  \right\|_{L^{\infty-}} \lesssim A + B $ with $ A := \left\| \partial^{\tilde{\gamma}^{'}} \tilde{g} (|f|^{2}) \partial^{\delta_{1}} f  \right\|_{L^{\infty-}} $ and $ B := \left\| \partial^{\tilde{\gamma}^{'}} \tilde{g} (|f|^{2}) \partial^{\delta_{1}} f  \right\|_{\dot{B}^{\alpha+}_{\infty-, \infty-}} $; a straightforward modification of $(\triangle)$ shows that $A \lesssim  \langle \| f \|_{H^{k- \frac{1}{4}}} \rangle^{C} $; the fundamental theorem of calculus yields

\begin{equation}
\begin{array}{l}
\left| \partial^{\tilde{\gamma}^{'}} \tilde{g} (|f|^{2}) \partial^{\delta_{1}} f(x+h) - \partial^{\tilde{\gamma}^{'}} \tilde{g} (|f|^{2})
\partial^{\delta_{1}} f (x) \right|_{L^{\infty-}} \\
\lesssim  \left|  \partial^{\tilde{\gamma}^{'}} \tilde{g} (|f|^{2})(x+h) - \partial^{\tilde{\gamma}^{'}} \tilde{g} (|f|^{2})(x) \right|
| \partial^{\delta_{1}} f(x+h)|
+ \left| \partial^{\delta_{1}} f(x+h) - \partial^{\delta_{1}} f(x) \right| \left| \partial^{\tilde{\gamma}^{'}} \tilde{g} (|f|^{2})(x) \right| \\
\lesssim  \left| \partial^{\delta_{1}} f(x+h) - \partial^{\delta_{1}} f(x) \right| \langle \partial^{\delta_{1}} f(x+h) \rangle \nonumber
\end{array}
\nonumber
\end{equation}
Hence using the boundedness of the Riesz transforms we get  $ B \lesssim \| \partial^{\delta_{1}} f \|_{\dot{B}^{\alpha+}_{\infty-, \infty-}}  \langle \| \partial^{\delta_{1}} f \|_{L^{\infty}} \rangle
\lesssim \| D^{|\delta_{1}|} f \|_{\dot{B}^{\alpha+}_{\infty-, \infty-}} \langle \| \partial^{\delta_{1}} f \|_{L^{\infty}} \rangle \lesssim   \langle \| f \|_{H^{k- \frac{1}{4}}} \rangle^{C} $. \\
If we deal with $Y_{b}$ then we choose $ ( \bar{Q}_{i}, \bar{R}_{i} ) := \frac{2(n+1)}{n-2} (1,1+) $,
$(\bar{Q}_{t,1}, \bar{R}_{t,1}) :=  \frac{2(n+1)}{n-2} (1,1)$, and
$(\bar{Q}_{t,2}, \bar{R}_{t,2}) :=  \frac{2(n+1)}{n-1} (1,1)$. Observe that
$ \left\| \langle D \rangle^{\bar{\gamma}_{i}} u \right\|_{L_{t}^{\frac{2(n+1)}{n-2}} L_{x}^{\frac{2(n+1)}{n-2}+}
(I)} \lesssim (\delta')^{c} \langle X_{k- \frac{1}{4}} \left( I,u \right) \rangle^{C} $ by using a similar scheme as that in
(\ref{Eqn:EstPartialGamma2}). If we deal with $Z_{j,b}^{'}$ then we have $\left\| X^{''}  \partial^{\tilde{\gamma}^{'}} \tilde{g}(|u|^{2})
\right\|_{L_{t}^{\infty} L_{x}^{\infty} (I)} \lesssim  \langle X_{k- \frac{1}{4}} \left( I,u \right) \rangle^{C}   $.
There are again two options. If $ | \gamma_{j}^{'} | \geq |\bar{\gamma}_{t}| $ then let
 $ ( \bar{Q}^{'}_{i}, \bar{R}^{'}_{i} ) := \frac{2(n+1)}{n-2} (1,1) $,
 $ ( \bar{Q}^{'}_{t,1}, \bar{R}^{'}_{t,1},\bar{Q}^{'}_{t,2}, \bar{R}^{'}_{t,2}) := \frac{2(n+1)}{n-2} (1,1,1,1) $,
 $ ( \bar{Q}^{'}_{j,1}, \bar{R}^{'}_{j,1} ) := \frac{2(n+1)}{n-2} (1,1) $, and
 $ ( \bar{Q}^{'}_{j,2}, \bar{R}^{'}_{j,2} ) := \frac{2(n+1)}{n-1} (1,1) $. Observe from
 $ \left\| \langle D \rangle^{\bar{\gamma}_{t}^{+}} f \right\|_{L^{\frac{2(n+1)}{n-2}}} \lesssim
\left\| \langle D \rangle^{k - \frac{3}{4}} f \right\|_{L^{r}} $ and (\ref{Eqn:InterpOneHalf}) that \\
 $ \left\| \langle D \rangle^{\bar{\gamma}_{t}} u \right\|_{L_{t}^{\frac{2(n+1)}{n-2}} L_{x}^{\frac{2(n+1)}{n-2}}
(I)} \lesssim \| u \|^{\theta}_{L_{t}^{\frac{2(n+1)}{n-2}} L_{x}^{\frac{2(n+1)}{n-2}} (I)}
\left\| \langle D \rangle^{\bar{\gamma}_{t}^{+}}  u \right\|^{1- \theta}_{L_{t}^{\frac{2(n+1)}{n-2}} L_{x}^{\frac{2(n+1)}{n-2}} (I)}
\lesssim  (\delta')^{c} \langle X_{k- \frac{1}{4}} \left( I,u \right) \rangle^{C} $. If $ | \gamma_{j}^{'} | \leq |\bar{\gamma}_{t}| $ then
 $ ( \bar{Q}^{'}_{j,1}, \bar{R}^{'}_{j,1}, \bar{Q}^{'}_{j,2}, \bar{R}^{'}_{j,2}) := \frac{2(n+1)}{n-2} (1,1,1,1) $,
 $ ( \bar{Q}^{'}_{i}, \bar{R}^{'}_{i}, \bar{Q}^{'}_{t,1}, \bar{R}^{'}_{t,1}  ) := \frac{2(n+1)}{n-2} (1,1,1,1) $ and
 $ ( \bar{Q}^{'}_{t,2}, \bar{R}^{'}_{t,2} ) := \frac{2(n+1)}{n-1} (1,1) $. Hence combining again the estimates above
with similar embeddings as those below (\ref{Eqn:InterpOneHalf}) we see that $ Y_{b} + Z_{j,b}^{'} \lesssim  (\delta ')^{c}
 \langle  X_{k- \frac{1}{4}} \left( I,u \right) \rangle^{C} $.

\subsection{Appendix $B$}
\label{Subsec:AppB}

In this appendix we prove Proposition \ref{Prop:LocalWell} by using a standard fixed point argument and standard techniques. \\
\\
Let $\delta:= \delta(M) >  0$  be a positive constant small enough such that all the estimates (and statements) below are true.  \\
\\
We define for some $C^{'} > 0$ large enough the following spaces

\begin{equation}
\begin{array}{l}
\mathcal{Z}_{1} := \mathcal{B} \left( \mathcal{C} \left( [0,T_{l}], H^{k} \right) \cap
\mathcal{C}^{1} \left( [0,T_{l}], H^{k-1} \right) \cap F([0,T_{l}]); C^{'} M \right), \; \text{and} \\
\mathcal{Z}_{2} := \mathcal{B} \left(  L_{t}^{\frac{2(n+1)}{n-2}} L_{x}^{\frac{2(n+1)}{n-2}} ([0,T_{l}]) ; 2 \delta \right) \cdot
\end{array}
\nonumber
\end{equation}
Here $\mathcal{B}(\mathcal{E}; \bar{r})$ denotes the closed ball centered at the origin with radius $\bar{r} > 0$ in the normed space $\mathcal{E}$. $ \mathcal{Z}_{1} \cap \mathcal{Z}_{2}$ is a closed space of the Banach space

\begin{equation}
\begin{array}{ll}
\mathcal{Z} & := \mathcal{C} \left( [0,T_{l}], H^{k} \right) \cap \mathcal{C}^{1} \left( [0,T_{l}], H^{k-1} \right)
\cap L_{t}^{\frac{2(n+1)}{n-2}} L_{x}^{\frac{2(n+1)}{n-2}} ([0,T_{l}]) \cap F([0,T_{l}]) :
\end{array}
\end{equation}
therefore it is also a Banach space. Let $\Psi$ be defined by

\begin{equation}
\begin{array}{ll}
 u \in  \mathcal{Z}_{1} \cap \mathcal{Z}_{2} & \rightarrow \Psi(u) :=  \cos{ \left( t \langle D \rangle \right)} u_{0} + \frac{\sin{ \left( t \langle D \rangle \right)}}{\langle D \rangle} u_{1} -
\int_{0}^{t} \frac{\sin{ \left( (t-  t') \langle D \rangle \right)}}{ \langle D \rangle } \left(  |u(t')|^{1_{2}^{*} - 2} u(t') g(|u(t')|) \right) \; d t^{'}
\end{array}
\nonumber
\end{equation}
\hphantom \\ \\
In the sequel we  prove  that  $\Psi(\mathcal{Z}_{1} \cap \mathcal{Z}_{2}) \subset \mathcal{Z}_{1} \cap \mathcal{Z}_{2}$ and that $\Psi$ is a contraction.
With these two results we can apply the fixed point theorem. Therefore Proposition \ref{Prop:LocalWell} holds.

\subsubsection{$\Psi \left( \mathcal{Z}_{1} \cap \mathcal{Z}_{2} \right) \subset \mathcal{Z}_{1} \cap \mathcal{Z}_{2}$}

Let $r$ be such that $\frac{n-2}{2(n+1)} + \frac{n}{r} = \frac{n}{2} - \frac{1}{2} $. Let $\tilde{k}$ be the number defined in Section $3$.
From (\ref{Eqn:StrichEst}), the Sobolev embedding $ H^{1 - \frac{1}{2},r} \hookrightarrow L^{\frac{2(n+1)}{n-2}} $, (\ref{Eqn:LeibnCompos}), and similar arguments as those in
(\ref{Eqn:InitContDalphX}), (\ref{Eqn:InitInterpStrichLittle}), and (\ref{Eqn:InitInterpOneHalf}), we see that

\begin{equation}
\begin{array}{l}
\| u \|_{L_{t}^{\frac{2(n+1)}{n-2}} L_{x}^{\frac{2(n+1)}{n-2}}([0,T_{l}])} -
\| u_{l,0} \|_{L_{t}^{\frac{2(n+1)}{n-2}} L_{x}^{\frac{2(n+1)}{n-2}}([0,T_{l}])} \\
\lesssim  \left\| u_{nl,0} \right\|_{L_{t}^{\frac{2(n+1)}{n-2}} H^{\frac{1}{2},r} ([0,T_{l}]) } \\
\lesssim \left\| \langle D \rangle^{\frac{1}{2}} \left( |u|^{1_{2}^{*} - 2} u g(|u|) \right) \right\|_{ L_{t}^{\frac{2(n+1)}{n+3}} L_{x}^{\frac{2(n+1)}{n+3}} ([0,T_{l}])} \\
\lesssim \left\| \langle D \rangle^{\frac{1}{2}} u  \right\|_{L_{t}^{\frac{2(n+1)}{n-1}} L_{x}^{\frac{2(n+1)}{n-1}} ([0,T_{l}])}
\left(
\| u \|^{1_{2}^{*}-2}_{L_{t}^{\frac{2(n+1)}{n-2}} L_{x}^{\frac{2(n+1)}{n-2}} ([0,T_{l}]) } +
\| u \|^{1_{2}^{*}-2}_{L_{t}^{\frac{2(n+1)}{n-2}} L_{x}^{\frac{2(n+1)}{n-2}+} ([0,T_{l}])} \| u \|^{C}_{L_{t}^{\infty} L_{x}^{1_{2}^{*}}([0,T_{l}]) }
\right) \\
\lesssim \delta^{1+} \langle M \rangle^{C}
\end{array}
\label{Eqn:StrichSmall}
\end{equation}
By Lemma \ref{lem:NonlinearControl}, (\ref{Eqn:StrichEst}), and
(\ref{Eqn:StrichEst2}), we also get $ X_{k} \left( \Psi(u),[0,T_{l}] \right) \lesssim  \left\| (u_{0},u_{1} ) \right\|_{H^{k} \times H^{k-1}}
+ \delta^{\bar{c}} \langle M \rangle^{\bar{C} +1} \leq C' M$. Hence $\Psi \left( \mathcal{Z}_{1} \cap \mathcal{Z}_{2} \right) \subset \mathcal{Z}_{1} \cap \mathcal{Z}_{2}$.

\subsubsection{$\Psi$ is a contraction}
\hphantom{a} \\
\\
\underline{Note}: Before starting the proof, we replace ``$\delta$'' with ``$\delta^{'}$'' in Appendix $A$. In the proof we will often refer to portions of the text or to paragraphs written in Appendix $A$, taking into account this substitution. \\
\\
Let $h( z,\bar{z}) :=|z|^{1_{2}^{*}- 2} z g(|z|) $. \\
\\
Assume that $n \in \{ 3,4 \}$. \\
\\
From the fundamental theorem of calculus and norm conservation properties by taking the conjugate we get

\begin{equation}
\begin{array}{l}
\left\| \Psi(u) - \Psi(v) \right\|_{\mathcal{Z}_{1} \cap \mathcal{Z}_{2}} \\ \lesssim
\left\|  \langle D \rangle^{ k - \frac{1}{2}} \left( h(u,\bar{u}) - h(v, \bar{v}) \right) \right\|_{L_{t}^{\frac{2(n+1)}{n+3}} L_{x}^{\frac{2(n+1)}{n+3}}([0,T_{l}])} \\
\lesssim \sup_{\tau \in [0,1]} \sum \limits_{q \in \{ z ,\bar{z} \}}
\left[
\begin{array}{l}
\left\| \langle  D \rangle^{k - \frac{1}{2}} \left( \p_{q} h( w_{\tau}, \overline{w_{\tau}}) \right) \right\|_{L_{t}^{\frac{2(n+1)}{5}} L_{x}^{\frac{2(n+1)}{5}} ([0,T_{l}])  }
 \| u - v \|_{L_{t}^{\frac{2(n+1)}{n-2}} L_{x}^{\frac{2(n+1)}{n-2}} ([0,T_{l}]) } \\
+ \left\| \p_{q} h( w_{\tau}, \overline{w_{\tau}}) \right\|_{L_{t}^{\frac{n+1}{2}} L_{x}^{\frac{n+1}{2}} ([0,T_{l}])}
\left\| \langle D \rangle^{k- \frac{1}{2}} (u - v) \right\|_{L_{t}^{\frac{2(n+1)}{n-1}} L_{x}^{\frac{2(n+1)}{n-1}} ([0,T_{l}]) }
\end{array}
\right]
\end{array}
\label{Eqn:DiffPsi1}
\end{equation}
Here $ w_{\tau} := (1- \tau) u + \tau v $. We first estimate  $ \left\| \p_{q} h( w_{\tau}, \overline{w_{\tau}}) \right\|_{L_{t}^{\frac{n+1}{2}} L_{x}^{\frac{n+1}{2}} ([0,T_{l}])} $.
By using similar arguments as those used in (\ref{Eqn:LowNormg})

\begin{equation}
\begin{array}{ll}
\left\| \p_{q} h( w_{\tau}, \overline{w_{\tau}}) \right\|_{L_{t}^{\frac{n+1}{2}} L_{x}^{\frac{n+1}{2}} ([0,T_{l}])}  \\
\lesssim \| w_{\tau} \|^{1_{2}^{*} - 2}_{L_{t}^{\frac{2(n+1)}{n-2}} L_{x}^{\frac{2(n+1)}{n-2}} ([0,T_{l}])}
+ \| w_{\tau} \|^{1_{2}^{*} - 2}_{L_{t}^{\frac{2(n+1)}{n-2}} L_{x}^{\frac{2(n+1)}{n-2}+} ([0,T_{l}])}
\| w_{\tau} \|^{C}_{L_{t}^{\infty} L_{x}^{1_{2}^{*}} ([0,T_{l}]} \\
 \lesssim \delta^{c} \langle M \rangle^{C} \cdot
\end{array}
\label{Eqn:EstParthwtau}
\end{equation}
We then estimate
$ Y_{k} := \left\| \langle D \rangle^{k - \frac{1}{2}} \left( \partial_{q} h (w_{\tau}, \overline{w_{\tau}}) \right)
\right\|_{L_{t}^{\frac{2(n+1)}{5}} L_{x}^{\frac{2(n+1)}{5}} ([0,T_{l}]) }$. We write
$ k - \frac{1}{2} = m + \alpha $ with $ 0 \leq \alpha < 1 $ and $m \in \mathbb{N}$. Proceeding as in  (\ref{Eqn:InterpHr}) we have
$ Y_{k} \lesssim \left\| \langle D \rangle^{\alpha} \partial_{q} h (w_{\tau}, \overline{w_{\tau}}) \right\|_{L_{t}^{\frac{2(n+1)}{5}} L_{x}^{\frac{2(n+1)}{5}} ([0,T_{l}])}
+  \sum \limits_{\gamma \in \mathbb{N}^{n}: |\gamma|=m } \left\| \langle D \rangle^{\alpha} \partial^{\gamma}
\partial_{q} h (w_{\tau}, \overline{w_{\tau}} )
\right\|_{L_{t}^{\frac{2(n+1)}{5}} L_{x}^{\frac{2(n+1)}{5}} ([0,T_{l}])} $. We get from (\ref{Eqn:LeibnCompos})

\begin{equation}
\begin{array}{l}
 \left\| \langle D \rangle^{\alpha} \partial_{q} h(w_{\tau}, \overline{w_{\tau}}) \right\|_{L_{t}^{\frac{2(n+1)}{5}} L_{x}^{\frac{2(n+1)}{5}} ([0,T_{l}])} \\
\lesssim \left\| \langle D \rangle^{\alpha} w_{\tau} \right\|_{{L_{t}^{\frac{2(n+1)}{n-1}} L_{x}^{\frac{2(n+1)}{n-1}} ([0,T_{l}])}}
\left(
\begin{array}{l}
\| w_{\tau} \|^{1_{2}^{*}-3}_{L_{t}^{\frac{2(n+1)}{n-2}} L_{x}^{\frac{2(n+1)}{n-2}} ([0,T_{l}]) } \\
+ \| w_{\tau} \|^{1_{2}^{*} - 3}_{L_{t}^{\frac{2(n+1)}{n-2}} L_{x}^{\frac{2(n+1)}{n-2}+} ([0,T_{l}])} \| w_{\tau} \|^{C}_{L_{t}^{\infty} L_{x}^{1_{2}^{*}} ([0,T_{l}])}
\end{array}
\right)
\\
\lesssim \delta^{c} \langle M \rangle^{C} \cdot
\end{array}
\label{Eqn:Partialqh}
\end{equation}
Hence if $m=0$ then $Y_{k} \lesssim \delta^{c} \langle M \rangle^{C} $. \\
\\
\underline{Note}: we may assume WLOG that $m > 0$. \\
\\
We have to estimate $ \left\| \langle D \rangle^{\alpha} \partial^{\gamma}  \partial_{q} h (w_{\tau}, \overline{w_{\tau}})
\right\|_{L_{t}^{\frac{2(n+1)}{5}} L_{x}^{\frac{2(n+1)}{5}} ([0,T_{l}])} $. \\
If $n=3$ (resp. $n=4$) and $m \leq 4$ (resp. $m \leq 2$) then by expanding $ \partial^{\gamma}  \partial_{q} h (w_{\tau},\overline{w_{\tau}}) $ we see that it is a finite sum of terms of the form \\
$ X^{'} := \partial^{\gamma^{'}} \tilde{g}(|w_{\tau}|^{2}) ) S_{\gamma^{'}} (w_{\tau},\overline{w_{\tau}}) (\partial^{\gamma_{1}} w_{\tau})^{\alpha_{1}}...( \partial^{\gamma_{p}} w_{\tau})^{\alpha_{p}} ( \partial^{\bar{\gamma}_{1}} \overline{w_{\tau}} )^{\bar{\alpha}_{1}}... ( \partial^{\bar{\gamma}_{p'}} \overline{w_{\tau}} )^{\bar{\alpha}_{p'}} $.
Here again $\gamma^{'} \in \mathbb{N} $ and $S_{\gamma'} \left( w_{\tau}, \overline{w_{\tau}} \right)$ is of the form $C^{'} w_{\tau}^{p_{1}} \overline{w_{\tau}}^{p_{2}}$  for some $C^{'} \in \mathbb{R}$
and some $(p_{1}, p_{2}) \in \mathbb{N}^{2}$ such that $p_{1} + p_{2} = \gamma^{'}$. Here $ p $, $ p^{'} $, $ \gamma_{1} $,..., $\gamma_{p}$, $\bar{\gamma}_{1} $,...,$\bar{\gamma}_{p'}$, $\alpha_{1}$,..., $\alpha_{p}$, $\bar{\alpha}_{1}$,...,
$ \bar{\alpha}_{p'} $ satisfy the following properties: $ p \neq 0 $ or $p' \neq 0 $,
$ \left( \gamma_{1},..., \gamma_{p},\bar{\gamma}_{1},..., \bar{\gamma}_{p'} \right) \in
\mathbb{N}^{n} \times... \times \mathbb{N}^{n} $,
there exists  $i \in \{ 1,...,p \}$ such that $\gamma_{i} \neq (0,...,0)$ or there exists $i^{'} \in \{1,..,p' \}$ such that
$\gamma_{i'} \neq (0,...,0)$, $ \left( \alpha_{1},...\alpha_{p},
\bar{\alpha}_{1},..., \bar{\alpha}_{p'} \right) \in \mathbb{N}^{*} \times ... \times \mathbb{N}^{*} $, $ \alpha_{1} |\gamma_{1}| ... + \alpha_{p} |\gamma_{p}|
+ \bar{\alpha}_{1} |\bar{\gamma}_{1}| + ... + \bar{\alpha}_{p'} |\bar{\gamma}_{p'}| = m $ and $\alpha_{1} + ... + \alpha_{p} + \bar{\alpha}_{1} + ... + \bar{\alpha}_{p'}= 1_{2}^{*} - 2 $. \\
If $n=3$ (resp. $n=4$) and $m  \leq 3 $ (resp. $m=1$) then $X^{'}$ satisfies $(A)$ or $(B)$ that are defined just below `` If $n=3$ (resp. $n=4$) and
$m \leq 4 $ (resp. $m \leq  2$) ...  ,$(B)$ defined by ''. If $n = 3$ (resp. $n=4$) and $m \geq 4 $ (resp. $m \geq 2$) then
$\partial^{\gamma} \partial_{q} h \left(w _{\tau}, \overline{w_{\tau}} \right)$ can be written as a finite sum of terms of the form $X_{1}^{'} := X^{'}$ and \\
 $ X_{2}^{'} := X^{''}  \partial^{\tilde{\gamma}^{'}} \tilde{g} (|w_{\tau}|^{2}) \left( \partial^{\gamma_{1}} w_{\tau} \right)^{\alpha_{1}} ...
\left( \partial^{\gamma_{p}}  w_{\tau} \right)^{\alpha_{p}}
 \left( \partial^{\bar{\gamma}_{1}} \overline{w_{\tau}} \right)^{\bar{\alpha}_{1}}...
 \left( \partial^{\bar{\gamma}_{p}} \overline{w_{\tau}} \right)^{\bar{\alpha}_{p'}} $. We then rewrite the paragraph
starting with  ``In the definition of $X_{2}^{'}$ above '' just below (\ref{Eqn:WriteXDer}) and finishing with `` of the type $ H^{k} \hookrightarrow L^{\infty}$ '', replacing
`` for the case $n=3$ (resp. $n=4$) and $m \leq 5$ (resp. $m \leq 3$) '' and
`` for $n=3$ (resp. $n=4$)
and $ \gamma \in \mathbb{N}^{n} $ such that $|\gamma | =5 $ (resp. $|\gamma| = 3 $) '' with
`` for the case $n=3$ (resp. $n=4$) and $m \leq 4$ (resp. $m \leq 2$) '' and
`` for $n=3$ (resp. $n=4$ ) and $ \gamma \in \mathbb{N}^{n}$ such that $ |\gamma | =4 $ (resp. $|\gamma| = 2$ )  '' respectively. \\
\\
Assume that $n=3$ (resp. $n=4$) and $ m \leq 3$ (resp. $m = 1$). Then we rewrite the text starting just after ``Assume that $n=3$ (resp. $n=4$) and
$ m \leq 4 $ (resp. $m \leq 2$)'' and finishing with `` and  $ \left( \bar{Q}_{t,2}, \bar{R}_{t,2} \right) :=  \frac{2(n+1)}{n-1} (1,1)$
that is below (\ref{Eqn:BesovEst}), replacing  `` $ 1_{2}^{*} - 2 $'', ``$ \frac{n+3}{2(n+1)} $ '',  ``$u$'' with
 `` $ 1_{2}^{*} - 3 $'', `` $\frac{5}{2(n+1)}$ '', ``$w_{\tau}$'' respectively.
In the case of $n=4$ there are other slight changes to be made: we disregard the observation starting just below
(\ref{Eqn:PropZjPrime}) with
``Hence $\bar{\gamma}_{t} \leq \frac{m}{\bar{\alpha}_{t}}$ ''
and finishing with `` $ |\gamma_{j}^{'}| \leq m - 1 $. ''; regarding $Y$: $ t=2 $, $\bar{\gamma}_{2}=1$, $\bar{\alpha}_{2} = 1$, $(Q, R ) := \frac{2(n+1)}{n-2} (1,1)$,
and $ \left( \bar{Q}_{t,2}, \bar{R}_{t,2} \right) := \frac{2(n+1)}{n-1} (1,1) $; regarding $Z^{'}_{j}$: $t=1$, $\bar{\alpha}_{1}=0$, we disregard the terms where ``$t$'' appears in (\ref{Eqn:EstZjPrime}), $j=1$, $|\gamma^{'}_{1}|=1$, $\alpha^{'}_{1}=1$, $(Q', R' ) := \frac{2(n+1)}{n-2} (1,1)$,  and $ \left( \bar{Q}^{'}_{j,2}, \bar{R}^{'}_{j,2} \right) := \frac{2(n+1)}{n-1} (1,1)$. Hence by using similar embeddings as those from (\ref{Eqn:LowNormg}) to (\ref{Eqn:InterpOneHalf}) we see that $Y_{k} \lesssim \delta^{c} \langle M \rangle^{C}$. Hence $\Psi$ is a contraction. \\
 \\
Assume that $n=3$ (resp. $n=4$) and that $ m \geq 4 $ (resp. $m \geq 2$). Then we rewrite the text starting with ``Assume that $n=3$ (resp. $n=4$)
and $m \geq 5$ (resp. $m \geq 3$) '' and finishing with `` We have $ Y_{b} + Z_{j,b}^{'} \lesssim \delta^{c} \langle X_{k- \frac{1}{4}} \left( [0,T_{l}],u \right) \rangle^{C} $. ''
just above Subsection \ref{Subsec:AppB}, replacing
`` $1_{2}^{*} - 1$ '', `` $1_{2}^{*} -2$ '', `` $\frac{2(n+1)}{n+3}$ '', ``$ \frac{n+3}{2(n+1)} $ '', ``$u$ '' with `` $ 1_{2}^{*} - 2 $ '', `` $ 1_{2}^{*} - 3 $ '',
`` $\frac{2(n+1)}{5}$'', `` $\frac{5}{2(n+1)}$ '',
``$w_{\tau}$'' respectively. In the case of $n=4$, regarding $Y_{a}$ and $Z^{'}_{j,a}$, there are again other slight changes to be made. We disregard the conclusions that we drew below (\ref{Eqn:CdtionsYaYbZpj}); regarding $Y_{a}$: $t=2$, $\bar{\gamma}_{2}=1$, $\bar{\alpha}_{2}=1$, and $ \left( \bar{Q}_{t,2} , \bar{R}_{t,2} \right):= \frac{2(n+1)}{n-1} (1,1)$; regarding $Z_{j,a}^{'}$: $t=1$, $\bar{\alpha}_{1}=0$, we disregard the terms where ``$t$'' appears in (\ref{Eqn:ZjPrimeAux}), $j=1$, $|\gamma_{1}^{'}|=1$, $\alpha_{1}^{'}=1$, and $\left( Q^{'}_{j,2}, R^{'}_{j,2} \right) := \frac{2(n+1)}{n-1} (1,1)$. Again we see that $Y_{k} \lesssim \delta^{c} \langle M \rangle^{C}$. Hence $\Psi$ is a contraction. \\
\\
Assume that $n=5$. \\
\\
Observe that $  \left( \frac{2(n+1)}{n-1}, \frac{2(n+1)}{n-2}, \frac{2n}{n-3} \right) = (3,4,5) $. We use (\ref{Eqn:StrichEst2})
and (\ref{Eqn:LeibnProd}) to get

\begin{equation}
\begin{array}{l}
\left\| \Psi(u) - \Psi(v) \right\|_{\mathcal{Z}_{1} \cap \mathcal{Z}_{2}} \\ \lesssim
\left\|  \langle D \rangle^{k - 1} \left( h(u,\bar{u}) - h(v, \bar{v}) \right) \right\|_{L_{t}^{1} L_{x}^{2}([0,T_{l}])} \\
\lesssim \sup_{\tau \in [0,1]} \sum \limits_{q \in \{ z ,\bar{z} \}}
\left[
\begin{array}{l}
\left\| \langle  D \rangle^{k -1} \left( \p_{q} h( w_{\tau}, \overline{w_{\tau}}) \right) \right\|_{L_{t}^{2} L_{x}^{\frac{10}{3}} ([0,T_{l}])  }
 \| u - v \|_{L_{t}^{2} L_{x}^{5} ([0,T_{l}]) } \\
+ \left\| \p_{q} h( w_{\tau}, \overline{w_{\tau}}) \right\|_{L_{t}^{2} L_{x}^{\frac{10}{3}} ([0,T_{l}])}
\left\| \langle D \rangle^{k- 1} (u - v) \right\|_{L_{t}^{2} L_{x}^{5} ([0,T_{l}]) }
\end{array}
\right]
\end{array}
\nonumber
\end{equation}
Let $Y_{k}:= \left\| \langle D \rangle^{k-1} \left( \partial_{q} h \left(  w_{\tau}, \overline{w_{\tau}} \right) \right)
\right\|_{L_{t}^{2} L_{x}^{\frac{10}{3}} ([0,T_{l}])} $. \\

Assume that $ 1 < k < 2 $. By (\ref{Eqn:LeibnCompos}) and by proceeding similarly as in (\ref{Eqn:LowNormg}) we get

\begin{equation}
\begin{array}{l}
Y_{k} \lesssim \left\| \langle D \rangle^{k-1} w_{\tau} \right\|_{L_{t}^{2+} L_{x}^{5-} ([0,T_{l}])}
\left\| |w_{\tau}|^{\frac{1}{3}} g(|w_{\tau}|) \right\|_{L_{t}^{\infty-} L_{x}^{10+} ([0,T_{l}]) } \\
\lesssim \left\| \langle D \rangle^{k-1} w_{\tau} \right\|_{L_{t}^{2+} L_{x}^{5-} ([0,T_{l}])}
\left(
\begin{array}{l}
\| w_{\tau} \|^{\frac{1}{3}}_{L_{t}^{\infty-} L_{x}^{\frac{10}{3}+} ([0,T_{l}])} \\
+ \| w_{\tau} \|^{\frac{1}{3}}_{L_{t}^{\infty-} L_{x}^{\frac{10}{3}++} ([0,T_{l}])}
\| w_{\tau} \|^{C}_{L_{t}^{\infty} L_{x}^{1_{2}^{*}} ([0,T_{l})]}
\end{array}
\right)
\\
\lesssim \delta^{c} \langle M \rangle^{C} \cdot
\end{array}
\nonumber
\end{equation}
In the expression above we used the embedding $ H^{k} \hookrightarrow  L^{\frac{10}{3}} $ to get

\begin{equation}
\begin{array}{l}
\| w_{\tau} \|_{L_{t}^{\infty-} L_{x}^{\frac{10}{3}+} ([0,T_{l}]) }  \lesssim \| w_{\tau} \|^{\theta}_{L_{t}^{\infty} L_{x}^{\frac{10}{3}} ([0,T_{l}])}
\| w_{\tau} \|^{1- \theta}_{L_{t}^{4} L_{x}^{4}([0,T_{l}])}  \lesssim \delta^{c} \langle M \rangle^{C} \cdot
\end{array}
\label{Eqn:CalcInterm1}
\end{equation}
We also used

\begin{equation}
\begin{array}{l}
\| w_{\tau} \|_{L_{t}^{\infty-} L_{x}^{\frac{10}{3}++} ([0,T_{l}]) }
\lesssim \| w_{\tau} \|^{\theta}_{L_{t}^{\infty-} L_{x}^{\frac{10}{3}+} ([0,T_{l}]) }
\| w_{\tau} \|^{1- \theta}_{L_{t}^{\infty-} L_{x}^{\frac{10}{3}+++} ([0,T_{l}]) }  \\
\lesssim \delta^{c} \langle M \rangle^{C}  \left\| \langle D \rangle^{0+} w_{\tau} \right\|^{1- \theta}_{L_{t}^{\infty-} L_{x}^{\frac{10}{3}+} ([0,T_{l}]) } \\
\lesssim \delta^{c} \langle M \rangle^{C} \cdot
\end{array}
\label{Eqn:CalcInterm2}
\end{equation}
In the expression above we used (\ref{Eqn:CalcInterm1}) and the embedding $ H^{0+, \frac{10}{3} +} \hookrightarrow L^{\frac{10}{3}+++} $. We also used the embedding
$ H^{0+,\frac{10}{3}} \hookrightarrow H^{k} $ to get

\begin{equation}
\begin{array}{l}
\left\| \langle D  \rangle^{0+} w_{\tau} \right\|_{L_{t}^{\infty-} L_{x}^{\frac{10}{3}+} ([0,T_{l}])}  \lesssim
\left\| \langle D  \rangle^{0+} w_{\tau} \right\|^{\theta}_{L_{t}^{\infty} L_{x}^{\frac{10}{3}} ([0,T_{l}])}
\left\| \langle D \rangle^{0+} w_{\tau}  \right\|^{1 - \theta}_{L_{t}^{2} L_{x}^{5} ([0,T_{l}]) } \lesssim \langle M \rangle^{C} \cdot
\end{array}
\nonumber
\end{equation}
We also have

\begin{equation}
\begin{array}{l}
\left\| \p_{q} h( w_{\tau}, \overline{w_{\tau}}) \right\|_{L_{t}^{2} L_{x}^{\frac{10}{3}} ([0,T_{l}])}
\lesssim  \| w_{\tau} \|^{\frac{4}{3}}_{L_{t}^{\frac{8}{3}} L_{x}^{\frac{40}{9}} ([0,T_{l}])}
+  \| w_{\tau} \|^{\frac{4}{3}}_{L_{t}^{\frac{8}{3}} L_{x}^{\frac{40}{9}+} ([0,T_{l}])} \| w_{\tau} \|^{C}_{L_{t}^{\infty} L_{x}^{1_{2}^{*}} ([0,T_{l}])}
\lesssim \delta^{c} \langle M \rangle^{C} \cdot
\end{array}
\nonumber
\end{equation}
In the expression above we used

\begin{equation}
\begin{array}{l}
\| w_{\tau} \|_{L_{t}^{\frac{8}{3}} L_{x}^{\frac{40}{9}} ([0,T_{l}])}  \lesssim
\| w_{\tau} \|^{\theta}_{L_{t}^{2} L_{x}^{5} ([0,T_{l}])}
\| w_{\tau} \|^{1 - \theta}_{L_{t}^{4} L_{x}^{4} ([0,T_{l}])} \lesssim \delta^{c} \langle M \rangle^{C} \cdot
\end{array}
\label{Eqn:EstFourtyNine}
\end{equation}
We also used

\begin{equation}
\begin{array}{l}
\| w_{\tau} \|_{L_{t}^{\frac{8}{3}} L_{x}^{\frac{40}{9}+} ([0,T_{l}])} \lesssim
\| w_{\tau} \|^{\theta}_{L_{t}^{\frac{8}{3}} L_{x}^{\frac{40}{9}} ([0,T_{l}])}
\| w_{\tau} \|^{1 - \theta}_{L_{t}^{\frac{8}{3}} L_{x}^{\frac{40}{9}++} ([0,T_{l}])}
\lesssim \delta^{c} \langle M \rangle^{C} \cdot
\end{array}
\nonumber
\end{equation}
Here we used the embeddings $ H^{0+,\frac{40}{9}} \hookrightarrow L^{\frac{40}{9}++} $ and
the estimate below to get

\begin{equation}
\begin{array}{l}
\left\| \langle D \rangle^{0+} w_{\tau} \right\|_{L_{t}^{\frac{8}{3}} L_{x}^{\frac{40}{9}} ([0,T_{l}])}
\lesssim \left\| \langle D \rangle^{0+} w_{\tau} \right\|^{\theta}_{L_{t}^{2} L_{x}^{5} ([0,T_{l}])}
\left\| \langle D \rangle^{0+} w_{\tau} \right\|^{1- \theta}_{L_{t}^{\infty} L_{x}^{\frac{10}{3}} ([0,T_{l}])}
\lesssim \langle M \rangle^{C} \cdot
\end{array}
\nonumber
\end{equation}
Assume that $ 2 \leq k < \frac{7}{3} $.  We define $\alpha$ to be such that $ k - 2 = \alpha $ with
$ 0  \leq \alpha < \frac{1}{3}$. We already saw (see case $ 1 <  k < 2$ ) that $ \left\| \p_{q} h( w_{\tau}, \overline{w_{\tau}}) \right\|_{L_{t}^{2} L_{x}^{\frac{10}{3}} ([0,T_{l}])} \lesssim \delta^{c} \langle
 M \rangle^{C}$. So it remains to estimate $ Y_{k} $. Expanding the gradient we have

 \begin{equation}
 \begin{array}{ll}
 Y_{k} & \lesssim  \left\| \partial_{q} h(w_{\tau}, \overline{w_{\tau}} ) \right\|_{L_{t}^{2} L_{x}^{\frac{10}{3}} ([0,T_{l}]) }
 + \left\| \langle D \rangle^{k-2} \nabla \left( \partial_{q} h(w_{\tau}, \overline{w_{\tau}}) \right) \right\|_{L_{t}^{2} L_{x}^{\frac{10}{3}} ([0,T_{l}]) } \\
 & \lesssim Y_{k,1} + Y_{k,2} + Y_{k,3} + \; \text{terms that are similar to} \; Y_{k,1}, Y_{k,2}, \; \text{and} \;  Y_{k,3} \cdot
\end{array}
\nonumber
\end{equation}
Here

\begin{equation}
\begin{array}{l}
Y_{k,1} := \left\| \langle D \rangle^{\alpha} \left( G(w_{\tau}, \overline{w_{\tau}}) \nabla w_{\tau} g(|w_{\tau}|) \right) \right\|_{L_{t}^{2} L_{x}^{\frac{10}{3}} ([0,T_{l}])}, \\
Y_{k,2} := \left\| \langle D \rangle^{\alpha} \left( G(w_{\tau}, \overline{w_{\tau}}) \nabla w_{\tau} g^{'}(|w_{\tau}|) |w_{\tau}| \right) \right\|_{L_{t}^{2}
L_{x}^{\frac{10}{3}} ([0,T_{l}])}, \; \text{and} \\
Y_{k,3} := \left\| \langle D \rangle^{\alpha} \left( G(w_{\tau}, \overline{w_{\tau}}) \nabla w_{\tau} g^{''}(|w_{\tau}|) |w_{\tau}|^{2} \right) \right\|_{L_{t}^{2}
L_{x}^{\frac{10}{3}} ([0,T_{l}])} \cdot
\end{array}
\nonumber
\end{equation}
Here $G$ is an H\"older function that is $\mathcal{C}^{1}$ except at the origin and that satisfies $G(f,\bar{f}) \approx |f|^{\frac{1}{3}}$ and $ |G^{'}(f,\bar{f})| \approx |f|^{-\frac{2}{3}} $. We only estimate $Y_{k,1}$: the other terms are estimated similarly. We write $Y_{k,1} \lesssim Y_{k,1,a} + Y_{k,1,b}$ with

\begin{equation}
\begin{array}{l}
Y_{k,1,a} := \left\| \langle D \rangle^{k-1} w_{\tau} \right\|_{L_{t}^{3-} L_{x}^{\frac{30}{7}+} ([0,T_{l}]) }  \left\| G ( w_{\tau},\overline{w_{\tau}} ) g(|w_{\tau}|)
\right\|_{L_{t}^{6+} L_{x}^{15-} ([0,T_{l}])}, \; \text{and} \\
%Y_{k,1,b} := \left\| \langle D \rangle^{\alpha} \left(  G( w_{\tau},\overline{w_{\tau}}) g(|w_{\tau}|) \right) \right\|_{L_{t}^{\infty-} L_{x}^{10+} ([0,T_{l}]) }
%\| \nabla w_{\tau} \|_{L_{t}^{2+} L_{x}^{5-} ([0,T_{l}])} OLD \\
Y_{k,1,b} := \left\| \langle D \rangle^{\alpha} \left(  G( w_{\tau},\overline{w_{\tau}}) g(|w_{\tau}|) \right) \right\|_{L_{t}^{\infty-} L_{x}^{10+} ([0,T_{l}]) }
\| \nabla w_{\tau} \|_{L_{t}^{2+} L_{x}^{5-} ([0,T_{l}])} \cdot
\end{array}
\nonumber
\end{equation}
We first estimate $Y_{k,1,a}$. We get from the embedding $H^{1} \hookrightarrow  L^{\frac{10}{3}}$

\begin{equation}
\begin{array}{l}
\left\| \langle D \rangle^{k-1} w_{\tau} \right\|_{L_{t}^{3-} L_{x}^{\frac{30}{7}+} ([0,T_{l}]) } \lesssim
\left\| \langle D \rangle^{k-1} w_{\tau} \right\|^{\theta}_{L_{t}^{\infty} L_{x}^{\frac{10}{3}} ([0,T_{l}]) }
\left\| \langle D \rangle^{k-1} w_{\tau} \right\|^{1- \theta}_{L_{t}^{2} L_{x}^{5} ([0,T_{l}])}
\lesssim \langle M \rangle^{C} \cdot
\end{array}
\nonumber
\end{equation}
Proceeding similarly  as in (\ref{Eqn:LowNormg}) we have

\begin{equation}
\begin{array}{ll}
\left\| G \left( w_{\tau}, \overline{w_{\tau}} \right) g(|w_{\tau}|)\right\|_{L_{t}^{6+} L_{x}^{15-} ([0,T_{l}])}
& \lesssim \| w_{\tau} \|^{\frac{1}{3}}_{L_{t}^{2+} L_{x}^{5-} ([0,T_{l}])}
+ \| w_{\tau} \|^{\frac{1}{3}}_{L_{t}^{2+} L_{x}^{5}([0,T_{l}]) }
\| w_{\tau} \|^{C}_{L_{t}^{\infty} L_{x}^{1_{2}^{*}} ([0,T_{l}])} \\
& \lesssim \delta^{c} \langle M \rangle^{C} \cdot
\end{array}
\nonumber
\end{equation}
In the expression above we used

\begin{equation}
\begin{array}{l}
\| w_{\tau} \|_{L_{t}^{2+} L_{x}^{5-} ([0,T_{l}])} \lesssim \| w_{\tau} \|^{\theta}_{L_{t}^{2} L_{x}^{5} ([0,T_{l}])}
\| w_{\tau} \|^{1 - \theta}_{L_{t}^{4} L_{x}^{4} ([0,T_{l}])} \lesssim \delta^{c} \langle M \rangle^{C} \cdot
\end{array}
\nonumber
\end{equation}
We also have

\begin{equation}
\begin{array}{l}
\| w_{\tau} \|_{L_{t}^{2+} L_{x}^{5} ([0,T_{l}])}  \lesssim \| w_{\tau} \|^{\theta}_{L_{t}^{2+} L_{x}^{5-}([0,T_{l}])}
 \| w_{\tau} \|^{1- \theta}_{L_{t}^{2+} L_{x}^{5+} ([0,T_{l}])}  \lesssim \delta^{c} \langle M \rangle^{C} \cdot
\end{array}
\nonumber
\end{equation}
Here we used the embedding $H^{0+,5-} \hookrightarrow L^{5+}$ and the estimate below

\begin{equation}
\begin{array}{l}
\| \langle D \rangle^{0+} w_{\tau} \|_{L_{t}^{2+} L_{x}^{5-} ([0,T_{l}]) }
\lesssim  \| \langle D \rangle^{0+} w_{\tau} \|^{\theta}_{L_{t}^{2} L_{x}^{5} ([0,T_{l}])}
\| \langle D \rangle^{0+} w_{\tau} \|^{1 - \theta}_{L_{t}^{\infty} L_{x}^{\frac{10}{3}} ([0,T_{l}])}
\lesssim \langle M \rangle^{C} \cdot
\end{array}
\nonumber
\end{equation}
We then estimate $Y_{k,1,b}$. We get from $ \left\| \langle D \rangle^{k-1} f \right\|_{L^{\frac{10}{3}}} \lesssim  \| f \|_{H^{k}} $

\begin{equation}
\begin{array}{ll}
\| \nabla w_{\tau} \|_{L_{t}^{2+} L_{x}^{5-} ([0,T_{l}])} & \lesssim \left\| \langle D \rangle^{k-1} w_{\tau} \right\|_{L_{t}^{2+} L_{x}^{5-} ([0,T_{l}])} \\
& \lesssim \left\|  \langle D \rangle^{k-1} w_{\tau} \right\|^{\theta}_{L_{t}^{2} L_{x}^{5} ([0,T_{l}])}
\left\|  \langle D \rangle^{k-1} w_{\tau} \right\|^{1 - \theta}_{L_{t}^{\infty} L_{x}^{\frac{10}{3}} ([0,T_{l}])} \\
& \lesssim  \langle M \rangle^{C} \cdot
\end{array}
\nonumber
\end{equation}
We have
$ \left\| \langle D \rangle^{\alpha} \left( G(w_{\tau}, \overline{w_{\tau}} ) g(|w_{\tau}|) \right) \right\|_{L_{t}^{\infty-} L_{x}^{10+} ([0,T_{l}])} \lesssim Z_{1} + Z_{2} $ with
$ Z_{1} :=  \left\|  G(w_{\tau}, \overline{w_{\tau}} )  g(|w_{\tau}|)  \right\|_{L_{t}^{\infty-} L_{x}^{10+}([0,T_{l}])} $ and
$ Z_{2} := \left\| D^{\alpha}  \left( G(w_{\tau}, \overline{w_{\tau}} )  g(|w_{\tau}|) \right) \right\|_{L_{t}^{\infty-} L_{x}^{10+}([0,T_{l}])} $.
We first estimate $Z_{1}$. Proceeding similarly as in (\ref{Eqn:LowNormg})

\begin{equation}
\begin{array}{l}
Z_{1} \lesssim \| w_{\tau} \|^{\frac{1}{3}}_{L_{t}^{\infty-} L_{x}^{\frac{10}{3}+} ([0,T_{l}]) }
+ \| w_{\tau} \|^{\frac{1}{3}}_{L_{t}^{\infty-} L_{x}^{\frac{10}{3}++} ([0,T_{l}])} \langle M \rangle^{C} \lesssim
\delta^{c} \langle M \rangle^{C} \cdot
\end{array}
\nonumber
\end{equation}
In the expression above we used (\ref{Eqn:CalcInterm1}) and (\ref{Eqn:CalcInterm2}). \\
We then estimate $Z_{2}$. Assume that $\alpha > 0 $.  We recall the following lemma:

\begin{lem}{ (see \cite{triroyjensen})}

Let $ 0 <  \alpha^{'} < 1 $. Let $r$ and $\beta$ be such that $\alpha^{'} < \beta < 1$ and $r \beta \geq 1$. Let
$H: \mathbb{R}^{2} \rightarrow \mathbb{R}^{2}$ be a H\"older continuous function with exponent $\beta$, which is $\mathcal{C}^{1}$
(except at the origin) and which satisfies $ |H(f,\bar{f})| \approx |f|^{\beta} $ and $|H^{'}(f,\bar{f})| \approx |f|^{\beta -1 }$.
Let $ 1- \beta \gg \epsilon > 0 $. Then

\begin{equation}
\begin{array}{ll}
\left\| H(f,\bar{f}) g(|f|) \right\|_{\dot{B}_{r,r}^{\alpha^{'}}} & \lesssim  \| f \|^{\beta}_{\dot{B}^{\frac{\alpha^{'}}{\beta}}_{\beta r , \beta r}}
+ \| f \|^{\beta + \epsilon}_{\dot{B}^{\frac{\alpha^{'}}{\beta + \epsilon}}_{(\beta + \epsilon)r,(\beta + \epsilon)r}} \cdot
\end{array}
\label{Eqn:BesovEst2}
\end{equation}

\label{lem:BesovEst}
\end{lem}

\begin{rem}
A straightforward modification of the proof of Lemma \ref{lem:BesovEst} shows that (\ref{Eqn:BesovEst2}) also holds if $g(|f|)$ is replaced with $g^{'}(|f|)|f|$ or
$g^{''}(|f|) |f|^{2}$.
\end{rem}
In the sequel we use similar arguments as those above (\ref{Eqn:DefADefB}). We have

\begin{equation}
\begin{array}{ll}
Z_{2} & \lesssim \left\| G(w_{\tau}, \overline{w_{\tau}} )  g(|w_{\tau}|) \right\|_{L_{t}^{\infty-} \dot{B}^{\alpha-}_{10+,10+} ([0,T_{l}])}
+ \left\| G(w_{\tau}, \overline{w_{\tau}} )  g(|w_{\tau}|) \right\|_{L_{t}^{\infty-} \dot{B}^{\alpha+}_{10+, 10+} ([0,T_{l}])}
\end{array}
\nonumber
\end{equation}
Hence $Z_{2}$ is bounded by Lemma \ref{lem:BesovEst}  by powers of terms of the form
$ \left\| w_{\tau} \right\|_{L_{t}^{\infty-} \dot{B}^{(3 \alpha) \pm}_{\frac{10}{3}+ , \frac{10}{3}+} ([0,T_{l}]) } $ or
$ \left\| w_{\tau} \right\|_{L_{t}^{\infty-} \dot{B}^{(3 \alpha) \pm}_{\frac{10}{3}++ , \frac{10}{3}++} ([0,T_{l}]) } $. Here
$\pm$ denotes the $+$ sign or the $-$ sign. Hence by proceeding as below (\ref{Eqn:BesovEst}) we see that $Z_{2}$
is bounded by powers of terms of the form $ A := \left\| \langle D \rangle^{(3 \alpha)++} w_{\tau} \right\|_{L_{t}^{\infty-} L_{x}^{\frac{10}{3}++} ([0,T_{l}]) }$
or $ B := \left\| \langle D \rangle^{(3 \alpha)++} w_{\tau} \right\|_{L_{t}^{\infty-} L_{x}^{\frac{10}{3}+} ([0,T_{l}]) }$.
We have

\begin{equation}
\begin{array}{l}
B  \lesssim \left\| \langle D  \rangle^{k-1} w_{\tau} \right\|^{\theta}_{L_{t}^{\infty-} L_{x}^{\frac{10}{3}+} ([0,T_{l}])}
\left\| w_{\tau} \right\|^{1- \theta}_{L_{t}^{\infty-} L_{x}^{\frac{10}{3}+} ([0,T_{l}]) } \lesssim \delta^{c} \langle M \rangle^{C} \cdot
\end{array}
\label{Eqn:EstB}
\end{equation}
Here we used (\ref{Eqn:CalcInterm1}) and $\left\|  \langle D  \rangle^{k-1} f \right\|_{L^{\frac{10}{3}}} \lesssim \| f \|_{H^{k}}$ to get

\begin{equation}
\begin{array}{l}
\left\| \langle D  \rangle^{k-1} w_{\tau} \right\|_{L_{t}^{\infty-} L_{x}^{\frac{10}{3}+} ([0,T_{l}])}
\lesssim \left\| \langle D  \rangle^{k-1} w_{\tau} \right\|^{\theta}_{L_{t}^{\infty} L_{x}^{\frac{10}{3}} ([0,T_{l}])}
 \left\| \langle D  \rangle^{k-1} w_{\tau} \right\|^{\theta}_{L_{t}^{2} L_{x}^{5} ([0,T_{l}])} \lesssim \langle M \rangle^{C} \cdot
\end{array}
\nonumber
\end{equation}
We also have

\begin{equation}
\begin{array}{ll}
A \lesssim  \left\| \langle D \rangle^{(3 \alpha)++} w_{\tau} \right\|^{\theta}_{L_{t}^{\infty-} L_{x}^{\frac{10}{3}+} ([0,T_{l}]) }
\left\| \langle D \rangle^{(3 \alpha) ++} w_{\tau} \right\|^{1 - \theta}_{L_{t}^{\infty-} L_{x}^{\frac{10}{3}+++} ([0,T_{l}]) } \lesssim  \delta^{c} \langle M  \rangle^{C} \cdot
\end{array}
\nonumber
\end{equation}
Here we used (\ref{Eqn:EstB}) and the embedding $ \left\| \langle D \rangle^{(3 \alpha) ++} f \right\|_{L^{\frac{10}{3}+++} } \lesssim
\left\| \langle D \rangle^{k-1} f  \right\|_{L^{\frac{10}{3}+}} $. Hence $Z_{2} \lesssim \delta^{c} \langle M \rangle^{C} $. \\
Hence $\Psi$ is a contraction.

\subsection{Appendix $C$}

The following refined Sobolev inequality for $p > q > 1 $ and $s > 0$ holds:

\begin{equation}
\begin{array}{ll}
\| f \|_{L^{p}} & \lesssim \| f \|^{1- \frac{q}{p}}_{B_{\infty,\infty}^{-\frac{qs}{p-q}}} \| \langle D \rangle^{s} f \|^{\frac{q}{p}}_{L^{q}} \cdot
\end{array}
\nonumber
\end{equation}
The proof is essentially well-known in the literature (see e.g \cite{bahchem} and references therein). For convenience we provide the reader with the proof.

\begin{proof}

Writing $ \left| P_{0} f(x) \right|^{p} =  \left| P_{0} f(x) \right|^{p-q} \left| P_{0} f(x) \right|^{q}$  we see that

\begin{equation}
\begin{array}{ll}
\| P_{0} f \|^{p}_{L^{p}} & \lesssim \| P_{0} f \|_{L^{\infty}}^{p-q} \| P_{0} f \|_{L^{q}}^{q}
\end{array}
\nonumber
\end{equation}
Hence $ \| P_{0} f \|_{L^{p}} \lesssim  \| f \|^{1- \frac{q}{p}}_{B_{\infty,\infty}^{-\frac{qs}{p-q}}} \| \langle D \rangle^{s} f \|^{\frac{q}{p}}_{L^{q}} $. \\
\\
Let $ L(x) := \sup_{N \in 2^{\mathbb{N}}} \left( N^{-\frac{qs}{p-q}}  | P_{N} f (x)| \right)$  and let
$ H(x) := \sup_{N \in 2^{\mathbb{N}}} \left( N^{s} |P_{N} f(x)| \right) $. Elementary considerations and the Paley-Littlewood theorem
show that

\begin{equation}
\begin{array}{ll}
\| H \|_{L^{q}} & \lesssim   \left\| \left( \sum \limits_{N \in 2^{\mathbb{N}}} N^{2s} |P_{N} f|^{2} \right)^{\frac{1}{2}}  \right\|_{L^{q}}
 \lesssim \left\| \left( \sum \limits_{N \in 2^{\mathbb{N}}} | \tilde{P}_{N} D^{s} f|^{2} \right)^{\frac{1}{2}} \right\|_{L^{q}} \\  %_{L^{p}} \\
 & \lesssim  \| D^{s} f \|_{L^{q}}  \lesssim  \| \langle D \rangle^{s} f \|_{L^{q}} \cdot
\end{array}
\nonumber
\end{equation}
Here $\tilde{P}_{N}$ is defined in the Fourier domain by
$ \widehat{\tilde{P}_{N} f}(\xi) := \tilde{\psi} \left( \frac{\xi}{N} \right) \hat{f}(\xi) $, with
$ \tilde{\psi}(\xi) := \frac{\psi(\xi)}{|\xi|^{s}} $ ( in other words $\tilde{P}_{N}$ is an operator that behaves like $P_{N}$
in the Fourier domain). We also have by definition of $ B_{\infty,\infty}^{-\frac{qs}{p-q}}$

\begin{equation}
\begin{array}{ll}
\| L \|_{L^{\infty}} & \lesssim \| f \|_{B_{\infty,\infty}^{-\frac{qs}{p-q}}}
\end{array}
\nonumber
\end{equation}
Let $M \in 2^{\mathbb{N}}$ to be chosen. Writing $ f(x) - P_{0} f(x) = \sum \limits_{N \in 2^{\mathbb{N}}} P_{N} f(x) $  and estimating separately the portion of the sum containing the terms $N \leq M $ and that containing the terms $N > M$  we get

\begin{equation}
\begin{array}{ll}
\left| f(x) - P_{0} f(x)  \right| & \lesssim M^{\frac{qs}{p-q}} L(x) + M^{-s} H(x) \\
& \lesssim  H^{\frac{q}{p}}(x) L^{1 - \frac{q}{p}}(x),
\end{array}
\nonumber
\end{equation}
since elementary considerations show that $ \sup \limits_{y \in \mathbb{R}^{+}}  y^{\frac{qs}{p-q}} L(x) + y^{-s} H(x)
\lesssim H^{\frac{q}{p}}(x) L^{1 - \frac{q}{p}}(x) $. Hence

\begin{equation}
\begin{array}{ll}
\left\| f - P_{0} f \right\|_{L^{p}} & \lesssim \| L \|^{1 - \frac{p}{q}}_{L^{\infty}} \| H \|^{\frac{q}{p}}_{L^{q}}  \\
& \lesssim  \| f \|^{1- \frac{p}{q}}_{B_{\infty,\infty}^{-\frac{qs}{p-q}}}
 \| \langle D \rangle^{s} f \|^{\frac{p}{q}}_{L^{q}}  \cdot
\end{array}
\nonumber
\end{equation}

\end{proof}

\subsection{Appendix $D$}

In this appendix we explain why the finiteness of the Strichartz-type norm $\| u \|_{L_{t}^{4} L_{x}^{12}(\mathbb{R})}$
and that of the norm $ \| u \|_{L_{t}^{\infty} \tilde{H}^{2}(\mathbb{R})} +  \| \partial_{t} u \|_{L_{t}^{\infty} \tilde{H}^{1}(\mathbb{R})} $
imply scattering. Here $u$ is a solution of a $3D$-loglog energy-supercritical wave equation studied in \cite{triroysmooth} with data
$(u_{0},u_{1}) \in \tilde{H}^{2} \times \tilde{H}^{1}$. \\
\\
We claim that $K^{-1}(t) \mathbf{u(t)} $ has a limit as $t \rightarrow  \pm \infty $ with

\begin{equation}
\begin{array}{ll}
\mathbf{u(t)} : = \left(
\begin{array}{l}
u(t) \\
\partial_{t} u(t)
\end{array}
\right)
, & \text{and} \;
K(t) := \left(
\begin{array}{ll}
\cos{(t D)} &  \frac{\sin{(t D)}}{D} \\
- D \sin{(t D)} & \cos{(t D)}
\end{array}
\right)
\end{array}
\cdot
\nonumber
\end{equation}
Let $\epsilon > 0$. Let $|t_{1}|$ be large enough so that the estimates below are true. Let $|t_{2}| > |t_{1}|$. Recall that $\mathbf{u(t_2)} = K(t_2- t_1) \mathbf{u(t_1)} + \int_{t_{1}}^{t_{2}} K(t_2 - t') \mathbf{F(u(t'))} \; dt'  $.   Here $\mathbf{F(u(t'))} := \left( 0, -|u|^{4} u \log^{\gamma} \left( \log( 10 + |u|^{2}) \right) (t') \right)^{T}$. Hence we see from Plancherel theorem, the Strichartz estimates for wave equations, H\"older inequality, the embedding $\tilde{H}^{2} \hookrightarrow L^{\infty}$,
and the embedding $\dot{H}^{1} \hookrightarrow L^{6} $ that

\begin{equation}
\begin{array}{ll}
\left\| K^{-1}(t_{2}) \mathbf{u(t_{2})} - K^{-1}(t_{1}) \mathbf{u(t_{1})} \right\|_{\tilde{H}^{2} \times \tilde{H}^{1}} \\
\lesssim \left\| K( t_{2} - t_{1} )  \mathbf{u(t_{1})} -  \mathbf{u(t_{2})} \right\|_{\tilde{H}^{2} \times \tilde{H}^{1}} \\
\lesssim \left\| |u|^{4} u g(|u|)  \right\|_{L_{t}^{1} L_{x}^{2}([t_1,t_2])} +
\left\| \nabla \left( |u|^{4} u g(|u|) \right) \right\|_{L_{t}^{1} L_{x}^{2}([t_1,t_2])} \\
\lesssim \| u \|^{4}_{L_{t}^{4} L_{x}^{12}([t_1,t_2])} \| u \|_{L_{t}^{\infty} L_{x}^{6}([t_1,t_2])}
+ \| u \|^{4}_{L_{t}^{4} L_{x}^{12}([t_1,t_2])} \| \nabla u \|_{L_{t}^{\infty} L_{x}^{6}([t_1,t_2])} \\
\leq \epsilon  \cdot
\end{array}
\nonumber
\end{equation}
Hence the Cauchy criterion is satisfied and there exists a limit $\mathbf{u_{\pm}} \in \tilde{H}^{2} \times \tilde{H}^{1}$ such that
$\left\| K^{-1} \mathbf{u(t)} -  \mathbf{u_{\pm}} \right\|_{\tilde{H}^{2} \times \tilde{H}^{1}} \rightarrow 0 $ as $t \rightarrow \pm \infty$.
Hence $\left\| \mathbf{u(t)} - K(t) \mathbf{u_{\pm}} \right\|_{\tilde{H}^{2} \times \tilde{H}^{1}} \rightarrow 0 $ as $t \rightarrow \pm \infty$.

\end{document}